\documentclass{amsart}

\usepackage{amsmath}
\usepackage{amssymb}
\usepackage{enumerate}
\usepackage{amsthm}
\usepackage{tikz-cd}
\usepackage{tikz}
\usepackage{tikz-3dplot}
\usetikzlibrary{arrows.meta,arrows}
\usepackage{hyperref}
\usepackage{multicol}

\calclayout

\newif\ifJordanCurveTheorem
\JordanCurveTheoremfalse
\newif\ifShowQEResults
\ShowQEResultstrue

\newcommand{\C}{\mathbb{C}}
\newcommand{\Z}{\mathbb{Z}}
\newcommand{\R}{\mathbb{R}}
\newcommand{\Q}{\mathbb{Q}}

\newcommand{\Pp}{\mathbb{P}}
\newcommand{\Aa}{\mathbb{A}}

\newcommand{\Hh}{\mathbb{H}}

\newcommand{\ACTS}{\curvearrowright}

\DeclareMathOperator{\trace}{tr}
\newcommand{\val}{\operatorname{val}}
\newcommand{\Trop}{\mathbf{Trop}}
\newcommand{\trop}{\mathrm{trop}}

\newcommand{\cell}{\mathcal{C}}
\newcommand{\ConeParam}{\phi}
\newcommand{\PreConeParam}{\Phi}
\newcommand{\ConeParamPi}{\psi}
\newcommand{\Fix}{\mathsf{Fix}}

\iffalse
\newcommand{\TODO}[1]{({\color{red}TODO} {#1})}
\else
\newcommand{\TODO}[1]{}
\fi

\numberwithin{equation}{section}

\newtheorem{lemma}{Lemma}[section]
\newtheorem{corollary}[lemma]{Corollary}
\newtheorem{theorem}[lemma]{Theorem}
\newtheorem{proposition}[lemma]{Proposition}

\newtheorem{theoremA}{Theorem}

\theoremstyle{definition}
\newtheorem{definition}[lemma]{Definition}

\newtheorem{question}[lemma]{Question}
\newtheorem{example}[lemma]{Example}
\newtheorem*{remark}{Remark}

\title{Tropical Dynamics of Markov Surfaces}
\author{Seung uk Jang}
\date{May 30, 2023. Last revised \today}

\begin{document}
\allowdisplaybreaks

\begin{abstract}
    We discuss the algebraic dynamics on Markov cubics generated by Vieta involutions, in the tropicalized setting. It turns out that there is an invariant subset of the tropicalized Markov cubic where the action by Vieta involutions can be modeled by that of $(\infty,\infty,\infty)$-triangle reflection group on the hyperbolic plane.

    This understanding of the tropicalized algebraic dynamics produces some results on Markov cubics over non-archimedean fields, including the existence of the Fatou domain and finitude of orbits with rational points having prime power denominators.
\end{abstract}

\maketitle

\section{Introduction}

By \emph{Markov surfaces}, we mean a family of surfaces
\[S_{ABCD}\colon X_1^2+X_2^2+X_3^2+X_1X_2X_3=AX_1+BX_2+CX_3+D,\]
where $A,B,C$ and $D$ are parameters that sits on a field. In this paper, we study the algebraic dynamics of this surface but in a tropical point of view.

The motivation of this surface may be found from character varieties of 4-punctured sphere or 1-punctures torus, but not limited to that. Various motivations to introduce this surfaces may be found in \cite{Goldman2006,Can09,RR21}, etc.

By \cite[Thm. 2]{ElHuti} (see also \cite[Thm. 3.1]{CL09}), the group $\mathsf{Aut}(S_{ABCD})$ of algebraic automorphisms of surfaces $S_{ABCD}$ has a finite-index subgroup $\Gamma_{ABCD}$, in which we call it the \emph{Vieta group}, generated by those maps called \emph{Vieta involutions},
\begin{align}
    s_1(X_1,X_2,X_3) &= (-X_2X_3-X_1,X_2,X_3) \nonumber \\
    &= ((X_2^2+X_3^2-BX_2-CX_3-D)/X_1,X_2,X_3)\quad(\text{if }X_1\neq 0), \label{eqn:vanilla-vieta-1}\\
    s_2(X_1,X_2,X_3) &= (X_1,-X_1X_3-X_2,X_3) \nonumber \\
    &= (X_1,(X_1^2+X_3^2-AX_1-CX_3-D)/X_2,X_3)\quad(\text{if }X_2\neq 0), \label{eqn:vanilla-vieta-2} \\
    s_3(X_1,X_2,X_3) &= (X_1,X_2,-X_1X_2-X_3) \nonumber \\
    &= (X_1,X_2,(X_1^2+X_2^2-AX_1-BX_2-D)/X_3)\quad(\text{if }X_3\neq 0). \label{eqn:vanilla-vieta-3}
\end{align}
Hence to study the algebraic dynamics of the surface $S_{ABCD}$, it suffices to study the action of the Vieta group. 

Following \cite{SV20,filip2019tropical}, we study the tropical algebraic dynamics of the surface under a specific embedding $S_{ABCD}\hookrightarrow\Aa^3$. Although there might be several versions of ``tropical viewpoints,'' in this paper we restrict ourselves to tropicalizations of varieties over non-archimedean fields.

The techniques here are directly based on the combinatorial structure of the tropicalizations and hence they are less robust by the nature. Nonetheless, they can be perhaps generalized to a family of hypersurfaces like Markov--Hurwitz varieties \cite{MR4024562}:
\[X_1^2+X_2^2+\cdots+X_n^2=AX_1\cdots X_n+K,\]
where $A,K$ are parameters.

\subsection{Main Results} We list our main results as follows. 

First is the principal geometric classification of the tropical action. Let $(K,\val)$ be a nontrivially vaued algebraically closed field. Consider Markov surfaces
\[S_{ABCD}\colon X_1^2+X_2^2+X_3^2+X_1X_2X_3=AX_1+BX_2+CX_3+D\]
with parameters $A,B,C,D\in K$, whose values are denoted by $a,b,c,d$ respectively. Let $\Gamma_{ABCD}=\langle s_1,s_2,s_3\rangle$ be the group generated by Vieta involutions. 
In the tropicalized surface of $S_{ABCD}$, we will define the \emph{skeleton} $Sk(a,b,c,d)\subset\Trop(S_{ABCD})$ (see Definition~\ref{def:skeleton}), a natural invariant subset with the \emph{tropical action} $\Gamma_{ABCD}\ACTS^{\trop} Sk(a,b,c,d)$, by:
\begin{align}
    \trop(s_1)(x_1,x_2,x_3) &= (\min\{2x_2,2x_3,b+x_2,c+x_3,d\}-x_1,x_2,x_3), \label{eqn:tropical-vieta-1} \\
    \trop(s_2)(x_1,x_2,x_3) &= (x_1,\min\{2x_1,2x_3,a+x_1,c+x_3,d\}-x_2,x_3), \label{eqn:tropical-vieta-2} \\
    \trop(s_3)(x_1,x_2,x_3) &= (x_1,x_2,\min\{2x_1,2x_2,a+x_1,b+x_2,d\}-x_3). \label{eqn:tropical-vieta-3}
\end{align}

\begin{theoremA}
    \label{thm:A}
    Let $Sk(a,b,c,d)\subset\Trop(S_{ABCD})$ and the tropical action $\Gamma_{ABCD}\ACTS^{\trop}Sk(a,b,c,d)$ be as above. Then the action is conjugate to the following models.
    \begin{enumerate}[(a)]
        \item \emph{(Theorem~\ref{lem:one-meromorphic-model})} If $\min\{a,b,c,d\}<0$, then there is an invariant dense open subset $U\subset Sk(a,b,c,d)$ such that the action $\Gamma_{ABCD}\ACTS U$ is topologically conjugate to the $(\infty,\infty,\infty)$-triangle group action on the hyperbolic plane.
        \item \emph{(Theorem~\ref{thm:tropical-dynamics-holomorphic-parameters})} If $\min\{a,b,c,d\}\geq 0$, then there is a group isomorphism $j\colon\Gamma_{ABCD}\to\ker(\mathrm{PGL}_2(\Z)\to\mathrm{PGL}_2(\Z/2\Z))$ such that $\Gamma_{ABCD}\ACTS Sk(a,b,c,d)$ is conjugate to the linear action $j(\Gamma_{ABCD})\ACTS\R^2/\{\pm 1\}$ on the plane modulo antipodes.
    \end{enumerate}
\end{theoremA}

We coin \emph{meromorphic parameters} for the case $\min\{a,b,c,d\}<0$ and \emph{holomorphic parameters} for the case $\min\{a,b,c,d\}\geq 0$. The geometric distinction between them lies on the degeneration of the region $\subset Sk(a,b,c,d)$ where the ping-pong lemma can be established (see Proposition~\ref{lem:cayley-domain-description}(v)).

In meromorphic parameters, we have an explicit triangular-shaped region in the skeleton whose images by $\Gamma$ copies so-called Farey tessellation of the hyperboic plane. In holomorphic parameters, this region degenerates to a point; in that case, the tropical action can be described with the \emph{Cayley cubic} $S_{0004}$, which is isomorphic to the algebraic torus $\mathbb{G}_m^2/\pm I_2$ quotiented by the monomial map $(-I_2).(x,y)=(x^{-1},y^{-1})$, and the Vieta action are given by monomial maps too. As this variety tropicalizes to the plane quotient $\R^2/\{\pm 1\}$, the model follows.

The other results are applications of Theorem \ref{thm:A} and they will be stated after sketching their backgrounds.

\subsubsection{Fatou Domains} Recall that Rebelo and Roeder defined the \emph{Fatou domain} of the action $\Gamma_{ABCD}\ACTS S_{ABCD}(\C)$ as follows (see \S{1.5} of \cite{RR21}):
\[\mathcal{F}_{ABCD} := \left\{p\in S_{ABCD}(\C) : \begin{array}{l} \Gamma_{ABCD}\text{ forms a normal family} \\ \text{in an open neighborhood of }p \end{array}\right\}.\]
Here, we say that $\Gamma_{ABCD}$ forms a \emph{normal family} in an open set $U$ if for any sequence $\{\gamma_n\}_{n=0}^\infty$ in $\Gamma_{ABCD}$, there exists a subsequence $\{\gamma_{n_k}\}$ such that the following holds. The restrictions $\gamma_{n_k}|_U$ are either (a) locally uniformly convergent to a continuous map on $U$, or (b) locally uniformly converging to infinity, i.e., for any compact subsets $\overline{V}\Subset U$ and $K\Subset S_{ABCD}(\C)$, we have $\gamma_{n_k}(\overline{V})\cap K=\varnothing$ for all but finitely many $k$'s.

Rebelo and Roeder have shown, in their theorems D and E of \cite{RR21}, that there is a parameter $(A,B,C,D)=(0,0,0,4)$ that leaves $\mathcal{F}_{ABCD}=\varnothing$, yet there is an open set of parameters that lets $\mathcal{F}_{ABCD}\neq\varnothing$. We say a parameter $(A,B,C,D)$ \emph{admits the Fatou domain} if $\mathcal{F}_{ABCD}\neq\varnothing$.

Over non-archimedean fields, we can analogously extend the definitions above and define the Fatou domain of the Markov surface. In fact, it is easier to discover parameters that admit the Fatou domain: namely, all meromorphic parameters admit it.
\begin{theoremA}[Fatou Sets over non-archimedean fields; see Theorem \ref{thm:p-adic-fatou-domain}]
    \label{thm:B}
   Suppose $S_{ABCD}$ is defined over a non-archimedean field $(K,\val)$. If $\min\{\val(A),\val(B),\val(C),\val(D)\}<0$, then $(A,B,C,D)$ admits the Fatou domain.
\end{theoremA}

We remark that the parameter $(0,0,0,4)$ still does not admit Fatou domain in this case; and indeed, the values of 0 and 4 are nonnegative on any non-archimedean field.

\subsubsection{Diophantine Questions} One of the pioneering interest of Markov surface \cite{Markoff1879,Markoff1880} comes from the diophantine equation
\[x^2+y^2+z^2=3xyz,\]
whose solution set over $\C$ (or any field of characteristic $\neq 3$) is an affine transform of the surface $S_{0000}(\C)$. It was proven in \cite[\S{10}]{Markoff1880} that all positive integer solution of the equation is found on the $\Gamma_{0000}$-orbit of $(1,1,1)$, by the method of infinite descent.

Nowadays such an infinite descent may be explained by the fundamental domain of the mapping class group action on the Teichm\"uller space of 1-cusped torus (see \cite[\S{3.1}]{Goldman03}: the action is isomorphic to the $(\infty,\infty,\infty)$-triangle group action on the hypoerbolic plane). But this is just a viewpoint over real points; works like \cite{GMS21,GS22} pays attention to $S$-integer points, where $S$ is a finite set of primes.

Our tropical point of view allows us to state a result about $\{p\}$-integer points.
\begin{theoremA}[{$\Z[\frac1p]$-points in the Compact Component; see Theorem \ref{thm:rational-points-compact-component}}]
    \label{thm:C}
    Suppose $p$ is an odd prime and $D$ is a parameter in $\Z[\frac1p]$ such that $0<D<1$. There are finitely many points $P_1,\ldots,P_N\in S_{000D}(\Z[\frac1p])$ such that, every other $\Z[\frac1p]$-point in the compact part of $S_{000D}(\R)$ is either
    \begin{enumerate}[(i)]
        \item on the $\Gamma_{000D}$-orbit of some $P_i$, or
        \item on the $\Gamma_{000D}$-orbit of some point in the set
        \[\left\{(X_1,X_2,X_3)\in S_{000D}(\Z[\tfrac1p]) : \begin{array}{l} X_1X_2X_3=0, \\ (\exists i)(2\val_p(X_i)<\val_p(D)) \end{array}\right\}.\]
    \end{enumerate}
    Moreover, if $p\equiv 1\pmod{4}$ and there are $x_0,y_0\in\Z[\frac1p]$ with $x_0^2+y_0^2=D$, then the set in (ii) is infinite. Otherwise, the set is empty.
\end{theoremA}

The core analysis of the above is to find a relation between the absolute value and the $p$-adic norm, for numbers in $\Z[\frac1p]$. By such, we can find estimates of the $p$-adic absolute values by archimedean absolute values.

One of the key point of the above theorem is that, the points with zero on one of its coordinates, described in case (ii), are mutually inequivalent and hence yields infinitely many $\Gamma$-orbits. We can avoid such a case when $p\equiv 3\pmod{4}$ and hence the compart part of $S_{000D}(\Z[\frac1p])$ has only finitely many orbits.

\subsection{Outline of the Proof, by an Example} Instead of outlining the proof of theorems by the most general language, we work on an illustrative example and list the key ingredients and notions for the proof.

The example comes from the \emph{punctured torus parameter}, meaning that our parameters are $(A,B,C,D)=(0,0,0,D)$. The case when $\val(D)<0$ is the most illustrative example of the results in this paper.

The tropicalization $\Trop(S_{000D})$ of $S_{000D}$ is the set of points $(x_1,x_2,x_3)\in\R^3$ in which the piecewise-linear function (here, $d:=\val(D)$)
\[\min\{2x_1,2x_2,2x_3,x_1+x_2+x_3,d\}\]
attains its minimum on at least two terms. If we plot such points on the space, we get a picture as in Figures \ref{fig:first-tropicalization} and \ref{fig:second-tropicalization}, depending on $d<0$ or $d\geq 0$ respectively. If $d<0$, we observe that the triangle (to be called the \emph{ping-pong table}) in the figure has vertices $(\frac12d,\frac12d,0)$, $(\frac12d,0,\frac12d)$, and $(0,\frac12d,\frac12d)$. If $d\geq 0$, this triangle degenerates to the point $(0,0,0)$.

\begin{figure}
    \centering
    \begin{tikzpicture}[scale=.7,
        x={({(13/16)^(1/2)},0)}, 
        y={(0,{(15/16)^(1/2)})}, 
        z={({-.5*cos(30)},{-.5*sin(30)})}]
        
            \draw[dashed,-latex] (-5,0,0) -- (5,0,0) node[right] {$x_1$};
            \draw[dashed,-latex] (0,-5,0) -- (0,5,0) node[above] {$x_2$};
            \draw[dashed,-latex] (0,0,-5) -- (0,0,5) node[left] {$x_3$};
            \foreach \i in {.5}{
            \fill[gray,opacity=\i] (-5,-5,5) -- (-1,-1,5) -- (-1,-1,0) -- (-5,-5,0) -- cycle;
            \fill[gray,opacity=\i] (5,-5,-5) -- (5,-1,-1) -- (0,-1,-1) -- (0,-5,-5) -- cycle;
            \fill[gray,opacity=\i] (-5,5,-5) -- (-1,5,-1) -- (-1,0,-1) -- (-5,0,-5) -- cycle;
            }
            \foreach \i in {.5}{
            \fill[gray,opacity=\i] (-1,-1,0) -- (-1,0,-1) -- (0,-1,-1) -- cycle;
            \fill[gray,opacity=\i] (-1,-1,0) -- (-5,-5,0) -- (-5,0,-5) -- (-1,0,-1) -- cycle;
            \fill[gray,opacity=\i] (-1,-1,0) -- (-5,-5,0) -- (0,-5,-5) -- (0,-1,-1) -- cycle;
            \fill[gray,opacity={.5*\i}] (0,-1,-1) -- (0,-5,-5) -- (-5,0,-5) -- (-1,0,-1) -- cycle;
            }
            \foreach \i in {.5}{
            \fill[gray,opacity=\i] (-1,-1,0) -- (-1,0,-1) -- (-1,5,-1) -- (-1,5,5) -- (-1,-1,5) -- cycle;
            \fill[gray,opacity=\i] (-1,0,-1) -- (0,-1,-1) -- (5,-1,-1) -- (5,5,-1) -- (-1,5,-1) -- cycle;
            \fill[gray,opacity=\i] (-1,-1,0) -- (0,-1,-1) -- (5,-1,-1) -- (5,-1,5) -- (-1,-1,5) -- cycle;
            }
            \draw (-1,-1,0) -- (-1,0,-1) -- (0,-1,-1) -- cycle;
            \draw (-5,-5,0) -- (-1,-1,0) -- (-1,-1,5);
            \draw (-5,0,-5) -- (-1,0,-1) -- (-1,5,-1);
            \draw (0,-5,-5) -- (0,-1,-1) -- (5,-1,-1);
        \end{tikzpicture}
    \caption{Tropicalization of $S_{000D}$, $\val(D)<0$} \label{fig:first-tropicalization}

    \begin{tikzpicture}[scale=.7,
        x={({(13/16)^(1/2)},0)}, 
        y={(0,{(15/16)^(1/2)})}, 
        z={({-.5*cos(30)},{-.5*sin(30)})}]
        
            \draw[dashed,-latex] (-5,0,0) -- (5,0,0) node[right] {$x_1$};
            \draw[dashed,-latex] (0,-5,0) -- (0,5,0) node[above] {$x_2$};
            \draw[dashed,-latex] (0,0,-5) -- (0,0,5) node[left] {$x_3$};
            \foreach \i in {.5}{
            \fill[gray,opacity=\i] (-5,-5,5) -- (1,1,5) -- (1,1,1) -- (0,0,0) -- (-5,-5,0) -- cycle;
            \fill[gray,opacity=\i] (5,-5,-5) -- (5,1,1) -- (1,1,1) -- (0,0,0) -- (0,-5,-5) -- cycle;
            \fill[gray,opacity=\i] (-5,5,-5) -- (1,5,1) -- (1,1,1) -- (0,0,0) -- (-5,0,-5) -- cycle;
            }
            \foreach \i in {.5}{
            \fill[gray,opacity=\i] (0,0,0) -- (-5,-5,0) -- (-5,0,-5) -- cycle;
            \fill[gray,opacity=\i] (0,0,0) -- (-5,-5,0) -- (0,-5,-5) -- cycle;
            \fill[gray,opacity={.5*\i}] (0,0,0) -- (0,-5,-5) -- (-5,0,-5) -- cycle;
            }
            \foreach \i in {.5}{
            \fill[gray,opacity=\i] (1,1,1) -- (1,5,1) -- (1,5,5) -- (1,1,5) -- cycle;
            \fill[gray,opacity=\i] (1,1,1) -- (5,1,1) -- (5,5,1) -- (1,5,1) -- cycle;
            \fill[gray,opacity=\i] (1,1,1) -- (5,1,1) -- (5,1,5) -- (1,1,5) -- cycle;
            }
            \draw (-5,-5,0) -- (0,0,0) -- (1,1,1) -- (1,1,5);
            \draw (-5,0,-5) -- (0,0,0) -- (1,1,1) -- (1,5,1);
            \draw (0,-5,-5) -- (0,0,0) -- (1,1,1) -- (5,1,1);
        \end{tikzpicture}
    \caption{Tropicalization of $S_{000D}$, $\val(D)\geq0$} \label{fig:second-tropicalization}
\end{figure}

Tropicalizations of maps $s_i$, $i=1,2,3$, in this case are defined as follows:
\begin{align*}
    \trop(s_1)(x_1,x_2,x_3) &= (\min\{2x_2,2x_3,d\}-x_1,x_2,x_3), \\
    \trop(s_2)(x_1,x_2,x_3) &= (x_1,\min\{2x_1,2x_3,d\}-x_2,x_3), \\
    \trop(s_3)(x_1,x_2,x_3) &= (x_1,x_2,\min\{2x_1,2x_2,d\}-x_3).
\end{align*}
It turns out that the subset $Sk=Sk(\infty,\infty,\infty,d)$ of $\Trop(S_{000D})$ characterized by that ``the degree 3 term $x_1+x_2+x_3$ is majorized,'' is invariant under all three tropicalized maps. See Figure \ref{fig:skeleton-and-exception-lines} for a sketch of that subset $Sk$ with $d<0$; the one with $d\geq 0$ will appear later in Figure~\ref{fig:tetrahedral-cone}. Such a subset will be called the \emph{skeleton} of $\Trop(S_{000D})$. If $d<0$, the skeleton consists of the triangle with vertices $(\frac12d,\frac12d,0)$, $(\frac12d,0,\frac12d)$, and $(0,\frac12d,\frac12d)$, mentioned above, and three infinite components that lie on planes $x_1=x_2+x_3$, $x_2=x_1+x_3$, and $x_3=x_1+x_2$, respectively. Such subsets will be called \emph{adjacency loci} later. Furthermore, the whole skeleton lies on the $(-,-,-)$-octant in the space.

The shape of the skeleton may be described as a trimmed tetrahedral cone (if $d<0$; if $d\geq 0$ then we do not trim), which also looks like a ``piecewise-linear hyperboloid.'' The plane $\Pi=\{x_1+x_2+x_3=0\}$ can be seen as the `base' of that hyperboloid, and in fact the orthogonal projection $Sk\to\Pi$ is a homeomorphism. The triangle and infinite components mentioned above form an equilateral tripod picture on $\Pi$; see Figure \ref{fig:projection-tripod}.

\begin{figure}
    \centering
    \begin{tikzpicture}
        \foreach \j in {.1}{
            \foreach \i/\k in {1/left,2/below,3/right}{
                \begin{scope}[rotate={120*(\i-1)}]
            \draw[dashed] ({3*cos(210)+\j*cos(120)},{3*sin(210)+\j*sin(120)}) -- ({cos(210)+\j*cos(120)},{sin(210)+\j*sin(120)}) -- (-\j,1) -- (-\j,3) node[\k] {$\Fix(\trop(s_{\i}))$};
                \end{scope}
            }
        }
        \draw (90:1) -- (210:1) -- (-30:1) -- cycle;
        \foreach \i in {90,210,-30}{
        \draw (\i:1) -- (\i:3);
        }
    \end{tikzpicture}
    \caption{Projection of $Sk(\infty,\infty,\infty,d)$ ($d<0$) to the Plane $\Pi$, with Cells and Fixed sets sketched} \label{fig:projection-tripod}
\end{figure}
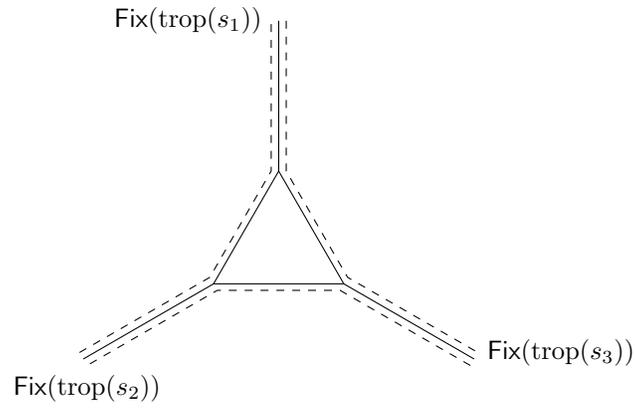

The action of $\trop(s_i)$'s on the plane $\Pi$ may be seen as a reflection along its fixed set (see Figure \ref{fig:projection-tripod} for sketches of fixed sets $\Fix(\trop(s_i))$). If we apply $\trop(s_i)$'s to the `tripod lines,' we see that the tripod reflects and develops to a new picture as in Figure \ref{fig:projection-tripod-lv1}. Freeze that figure, and apply reflections once again. This gives Figure \ref{fig:projection-tripod-lv2}. If we continue do this, we get a family of triangles that reminds us the $(\infty,\infty,\infty)$-triangle group action on $\Hh^2$, as sketched in Figure \ref{fig:inf-inf-inf-triangle}.

\begin{figure}
    \begin{minipage}{.5\linewidth}
        \centering
        \begin{tikzpicture}
    \draw (90:1) -- (210:1) -- (-30:1) -- cycle;
    \foreach \i in {90,210,-30}{
    \draw (\i:1) -- (\i:3);
    }
    \draw (30:1) -- (90:1) -- (150:1) -- (210:1) -- (270:1) -- (-30:1) -- cycle;
    \foreach \i in {30,150,270}{
    \draw (\i:1) -- (\i:3);
    }
        \end{tikzpicture}
        \caption{Reflections of the Tripod} \label{fig:projection-tripod-lv1}
      \end{minipage}%
      \begin{minipage}{.5\linewidth}
        \begin{tikzpicture}
    \draw (90:1) -- (210:1) -- (-30:1) -- cycle;
    \foreach \i in {90,210,-30}{
    \draw (\i:1) -- (\i:3);
    }
    \draw (30:1) -- (90:1) -- (150:1) -- (210:1) -- (270:1) -- (-30:1) -- cycle;
    \foreach \i in {30,150,270}{
    \draw (\i:1) -- (\i:3);
    }
        \foreach \i in {0,1,2,3,4,5}{
            \draw ({60*\i}:{3^(1/2)}) -- ({60*\i+30}:1) -- ({60*\i+60}:{3^(1/2)});
            \draw ({60*\i}:{3^(1/2)}) -- ({60*\i}:{3^(3/2)/2});
            }
        \end{tikzpicture}
        \caption{More Reflections of the Tripod} \label{fig:projection-tripod-lv2}
      \end{minipage}
\end{figure}

\begin{figure}
    \centering
    \begin{tikzpicture}[scale=2]
            \draw (0,0) circle (1);
            \draw (-30:1) arc (-120:-180:{tan(60)}) arc (0:-60:{tan(60)}) arc (120:60:{tan(60)});
            \draw (-30:1) arc (-120:-240:{tan(30)}) arc (-60:-180:{tan(30)}) arc (0:-120:{tan(30)}) arc (60:-60:{tan(30)}) arc (120:0:{tan(30)}) arc (180:60:{tan(30)});
            \draw (-30:1) arc (-120:-270:{tan(15)}) arc (-90:-240:{tan(15)}) arc (-60:-210:{tan(15)}) arc (-30:-180:{tan(15)}) arc (0:-150:{tan(15)}) arc (30:-120:{tan(15)}) arc (60:-90:{tan(15)}) arc (90:-60:{tan(15)}) arc (120:-30:{tan(15)}) arc (150:0:{tan(15)}) arc (180:30:{tan(15)}) arc (210:60:{tan(15)});
    \end{tikzpicture}
    \caption{The $(\infty,\infty,\infty)$-triangle Tessellation of $\Hh^2$} \label{fig:inf-inf-inf-triangle}
\end{figure}
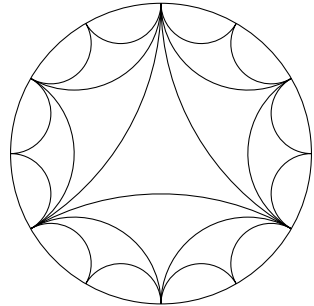

Now we have observed that the triangles in Figure \ref{fig:projection-tripod-lv2} correspond to ideal triangles in the hyperbolic plane. If we exclude the rays emitting from triangles, we get a subset of the skeleton corresponding to $\Hh^2$ \emph{equivariantly}. It turns out that it fills up `rational rays' in $\R^3$: one can even describe those rays as $\{(tq,tp,t(p+q)) : t\leq\frac12d\}$ or its images under a coordinate permutation, where $(p,q)$ runs through coprime nonnegative integers. See Figure \ref{fig:skeleton-and-exception-lines} to see the spatial configuration of these rays.

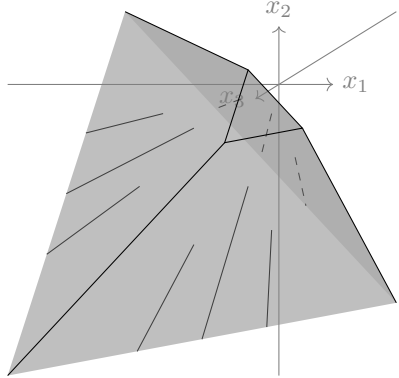
\begin{figure}
    \centering
    \begin{tikzpicture}[scale=.8,
        x={({(13/16)^(1/2)},0)}, 
        y={(0,{(15/16)^(1/2)})}, 
        z={({-.5*cos(30)},{-.5*sin(30)})}]
        
            \draw[gray,->] (-5,0,0) -- (1,0,0) node[right] {$x_1$};
            \draw[gray,->] (0,-5,0) -- (0,1,0) node[above] {$x_2$};
            \draw[gray,->] (0,0,-5) -- (0,0,1) node[left] {$x_3$};
            \draw (-1,-2,-1) -- (-2.5,-5,-2.5);
            \draw (-2,-1,-1) -- (-5,-2.5,-2.5);
            \draw[dashed] (-1,-1,-2) -- (-2.5,-2.5,-5);
            
            \draw[dashed] (-1,-2,-3) -- ({-5/3},{-10/3},-5);
            \draw[dashed] (-2,-1,-3) -- ({-10/3},{-5/3},-5);
            \draw (-1,-3,-2) -- ({-5/3},-5,{-10/3});
            \draw (-2,-3,-1) -- ({-10/3},-5,{-5/3});
            \draw (-3,-1,-2) -- (-5,{-5/3},{-10/3});
            \draw (-3,-2,-1) -- (-5,{-10/3},{-5/3});
            \foreach \i in {.5}{
            \fill[gray,opacity=\i] (-1,-1,0) -- (-1,0,-1) -- (0,-1,-1) -- cycle;
            \fill[gray,opacity=\i] (-1,-1,0) -- (-5,-5,0) -- (-5,0,-5) -- (-1,0,-1) -- cycle;
            \fill[gray,opacity=\i] (-1,-1,0) -- (-5,-5,0) -- (0,-5,-5) -- (0,-1,-1) -- cycle;
            \fill[gray,opacity={.5*\i}] (0,-1,-1) -- (0,-5,-5) -- (-5,0,-5) -- (-1,0,-1) -- cycle;
            }
            \draw (-1,-1,0) -- (-1,0,-1) -- (0,-1,-1) -- cycle;
            \draw (-5,-5,0) -- (-1,-1,0) ;
            \draw (-5,0,-5) -- (-1,0,-1) ;
            \draw (0,-5,-5) -- (0,-1,-1) ;
        \end{tikzpicture}
    \caption{Skeleton $Sk(\infty,\infty,\infty,d)$ ($d<0$) and Exceptional Lines} \label{fig:skeleton-and-exception-lines}
\end{figure}

We list what we have observed from the example above:
\begin{multicols}{2}
    \begin{itemize}
        \item Tropicalization $\Trop(S_{000D})$
        \item Skeleton $Sk(\infty,\infty,\infty,d)$
        \item Adjacency loci and Fixed sets of $\trop(s_i)$
        \item The `central triangle,' a.k.a. the Ping-pong Table, and its orbit
        \item Correspondence with the $(\infty,\infty,\infty)$-triangle tessellation of $\Hh^2$
        \item Exceptional rays
    \end{itemize}
\end{multicols}
Theorem \ref{thm:A} is then established by reproducing these observations for general $Sk(a,b,c,d)$ with $\min\{a,b,c,d\}<0$. The case $\min\{a,b,c,d\}\geq 0$ is treated on the basis of its `uniform dynamical model,' together with the Cayley cubic $S_{0004}$ that has a special identification with a pair of nonempty numbers.

Theorem \ref{thm:B} uses the fact that any point in the interior of the ping-pong table converges to infinity as we arbitrarily continue the reflections $\trop(s_i)$'s. This is a normality that we can observe in the tropicalized picture. One can utilize this to infer normality in any algebraically closed non-archimedean fields.

Theorem \ref{thm:C}, on contrary, focuses not only the ping-pong table but also the exceptional rays. The cases (i) and (ii) stated in the theorem in fact corresponds to points on the ping-pong table and exceptional rays (especially those emitting from the ping-pong table), respectively. Finitude of case (i) follows from that ping-pong tables are always bounded, so together with discrete valuations, we have our finitude. The statement is designed to lift this tropical behavior of orbits to that of rational points with prime power denominators.

\subsection{Structure of the Paper}

In section~\ref{sec:tropicalization}, we discuss the basics of tropicalizing affine varieties, and discuss the tropicalization of Markov cubics $S_{ABCD}$, as a subset of $\R^3$. Our approach do not rely on Newton polygons and we rather explicitly view a tropicalization as a union of boundaries of a convex region, called \emph{cells}. 
In section~\ref{sec:tropical-vieta-involutions}, we define tropicalized Vieta involutions, starting from a quadratic equation. We discuss how these tropicalized involutions interact with the algebraic involutions, and from that, we introduce the notion of \emph{skeleton}, the key invariant subset of our interest. 

In section~\ref{sec:holomorphic}, we focus on the holomorphic parameter case (all paramters having nonnegative values), which is the case when the tropical action on the skeleton can be completely linearlized. The analysis on the Cayley cubic surface $S_{0004}$ plays a key role here. Consequently, we prove one half of Theorem~\ref{thm:A} in the section.

In sections~\ref{sec:skeleton} to \ref{sec:quad-cell}, we establish detailed understandings of skeletons, aimed to be applicable for any parameters. Section~\ref{sec:skeleton} focuses on the topology of the skeleton and fixed loci of tropicalized Vieta involutions: a skeleton is homeomorphic to a plane. Section~\ref{sec:alignment-fixed-loci} focuses on how fixed loci of tropicalized involutions align on the plane. Section~\ref{sec:quad-cell} discusses on reducing skeletal points near the infinity into a compact region or a point on a(n exceptional) ray.

Section~\ref{sec:pp-theory} is a section that is aimed to develop some topological theory to compare our Vieta system $\Gamma_{ABCD}$ acting on the skeleton with the hyperbolic $(\infty,\infty,\infty)$-triangle action on the hyperbolic plane. The section is aimed to answer the following question: if two spaces $X,Y$ admit proper $\Gamma=(\Z/2\Z)^{\ast 3}$-actions, when can we say $X$ and $Y$ are $\Gamma$-equivariantly homeomorphic? Once this is established, section~\ref{sec:meromorphic-case} finishes the proof of Theorem~\ref{thm:A}, in the case of meromorphic parameters (one of parameters having negative value).

The rest are applications of the theorem established. Sections~\ref{sec:nA-Fatou} and \ref{sec:rational-points-compact-component} are for Theorems~\ref{thm:B} and \ref{thm:C}, respectively. In each case, some works are requied to bridge the tropical dynamics to those applicational interest, and majorities of the sections are on that.

\subsection*{Acknowledgements}
I would like to thank Roland Roeder, who introduced me these interesting systems, and Simion Filip, who suggested me to illuminate this system in a tropical view. I am also greatful to gain comments, discussions, and suggestions from Juhun Baik, Beno\^it Bertrand, Serge Cantat, Alex Eskin, Amit Ghosh, Joe Silverman, Christopher-Lloyd Simon, Peter Whang, and Shengyuan Zhao regarding this work.

This material is based upon work supported by the National Science Foundation under Grant No. DMS-2005470 and DMS-2305394. The research activities of the author are partially funded by the European Research Council (ERC GOAT 101053021).

\section{Tropicalization of Markov Surfaces}
\label{sec:tropicalization}

In this section, we recall the tropicalization of an affine surface defined over an algebraically closed non-archimedean field, and describe what is the tropicalization of Markov surfaces.

\subsection{Non-archimedean Fields}

We say a field $K$ a \emph{valued field} if there is a \emph{valuation map} $\val\colon K\to\R\cup\{\infty\}$ such that, whenever $x,y\in K$,
\begin{enumerate}[(i)]
    \item $\val(0)=\infty$, and $\val(x)<\infty$ whenever $x\neq 0$;
    \item $\val(xy)=\val(x)+\val(y)$; 
    \item $\val(x+y)\geq\min\{\val(x),\val(y)\}$; and
    \item $\val(x+y)=\min\{\val(x),\val(y)\}$ if $\val(x)\neq\val(y)$.
\end{enumerate}
We conventionally understand that $\infty+x=\infty$ and $\min\{\infty,x\}=x$ for all $x\in\R\cup\{\infty\}$. Note that $\val(1)=\val(1\cdot 1)=2\val(1)$ shows that $\val(1)=0$, and $2\val(-1)=\val((-1)^2)=\val(1)=0$ shows not only $\val(-1)=0$ but also $\val(-x)=\val(x)$.

A valuation is \emph{nontrivial} if there is $x\in K^\times$ with $\val(x)\neq 0$. Unlike usual conventions, we say a nontrivially valued field $(K,\val)$ a \emph{non-archimedean field} (that is, we are omitting completeness here). Examples include the $p$-adic number field $(\Q_p,\val_p)$ and the field of Laurent series $(\C((t)),\val_{t=0})$.

By Hensel's lemma (see, e.g., \cite[\S{II.4}, (4.6)]{AlgNumThry}), one can show that zeroes of an irreducible polynomial over a non-archimedean field $(K,\val)$ have the same values. Hence in any algebraic extension of $(K,\val)$ we can uniquely extend the valuation. By such, we understand that algebraic closure of $p$-adic numbers $(\overline{\Q_p},\val_p)$, or the field of Puiseux series $(\bigcup_{n=1}^\infty\C((t^{1/n})),\val_{t=0})$ are example of algebraically closed non-archimedean fields (which are not complete). For algebraic closedness of Puiseux series, see \cite[\S{IV.2} Proposition 8]{Serre}\cite[Theorem 2.1.5]{MS15}.

\subsection{Tropicalization of an Affine Hypersurface}

Despite various contexts that speak tropicalizations, we specifically consider the one for affine hypersurfaces defined over a non-archimedean field $(K,\val)$.

Let $F(X_1,\ldots,X_d)\in K[X_1^{\pm 1},\ldots,X_d^{\pm 1}]$ be a Laurent polynomial which is not a monomial. Consider the hypersurface $V\subset\mathbb{G}_m^d$ in the algebraic torus defined by the equation $F(X_1,\ldots,X_d)=0$. The \emph{tropicalization} of $V$ is then defined as the closure of values of $K$-points of $V$:
\begin{align*}
    \Trop(V) &:= \overline{\val(V(K))} \\
    &= \overline{\left\{(x_1,\ldots,x_d)\in\R^d : \begin{array}{l}
        \text{there is }(X_1,\ldots,X_d)\in V(K) \\
        \text{in which }\val(X_i)=x_i\ \forall i
    \end{array}\right\}}.
\end{align*}

There are few instances in which this definition can compute the tropicalization of a hypersurface (e.g., $V=(X_1+X_2-1=0)$), but in general we compute it via Kaparnov's theorem \cite[Theorem 2.1.1]{EKL06}\cite[Theorem 3.1.3]{MS15}.

To explain what the theorem is, define the \emph{tropicalized polynomial} $\trop(F)$ of the defining polynomial $F(X)$ as follows. Write $F(X)=\sum c_\alpha X^\alpha$, where $\alpha=(\alpha_1,\ldots,\alpha_d)\in\Z^d$ runs through multiindices, $X^\alpha=X_1^{\alpha_1}\cdots X_d^{\alpha_d}$ is the corresponding monomial, and $c_\alpha\in K$ is a coefficient. Then $\trop(F)$ is
\begin{align*}
    \trop(F)(x_1,\ldots,x_d) &= \min_\alpha\left(\val(c_\alpha)+\alpha\cdot(x_1,\ldots,x_d)\right) \\
    &= \min_\alpha\left(\val(c_\alpha)+\sum_{i=1}^d\alpha_ix_i\right).
\end{align*}
Here, as all but finitely many $\alpha$'s have $\val(c_\alpha)=\infty$, such $\alpha$'s does not contribute to the minimum above. By a \emph{tropical monomial} we mean the terms $\val(c_\alpha)+\alpha\cdot x$ (where $x=(x_1,\ldots,x_d)$) in the minimum. Morally, it is expressing the value of $c_\alpha X^\alpha$, viewing that $x_i=\val(X_i)$ for $1\leq i\leq d$.

\begin{theorem}[Kapranov]
    With $V=(F=0)$ and $\trop(F)$ above, we have
    \begin{align*}
    \Trop(V) 
    &= \left\{ x\in\R^d : (\exists \alpha\neq\beta)(\val(c_\alpha)+\alpha\cdot x = \val(c_{\beta})+\beta\cdot x=\trop(F)(x))\right\} \\
    &= \left\{ x\in\R^d : \begin{array}{l}\text{two tropical monomials of $\trop(F)$, at $x$,} \\ \text{evaluate to the same value }=\trop(F)(x) \end{array}\right\}.
    \end{align*}
\end{theorem}

Here is a brief explanation on why we believe this to be true. First, observe the following. 
If $X\in V(K)$ is a point, then at least two monomials of $F(X)$ have the same value which is least among values of monomials of $F(X)$. (Suppose otherwise; then by that $\val(x+y)=\val(x)$ if $\val(x)<\val(y)$, one sees that $\val(F(X))$ equals to the minimum value of monomials of $F(X)$, which cannot be infinite since $X\in\mathbb{G}_m^d(K)=(K^\times)^d$. This contradicts to $F(X)=0$.) 
If we translate this to a ``tropical'' fashion, if $x\in\R^d$ is such that there is a unique tropical monomial of $\trop(F)(x)$ that attains the minimum, then $x$ cannot be the value of any point $X\in V(K)$. The aforementioned Kapranov's theorem states that the converse is true.

Observe that $\trop(F)$ defines a continuous piecewise-linear function $\R^d\to\R$, and on an open subset, $\trop(F)$ locally equals to a tropical monomial $\val(c_\alpha)+\alpha\cdot x$. The complement of that open subset is precisely where $\trop(F)$ is not differentiable, or the minimum is attained by more than two tropical monomials. Hence we have the following identification as well:
\begin{align*}
    \Trop(V) 
    &= \left\{ x\in\R^d : \trop(F)\text{ is not differentiable at }x\right\}.
\end{align*}

\subsubsection{Cells}

The complement of $\Trop(V)$, which is precisely where $\trop(F)$ is differentiable, has several connected components which are characterized by the tropical monomial to which $\trop(F)$ equals. So for each monomial $c_\alpha X^\alpha$ of $F(X)$, it is natural to define the \emph{cell} of $c_\alpha X^\alpha$ by $\trop(F)$ as an open set
\begin{align}
\label{eqn:cell-by-trop-polynomial}
\cell(c_\alpha X^\alpha)&:=\left\{x\in\R^d : \val(c_\alpha)+\alpha\cdot x=\trop(F)(x)\right\}^\circ \\
&= \left\{x\in\R^d : \val(c_\alpha)+\alpha\cdot x<\trop(F-c_\alpha X^\alpha)(x)\right\}, \nonumber
\end{align}
where $A^\circ$ denotes the interior of $A$.

The union of all cells equals to the complement $\R^d\setminus\Trop(V)$. The tropicalization $\Trop(V)$ equals to the union of boundaries of cells. We say two cells $\cell(c_\alpha X^\alpha)$ and $\cell(c_\beta X^\beta)$ are \emph{adjacent} if the intersection of their boundaries $\partial\cell(c_\alpha X^\alpha)\cap\partial\cell(c_\beta X^\beta)$ contains an open subset of $\Trop(V)$.
 
We introduce a shorthand for the intersection
\begin{equation}
\label{eqn:adjacency-locus}
\partial\cell(c_\alpha X^\alpha\cap c_\beta X^\beta):=\partial\cell(c_\alpha X^\alpha)\cap\partial\cell(c_\beta X^\beta),
\end{equation}
and call it the \emph{adjacency locus} if the cells $\cell(c_\alpha X^\alpha)$ and $\cell(c_\beta X^\beta)$ are adjacent. We see that this locus is a convex subset of the hyperplance $(\alpha-\beta)\cdot x=\val(c_\beta/c_\alpha)$. If there is an embedding of an open subset $\subset\R^{d-1}$ into $\partial\cell(c_\alpha X^\alpha\cap c_\beta X^\beta)$, then it shows that the cells $\cell(c_\alpha X^\alpha)$ and $\cell(c_\beta X^\beta)$ are adjacent.

Our approach to understand the tropicalization of the Markov surface will be based on specifying which cells are nonempty and listing adjacent cells of each cell. This may be away from classical approaches related to Newton polygons or toric varieties (see, e.g., \cite{KKE21}), but we choose this approach to explicitly study internal coordinates on the tropicalization and dynamical properties based on that coordinates.

\subsection{Tropicalization of Markov Surfaces}

We know that
\[F_{ABCD}(X_1,X_2,X_3)=X_1^2+X_2^2+X_3^2+X_1X_2X_3-AX_1-BX_2-CX_3-D\]
is the defining equation of the Markov surface $S_{ABCD}$, i.e., $S_{ABCD}\colon F_{ABCD}=0$. If parameters $A$, $B$, $C$, and $D$ lie on a (fixed) non-archimedean field $(K,\val)$, then the tropicalized polynomial of $F_{ABCD}$ is
\begin{align}
    f_{abcd}(x_1,x_2,x_3) &= \min\left\{\begin{array}{l}2x_1,2x_2,2x_3,x_1+x_2+x_3, \\ a+x_1,b+x_2,c+x_3,d \end{array}\right\}, \label{eqn:trop-polynomial}
\end{align}
where $a=\val(A)$, $b=\val(B)$, $c=\val(C)$, and $d=\val(D)$ are valuations of coefficients. Then $\Trop(S_{ABCD})$ is precisely the locus where $f_{abcd}$ is not differentiable. In particular, this locus only depends on the values of parameters $(a,b,c,d)$.

Examples of sketches of $\Trop(S_{ABCD})$ for (i) $a=b=c=\infty$ and $d>0$ can be found in Figure \ref{fig:second-tropicalization}, (ii) $a=b=c=\infty$ and $d<0$ can be found in Figure \ref{fig:first-tropicalization}, (iii) $(a,b,c,d)=(\infty,\infty,-1.5,-2)$ can be found in Figure \ref{fig:trop-markov-1}, and (iv) $(a,b,c,d)=(-1.3,\infty,-1.5,-2.3)$ can be found in Figure \ref{fig:trop-markov-2}.

\begin{figure}
    \centering
    \includegraphics[width=.8\textwidth]{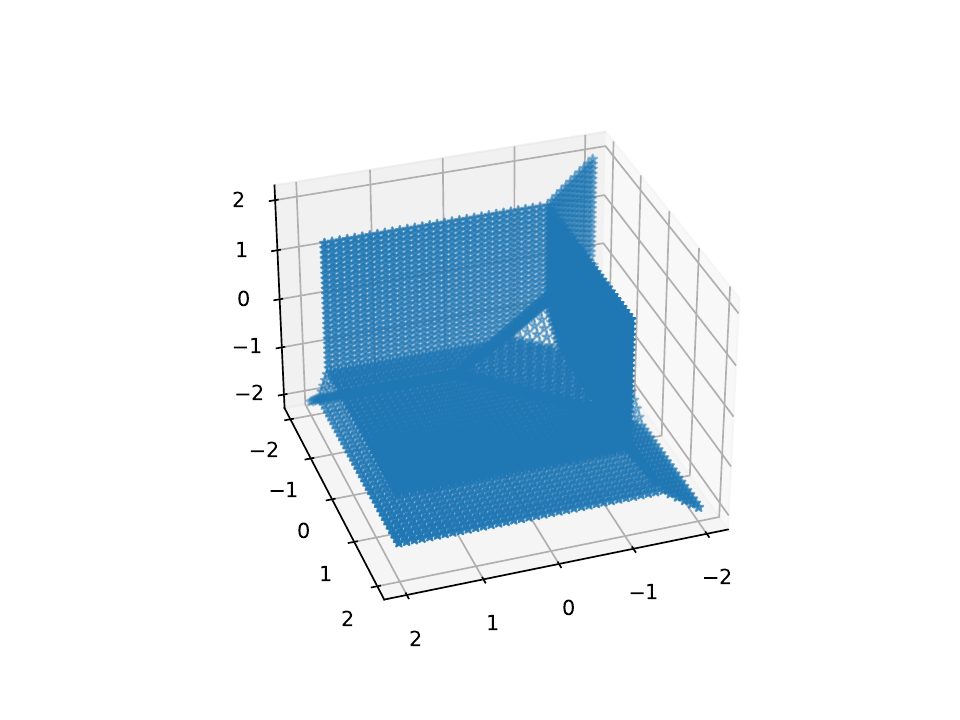}
    \caption{$\Trop(S_{ABCD})$ with $(a,b,c,d)=(\infty,\infty,-1.5,-2)$.} \label{fig:trop-markov-1}
\end{figure}
\begin{figure}
    \centering
    \includegraphics[width=.8\textwidth]{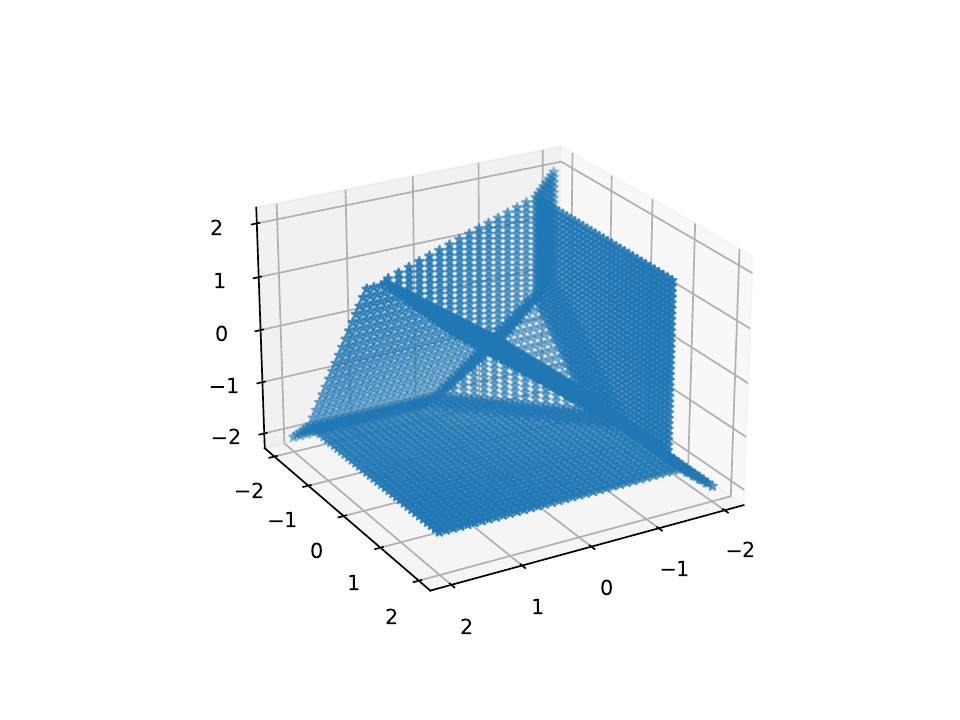}
    \caption{$\Trop(S_{ABCD})$ with $(a,b,c,d)=(-1.3,\infty,-1.5,-2.3)$.} \label{fig:trop-markov-2}
\end{figure}

\subsubsection{Analysis of Cells} By $f_{abcd}$, we have at most 8 cells in the space, in which we classify them according to the monomial degrees:
\begin{itemize}
    \item the \emph{cubic cell} $\cell(X_1X_2X_3)$,
    \item the \emph{quadratic cells} $\cell(X_1^2)$, $\cell(X_2^2)$, $\cell(X_3^2)$,
    \item the \emph{linear cells} $\cell(-AX_1)$,$\cell(-BX_2)$, $\cell(-CX_3)$, and
    \item the \emph{0-degree cell} $\cell(-D)$.
\end{itemize}
We may drop signs and view $\cell(D):=\cell(-D)$ and $\cell(AX_1):=\cell(-AX_1)$, etc. The linear cells and 0-degree cell are called \emph{sub-quadratic cells}. Various observations about the nonemptiness of cells and adjacency of cells are established by elementary algebra of inequalities: they are listed as follows.

\begin{lemma}[Nonemptiness]
\label{lem:nonempty}
    \begin{enumerate}[(a)]
        \item The cubic cell and quadratic cells are always nonempty.
        \item The 0-degree cell is nonempty iff $d<\infty$.
        \item The linear cell $\cell(AX_1)$ is nonempty iff $2a<d$; similar for other two linear cells.
    \end{enumerate}
\end{lemma}
\iftrue
\begin{proof}
    (a) If $M\gg 0$, relative to $|a|$, $|b|$, $|c|$, and $|d|$, then $(-M,-M,-M)$, $(-M,M,M)$, $(M,-M,M)$, and $(M,M,-M)$ are respectively in $\cell(X_1X_2X_3)$, $\cell(X_1^2)$, $\cell(X_2^2)$, and $\cell(X_3^2)$.

    (b) If $M\gg 0$, relative to $|a|$, $|b|$, $|c|$, and $|d|$, then $(M,M,M)$ is in $\cell(D)$, unless $d=\infty$.

    (c) If a point $(x_1,x_2,x_3)$ exists in $\cell(AX_1)$, then that point verifies $a+x_1<d$ and $a+x_1<2x_1$. Hence $a,x_1<\infty$ and
    \[2a=a+(a+x_1)-x_1<a+(2x_1)-x_1=a+x_1<d.\]
    If $2a<d$, then $a<\infty$, and as $a<d-a$, there is $x_1<\infty$ sitting between two strictly. If $M\gg 0$, then $(x_1,M,M)\in\cell(AX_1)$.
\end{proof}
\fi

\begin{lemma}[Adjacency]
\label{lem:adjacency}
    \begin{enumerate}[(a)]
        \item The cubic cell and quadratic cells are pairwise adjacent.
        \item The 0-degree cell, if nonempty, is adjacent to exactly one of $\cell(AX_1)$ or $\cell(X_1^2)$, depending on whether $\cell(AX_1)\neq\varnothing$ or not; similar for other two variables.
        \item A linear cell $\cell(AX_1)$, if nonempty, is adjacent to $\cell(X_1^2)$; similar for other two.
        \item Any two nonempty linear cells are adjacent.
        \item The cells $\cell(AX_1)$ and $\cell(X_2^2)$ are adjacent iff $\cell(AX_1)\neq\varnothing$ and $a<b$.
        Similar for other 5 pairs, e.g., $\cell(CX_3)$ and $\cell(X_1^2)$.
        \item\label{enum:adjacency-d-x1x2x3} The cubic cell and the 0-degree cell are adjacent iff
        \begin{equation}
        \label{eqn:adjacency-d-x1x2x3}
        	d<\min\left\{0,2a-d\right\}+\min\{0,2b-d\}+\min\{0,2c-d\}.
        \end{equation}
        \item\label{enum:adjacency-lin-x1x2x3} The cubic cell and a linear cell $\cell(AX_1)$ are adjacent iff $\cell(AX_1)\neq\varnothing$ and
        \begin{equation}
        \label{eqn:adjacency-lin-x1x2x3}
        a<\min\{0,b-a\}+\min\{0,c-a\}.
        \end{equation}
        Similar for other two linear cells.
    \end{enumerate}
\end{lemma}
\iftrue
\begin{proof}
    (a) For $M,N\gg 0$, observe that the point $(-M,-N,-M-N)$ is in $\partial\cell(X_1X_2X_3\cap X_3^2)$. Argue symmetrically for other quadratic cells.
    
    For two quadratic cells, for $M,N\gg 0$, observe that the point $(-M,-M,N)$ is in $\partial\cell(X_1^2\cap X_2^2)$. Argue symmetrically for other pairs.
    
    (b) Suppose $d<\infty$. If $2a<d$, so that $\cell(AX_1)\neq\varnothing$, then for any $M,N\gg 0$, we have $(d-a,M,N)\in\partial\cell(AX_1\cap D)$. Meanwhile, $\cell(D)$ cannot be adjacent to $\cell(X_1^2)$ because any point on $2x_1=d$ verifies $a+x_1<\frac12d+x_1=2x_1=d$, which cannot belong to the closure of $\cell(D)$ nor $\cell(X_1^2)$. 
    If $2a\geq d$, then $\cell(AX_1)=\varnothing$, and for any $M,N\gg 0$, we have $(\frac12d,M,N)\in\partial\cell(X_1^2\cap D)$.
    
    (c) Suppose $2a<d$. For any $M,N\gg 0$, any point $(a,M,N)\in\partial\cell(AX_1\cap X_1^2)$.
    
    (d) Suppose, without loss of generality, $\cell(AX_1)$ and $\cell(BX_2)$ are nonempty, and $a\leq b$. Then we have $2a\leq 2b<d$. Any point $(x_1,x_2,x_3)\in\partial\cell(AX_1\cap BX_2)$ verifies $a+x_1=b+x_2$ and the inequalities
    \begin{align*}
    	a+x_1 &< 2x_1, &
    	a+x_1 &< 2x_2, &
    	a+x_1 &< 2x_3, \\
    	a+x_1 &< x_1+x_2+x_3, &
    	a+x_1 &< c+x_3, &
    	a+x_1 &< d.
    \end{align*}
    These inequalities solve to $\max\{a,2b-a\}<x_1<d-a$ and $x_3>\max\{\frac12(a+x_1),b-x_1,c-a+x_1\}$. As $a<d-a$ (from $2a<d$) and $2b-a<d-a$ (from $2b<d$), we see that $x_1$ lies on a nondegenerate open interval; so if we choose $x_3\gg 0$, we see that $(x_1,a-b+x_1,x_3)$ lies on the adjacency locus.
    
    (e) If $(x_1,x_2,x_3)\in\partial\cell(D\cap X_1X_2X_3)$, it verifies $d=x_1+x_2+x_3$, thus $x_3=d-x_1-x_2$, and thus
    \begin{align*}
    	d &< 2x_1, &
    	d &< 2x_2, &
    	d &< 2(d-x_1-x_2), \\
    	d &< a+x_1, &
    	d &< b+x_2, &
    	d &< c+(d-x_1-x_2).
    \end{align*}
    Solving the inequality, just as above, we see that the inequality~\eqref{eqn:adjacency-d-x1x2x3} is equivalent to that the above inequalities have an open set of solutions $(x_1,x_2)$.
    
    (f) If $(x_1,x_2,x_3)\in\partial\cell(AX_1\cap X_1X_2X_3)$, it verifies $a+x_1=x_1+x_2+x_3$, thus $x_3=a-x_2$, and thus
    \begin{align*}
    	a+x_1 &< 2x_1, &
    	a+x_1 &< 2x_2, &
    	a+x_1 &< 2(a-x_2), \\
    	a+x_1 &< b+x_2, &
    	a+x_1 &< c+(a-x_2), &
    	a+x_1 &< d.
    \end{align*}
    Solving the inequality, we see that the inequality~\eqref{eqn:adjacency-lin-x1x2x3} and $2a<d$ is equivalent to that the above inequalities have an open set of solutions $(x_1,x_2)$.
\end{proof}
\fi


In Lemma~\ref{lem:adjacency}(\ref{enum:adjacency-d-x1x2x3}), the inequality is false if $\cell(D)$ is empty: indeed, if we plug in $d=\infty$ we have $\infty<-\infty$, which is (conventionally) false. Also, combining \eqref{eqn:adjacency-d-x1x2x3} and \eqref{eqn:adjacency-lin-x1x2x3}, we have the following

\begin{corollary}
    \label{lem:cubic-cell-adj-subquadratic-iff-meromorphic}
    The cubic cell is adjacent to a sub-quadratic cell iff $\min\{a,b,c,d\}<0$.
\end{corollary}

\section{Vieta Involutions in Tropicalizations}
\label{sec:tropical-vieta-involutions}

In this section, we argue that the appropriate tropicalizations of the Vieta involutions
\begin{align*}
    s_1(X_1,X_2,X_3) &= (-X_1-X_2X_3+A, X_2, X_3), \\
    s_2(X_1,X_2,X_3) &= (X_1,-X_2-X_1X_3+B,X_3), \\
    s_3(X_1,X_2,X_3) &= (X_1,X_2,-X_3-X_1X_2+C)
\end{align*}
are
\begin{align}
    \trop(s_1)(x_1,x_2,x_3) &= (\min\{2x_2,2x_3,b+x_2,c+x_3,d\}-x_1,x_2,x_3), \label{eqn:trop-vieta-1} \\
    \trop(s_2)(x_1,x_2,x_3) &= (x_1,\min\{2x_1,2x_3,a+x_1,c+x_3,d\}-x_2,x_3), \label{eqn:trop-vieta-2} \\
    \trop(s_3)(x_1,x_2,x_3) &= (x_1,x_2,\min\{2x_1,2x_2,a+x_1,b+x_2,d\}-x_3). \label{eqn:trop-vieta-3}
\end{align}
The appropriateness can only be validated at a special subset, the \emph{tropical domains} $\mathrm{dom}(s_i)$'s (see Proposition~\ref{lem:tropical-domain}). The intersection of these domains lie densely in the special invariant subset called the \emph{skeleton} $Sk(a,b,c,d)\subset\Trop(S_{ABCD})$ of the tropicalization, which is the boundary of the cubic cell $\cell(X_1X_2X_3)$ (see Definition~\ref{def:skeleton}).

\subsection{Vieta Involutions in 1 Variable}

Consider the quadratic equation $X^2+C_1X+C_0=0$ over a non-archimedean field $(K,\val)$, with $C_0\neq 0$. Let $c_1=\val(C_1)$ and $c_0=\val(C_0)$. For its zeroes $X_1,X_2$, we can estimate their values by the tropicalization:
\begin{align}
    \Trop(X^2+C_1X+C_0=0) &= \left\{\val(X_1),\val(X_2)\right\} \nonumber\\
    &= \left\{x_0\in\R : \begin{array}{l}\min\{2x,c_1+x,c_0\}\text{ is not} \\ \text{differentiable at }x_0 \end{array}\right\} \nonumber\\
    &=\begin{cases} \{c_1,c_0-c_1\} & (2c_1<c_0), \\ \{\frac12c_0\} & (2c_1\geq c_0).\end{cases} \label{eqn:newton-polytope-for-quadratic}
\end{align}

The Vieta involution $s(X)=C_1-X=C_0/X$ should be switching the zeroes $X_1$ and $X_2$; as $X_1X_2=C_0$, we see that $\val(X_1)+\val(X_2)=\val(C_0)=c_0$, hence $x\mapsto c_0-x$ is transposing the values of $X_1$ and $X_2$, as we can see from the sets $\{c_1,c_0-c_1\}$ and $\{\frac12c_0\}$.

\subsubsection{General Quadratic Family}

In general, suppose we have an affine hypersurface of the form $X_0+C_1(Y)X_0+C_0(Y)=0$ with additional variables $Y=(Y_1,\ldots,Y_k)$, and $C_1,C_0$ being Laurent polynomials with respect to $Y$.

Then, on the locus where $C_0(Y)\neq 0$, the tropicalization of the Vieta involution $(X_0;Y)\mapsto(C_0(Y)/X_0;Y)$ (that aligns with the above observation) is sending $\val(X_0)$ to $\val(C_0(Y))-\val(X_0)$. Especially when $y=\val(Y)$ (i.e., $y=(y_1,\ldots,y_k)$ with $y_i=\val(Y_i)$) with $y\notin\Trop(C_0(Y)=0)$, then $\val(C_0(Y))=\trop(C_0)(y)$, and we have
\[x_0\mapsto\trop(C_0)(y)-x_0\]
switching the values of roots.

\subsubsection{Tropcial Domain}

If we view the equation $X_0^2+C_1(Y)X_0+C_0(Y)=0$ as an affine hypersurface $V$ in variables $(X_0;Y)=(X_0;Y_1,\ldots,Y_k)$, then we may ask about the relation between the Vieta involution $s(X_0;Y)=(C_0(Y)/X_0;Y)$ and its tropicalized version $\trop(s)(x_0;y)=(\trop(C_0)(y)-x_0;y)$, say whether the valuation conjugates two. But we need to be aware of the locus $C_0(Y)=0$ as an exception; that is, we rather consider
\[\Trop(V)\setminus(\R\times\Trop(C_0(Y)=0))\]
as the correct domain of $\trop(s)$. This is an open subset of $\Trop(V)$, called the \emph{tropical domain} of $\trop(s)$, and denoted by $\mathrm{dom}(s)$. On there, we see that the valuation conjugates $s$ and $\trop(s)$:

\begin{proposition}
    \label{lem:tropical-domain}
    Let $V$ be an affine hypersurface, with space coordinates $(X_0;Y_1,\ldots,Y_k)=:(X_0;Y)$, by the equation $X_0^2+C_1(Y)X_0+C_0(Y)=0$. Let $s(X_0;Y)=(C_0(Y)/X_0;Y)$ be the Vieta involution and let $\trop(s)(x_0;y)=(\trop(C_0)(y)-x_0;y)$ be its tropicalization.
    
    Let $P=(P_0;Q_1,\ldots,Q_k)\in V(K)$ be a $K$-point such that none of its coordinates are zero. If $\val(P)=(\val(P_0);\val(Q_1),\ldots,\val(Q_k))\in\mathrm{dom}(s)$, then
    \[\trop(s)(\val(P))=\val(s(P))\in\mathrm{dom}(s).\]
\end{proposition}

Note that the condition $C_0(Q_1,\ldots,Q_k)\neq 0$ is \emph{weaker} than $\val(P)\in\mathrm{dom}(s)$.

\begin{proof}
    Define $p_0=\val(P_0)$, $Q=(Q_1,\ldots,Q_k)$, and $q=\val(Q)$. Because $(p_0;q)\in\mathrm{dom}(s)$, we have $q\notin\Trop(C_0(Y)=0)$. Hence $\trop(C_0)(q)$ is evaluated by a unique tropical monomial of $\trop(C_0)$. The value of this tropical monomial must equal to the value of $C_0(Q)$, because the corresponding (polynomial) monomial of $C_0$ takes the least value among its monomials. Hence we have $\trop(C_0)(q)=\val(C_0(Q))$. This suffices to show that
    \begin{align*}
        \trop(s)(\val(P)) = \trop(s)(p_0;q) &= (\trop(C_0)(q)-p_0;q) \\
        &= (\val(C_0(Q))-\val(P_0);\val(Q)) \\
        &= (\val(C_0(Q)/P_0);\val(Q))=\val(s(P)).
    \end{align*}
    The inclusion in $\mathrm{dom}(s)$ is clear, because it only depends on the $q$ part of the coordinates.
\end{proof}

Outside of the tropical domain, all we can say about how $\trop(s)(\val(P))$ and $\val(s(P))$ are related is that each coordinate of the former is less than the latter: symbolically, we write
\begin{equation}
    \label{eqn:251218-1}
    \trop(s)(\val(P))\leq\val(s(P)).
\end{equation}

\subsection{Markov Context}
\label{sec:markov-trop-vieta}

First above all, we now note that equations \eqref{eqn:trop-vieta-1}, \eqref{eqn:trop-vieta-2}, and \eqref{eqn:trop-vieta-3} are re-iterating the formula $\trop(s)(x_0;y)=(\trop(C_0)(y)-x_0;y)$. For example, to tropicalize $s_3$, we put $C_0(X_1,X_2)=X_1^2+X_2^2-AX_1-BX_2-D$ and get
\[\trop(s_3)(x_3;x_1,x_2)=(\min(2x_1,2x_2,a+x_1,b+x_2,d)-x_3;x_1,x_2),\]
and rearrange the coordinates in the correct order to get \eqref{eqn:trop-vieta-3}.

By Proposition \ref{lem:tropical-domain}, the intersection $\bigcap_{i=1}^3\mathrm{dom}(s_i)$ would be the most dynamically relevant subset of $\Trop(S_{ABCD})$, under the tropicalized Vieta involutions. However, as the tropical domain $\mathrm{dom}(s_i)$ is not $\trop(s_j)$-invariant whenever $i\neq j$, the correct ``main invariant'' subset should be
\[\bigcap_{\gamma\in\trop(\Gamma_{ABCD})}\gamma.\bigcap_{i=1}^3\mathrm{dom}(s_i),\]
where $\trop(\Gamma_{ABCD})$ is a shorthand of $\langle\trop(s_1),\trop(s_2),\trop(s_3)\rangle$.

\begin{remark}
    Later, we will show that there is a group map $\trop\colon\Gamma_{ABCD}\to\trop(\Gamma_{ABCD})$, to be called the \emph{tropical representation}, after developing some theory of tropical actions. See Definition~\ref{def:tropical-representation} below. The notation $\trop(\Gamma_{ABCD})$ may be understood as the image of this group map, but to avoid circular logic, we stick to view the group as a formal name of the group (of piecewise-linear automorphisms of $\R^3$) generated by tropicalized Vieta involutions. (Nonetheless, if we insist, we may use El'Huti's result $\Gamma_{ABCD}\cong(\Z/2\Z)^{\ast 3}$~\cite[Thm. 1]{ElHuti} to introduce the map in advance.)
\end{remark}

To study the shape of this ``main invariant'' subset, we introduce a subset of $\Trop(S_{ABCD})$.

\begin{definition}[Skeleton]
\label{def:skeleton}
    The boundary of the cubic cell $\cell(X_1X_2X_3)$ is called the \emph{skeleton} of $S_{ABCD}$, and is denoted by $Sk(a,b,c,d)$, where $a=\val(A)$, $b=\val(B)$, $c=\val(C)$, and $d=\val(D)$.
\end{definition}

Recall that the cubic cell is characterized by the following inequality,
\[x_1+x_2+x_3<\min\left\{\begin{array}{l}2x_1,2x_2,2x_3, \\ a+x_1,b+x_2,c+x_3,d \end{array}\right\}.\]
If we trace adjacent loci attached to the cubic cell, we see that the skeleton is characterized by the following equality:
\begin{equation}
    \label{eqn:skeleton-equation}
    x_1+x_2+x_3=\min\left\{\begin{array}{l}2x_1,2x_2,2x_3, \\ a+x_1,b+x_2,c+x_3,d \end{array}\right\}.
\end{equation}

The name of this subset comes from Kontsevich--Soibelman skeleton \cite[\S{6.6}]{KS06} of a smooth proper algebraic variety over a non-archimedean field with trivial canonical divisor, and is named \emph{skeleton} in the K3 context, according to \cite[\S{2.4.2}]{filip2019tropical}. 
But instead of introducing this set by heritage, we show that the ``main invariant'' subset is a dense subset of the skeleton.

\begin{proposition}
\label{lem:pre-skeleton}
    The intersection $\bigcap_{\gamma\in\trop(\Gamma_{ABCD})}\bigcap_{i=1}^3\gamma.\mathrm{dom}(s_i)$ is a dense subset of the skeleton of the surface.
\end{proposition}

Furthermore, the intersection is an open subset if one of the parameters has negative value. But we will establish this after discussing the action of $\trop(\Gamma_{ABCD})$ on the skeleton (see Lemma~\ref{lem:tropical-domain-open}).

\subsubsection{Proof of Proposition \ref{lem:pre-skeleton}}

We need two facts to be shown: inclusion and density. The inclusion is easy, as seen in the following
\begin{lemma}
    The intersection of tropical domains $\bigcap_{i=1}^3\mathrm{dom}(s_i)$ is in the boundary of the cubic cell $\cell(X_1X_2X_3)$.
\end{lemma}
\begin{proof}
    Any point $x\in\Trop(S_{ABCD})$ which is not on the boundary of the cubic cell admits two tropical monomials among $2x_1$, $2x_2$, $2x_3$, $a+x_1$, $b+x_2$, $c+x_3$, or $d$, whose values are the same (and equal to $f_{abcd}(x)$). By a manual check, one sees that these monomials appear simultaneously in one of the sets $S_1=\{2x_2,2x_3,b+x_2,c+x_3,d\}$, $S_2=\{2x_1,2x_3,a+x_1,c+x_3,d\}$, or $S_3=\{2x_1,2x_2,a+x_1,b+x_2,d\}$.
    
    Say $S_1$ is containing them. Then the minimum of $S_1$ (evaluated at $x$) is attained at two or more tropical monomials, thus $(x_2,x_3)\in\Trop(X_2^2+X_3^2=BX_2+CX_3+D)$, whence $x\notin\mathrm{dom}(s_1)$.
\end{proof}

\begin{lemma}
\label{lem:skeleton-invariance}
    The skeleton is invariant under the group $\trop(\Gamma_{ABCD})$.
\end{lemma}
\begin{proof}
    By symmetry, it suffices to show the invariance by $\trop(s_1)$. Recall the equality~\eqref{eqn:skeleton-equation} characterizing the skeleton.

    Write $\trop(s_1)(x_1,x_2,x_3)=(A-x_1,x_2,x_3)$, with $A=\min\{2x_2,2x_3,b+x_2,c+x_3,d\}$. Then the equation~\eqref{eqn:skeleton-equation} may be rewritten as $x_1+x_2+x_3-\min\{2x_1,a+x_1,A\}=0$. If we replace $x_1$ in the left hand side to $A-x_1$, we have
    \begin{align*}
        &(A-x_1)+x_2+x_3-\min\{2(A-x_1),a+(A-x_1),A\} \\
        &= x_1+x_2+x_3-(2x_1-A)-\min\{2A-2x_1,a+A-x_1,A\} \\
        &= x_1+x_2+x_3-\min\{2A-2x_1+(2x_1-A),a+A-x_1+(2x_1-A),A+(2x_1-A)\} \\
        &= x_1+x_2+x_3-\min\{A,a+x_1,2x_1\},
    \end{align*}
    showing the invariance.
\end{proof}

So this shows that the intersection $\bigcap_{\gamma\in\trop(\Gamma_{ABCD})}\bigcap_{i=1}^3\gamma.\mathrm{dom}(s_i)$ is in the skeleton. 
Density is more involved, and this requires more discussion on the structure of the boundary of the cubic cell. 

\begin{lemma}
    The domain $\mathrm{dom}(s_1)$ intersected with the skeleton is dense in the skeleton.
\end{lemma}
\begin{proof}
    It suffices to show that the intersection of the skeleton with $T=\R\times\Trop(X_2^2+X_3^2=BX_2+CX_3+D)$ is nowhere dense in the skeleton. Observe that the skeleton consists of at most 7 adjacency loci, and the set $T$ consists of (convex) subsets of hyperplanes parallel to the $x_1$-axis. So it suffices to show that the hyperplanes for adjacency loci and those for $T$ intersect transversally.
    
    To see this, it suffices to see that every adjacency locus in the skeleton has the hyperplane with a normal vector $n=(n_1,n_2,n_3)$ with $n_1\neq 0$. It turns out that all adjacency locus except $\partial\cell(X_1X_2X_3\cap AX_1)$ enjoy this property. Nonetheless, as the hyperplane
    \[H=\{(x_1,x_2,x_3) : x_2+x_3=a\}\]
    containing the exceptional adjacency locus 
    intersects hyperplanes of $T$ transversally (hyperplanes of $T$ are normal to $(0,1,-1)$, $(0,1,0)$, $(0,0,1)$, $(0,2,-1)$, or $(0,-1,2)$, none of them parallel to $(0,1,1)$), we see that this exception does not cause a problem.
\end{proof}

By symmetry, we see that the intersection $\bigcap_{i=1}^3\mathrm{dom}(s_i)$ is a dense open subset of the skeleton. For any $\gamma\in\trop(\Gamma)$, its image $\gamma.\bigcap_{i=1}^3\mathrm{dom}(s_i)$ is again dense open in the skeleton. 
The skeleton is a closed subset of $\R^3$, thus a complete metric space. Thus by the Baire category theorem, the intersection $\bigcap_{\gamma\in\trop(\Gamma)}\bigcap_{i=1}^3\gamma.\mathrm{dom}(s_i)$ is a dense subset of the skeleton. This proves Proposition \ref{lem:pre-skeleton}.

\subsection{El'Huti's Isomorphism and Linearlization}

As our Markov surface $S_{ABCD}$ is defined over algebraically closed field, by El'
Huti~\cite[Thm. 1]{ElHuti}, the automorphism group $\Gamma_{ABCD}$ generated by Vieta involutions is isomorphic to the free product of order 2 groups generated by each Vieta involutions. Hence, for any group $G$, one can define a group map $\Gamma_{ABCD}\to G$ by sending each $s_i\in\Gamma_{ABCD}$, $i=1,2,3$, to an order 2 element of $G$.

Among such group maps, the one that we will repetitively introduce is the following.
\begin{proposition}
    \label{lem:PGL2Z-Vieta-group-identification}
    \begin{enumerate}[(a)]
        \item There is a group map $j\colon\Gamma_{ABCD}\to\mathsf{PGL}_2(\Z)$, sending
    \begin{align}
        \label{eqn:PGL2Z-Vieta-group-identification}
        s_1 &\mapsto\begin{bmatrix} -1 & -2 \\ 0 & 1 \end{bmatrix}, &
        s_2 &\mapsto\begin{bmatrix} 1 & 0 \\ -2 & -1 \end{bmatrix}, &
        s_3 &\mapsto\begin{bmatrix} 1 & 0 \\ 0 & -1 \end{bmatrix}.
    \end{align}
        \item The image of $j$ is the kernel of the mod 2 reduction map $\mathsf{PGL}_2(\Z)\to\mathsf{PGL}_2(\Z/2\Z)$.
    \end{enumerate}
\end{proposition}
\begin{proof}
    (a) This is immedaite from El'Huti's result, noting that all images $j(s_i)$'s are set to be involutions in $\mathsf{PGL}_2(\Z)$.
    
    (b) This is standard; see \cite{KoblitzECMF}, \S{III.1}, Exercise~13(c) (see ibid., p.~99 for their $\overline{\Gamma}(2)$).
\end{proof}

\section{The Case of Holomorphic Parameters}
\label{sec:holomorphic}

We say the parameters $A,B,C,D$ of $S_{ABCD}$ are \emph{holomorphic} if their values are nonnegative, i.e., $\min\{a,b,c,d\}\geq 0$, and \emph{meromorphic} otherwise. (The names are from an analogy with the field of Puiseux series.) With holomorphic parameters, all tropical dynamics on the skeleton are isomorphic to the linear action of $\mathsf{PGL}_2(\Z)$ on the plane modulo antipodes, $\R^2/\langle\pm 1\rangle$. This section is devoted to establish this fact.

\subsection{All Holomorphic Skeletons are the Same}

The following fact establishes that all skeletons with holomorphic parameters are the same subsets of $\R^3$, as sketched in Figure~\ref{fig:tetrahedral-cone}.

\begin{figure}
    \centering
    \begin{tikzpicture}[scale=.8,
        x={({(13/16)^(1/2)},0)}, 
        y={(0,{(15/16)^(1/2)})}, 
        z={({-.5*cos(30)},{-.5*sin(30)})}]
        
            \draw[gray,->] (-5,0,0) -- (1,0,0) node[right] {$x_1$};
            \draw[gray,->] (0,-5,0) -- (0,1,0) node[above] {$x_2$};
            \draw[gray,->] (0,0,-5) -- (0,0,1) node[left] {$x_3$};

            \foreach \i in {.5}{
            \fill[gray,opacity=\i] (0,0,0) -- (-5,-5,0) -- (-5,0,-5) -- cycle;
            \fill[gray,opacity=\i] (0,0,0) -- (-5,-5,0) -- (0,-5,-5) -- cycle;
            \fill[gray,opacity={.5*\i}] (0,0,0) -- (0,-5,-5) -- (-5,0,-5) -- cycle;
            }
            \draw (-5,-5,0) -- (0,0,0) ;
            \draw (-5,0,-5) -- (0,0,0) ;
            \draw (0,-5,-5) -- (0,0,0) ;
        \end{tikzpicture}
    \caption{The skeleton with holomorphic parameters.}
    \label{fig:tetrahedral-cone}
\end{figure}

\begin{proposition}
    \label{lem:holomorphic-same-skeletons}
    The followings are equivalent.
    \begin{enumerate}[(a)]
        \item The parameters are holomorphic.
        \item The origin $(0,0,0)$ is in the skeleton $Sk(a,b,c,d)$.
        \item The skeleton equals to the locus of points $(x_1,x_2,x_3)\in\R^3$ such that
        \begin{equation}
            \label{eqn:image-cone-parametrization}
            x_1+x_2+x_3=\min\{2x_1,2x_2,2x_3\}.
        \end{equation}
    \end{enumerate}
\end{proposition}
\begin{proof}
    ((a)$\Leftrightarrow$(b)) We know that the skeleton is characterized by the equation~\eqref{eqn:skeleton-equation}. So plugging in $(x_1,x_2,x_3)=(0,0,0)$ verifies \eqref{eqn:skeleton-equation} if and only if $0=\min\{0,a,b,c,d\}$ if and only if $0\leq\min\{a,b,c,d\}$.

    ((c)$\Rightarrow$(b)) The set in question contains $(0,0,0)$.

    ((a)$\Rightarrow$(c)) It suffices to show that \eqref{eqn:skeleton-equation} if and only if \eqref{eqn:image-cone-parametrization}, provided that $\min\{a,b,c,d\}\geq 0$.

    Suppose \eqref{eqn:skeleton-equation} holds yet \eqref{eqn:image-cone-parametrization} fails, so that we have
    \[x_1+x_2+x_3<\min\{2x_1,2x_2,2x_3\}.\]
    Then we have $x_1+x_2+x_3<2x_i$, for $i=1,2,3$. They yield $x_3<x_1-x_2$ and $x_3<x_2-x_1$, so we have $x_3<-|x_1-x_2|\leq 0$, etc., so that all coordinates are negative. Hence
    \begin{align*}
        x_1+x_2+x_3 &<x_1\leq a+x_1, & x_1+x_2+x_3 &<x_2\leq b+x_2, \\
        x_1+x_2+x_3 &<x_3\leq c+x_3, & x_1+x_2+x_3 &<0\leq d,
    \end{align*}
    and the strict inequality for \eqref{eqn:skeleton-equation} follows. Contradiction, so \eqref{eqn:image-cone-parametrization} must hold.

    Suppose \eqref{eqn:image-cone-parametrization} holds. Suppose
    \[\min\{2x_1,2x_2,2x_3\}=x_1+x_2+x_3>\min\{a+x_1,b+x_2,c+x_3,d\}.\]
    By a similar argument form the above paragraph, we have $x_1,x_2,x_3\leq 0$. The strict inequality above yields that \emph{one} of the following must hold:
    \begin{align*}
        a &< x_2+x_3\leq 0, & b &< x_1+x_3\leq 0, \\
        c &< x_1+x_2\leq 0, & d &< x_1+x_2+x_3\leq 0.
    \end{align*}
    But then we have $\min\{a,b,c,d\}<0$. Contradiction to (a), and \eqref{eqn:skeleton-equation} is verified.
\end{proof}

\subsection{All Holomorphic Tropical Dynamics are the Same} Next, we establish a
\begin{proposition}
    \label{lem:holomorphic-same-dynamics}
    Suppose $\min\{a,b,c,d\}\geq 0$. Then $\trop(s_i)$, $i=1,2,3$, on $Sk(a,b,c,d)$ appear as
    \begin{align}
        \trop(s_1)(x_1,x_2,x_3) &= \left(\min\{2x_2,2x_3\}-x_1,x_2,x_3\right), \label{eqn:trop-s1-holomorphic} \\
        \trop(s_2)(x_1,x_2,x_3) &= \left(x_1,\min\{2x_1,2x_3\}-x_2,x_3\right), \label{eqn:trop-s2-holomorphic} \\
        \trop(s_3)(x_1,x_2,x_3) &= \left(x_1,x_2,\min\{2x_1,2x_2\}-x_3\right). \label{eqn:trop-s3-holomorphic}
    \end{align}
\end{proposition}
\begin{proof}
    We verify \eqref{eqn:trop-s1-holomorphic}, and argue by symmetry. That is, we verify that $\min\{2x_2,2x_3,b+x_2,c+x_3,d\}=\min\{2x_2,2x_3\}$ on $Sk(a,b,c,d)$, or equivalently,
    \[\min\{2x_2,2x_3\}\leq\min\{b+x_2,c+x_3,d\}.\]
    As $Sk(a,b,c,d)\subset(-\infty,0]^3$ from Proposition~\ref{lem:holomorphic-same-skeletons}, and $b,c,d\geq 0$, we have
    \begin{align*}
        2x_2 &\leq x_2\leq b+x_2, & 2x_3 &\leq x_3\leq c+x_3, &
        2x_2 &\leq 0\leq d.
    \end{align*}
    These suffice to verify that $\min\{2x_2,2x_3\}$ is no greater than $b+x_2$, $c+x_3$, and $d$.
\end{proof}

\subsection{The Cayley Cubic} The choice $(A,B,C,D)=(0,0,0,4)$ of parameters yields a special member in the family of Markov surfaces, called the \emph{Cayley cubic surface} $S_{0004}$,
\[S_{0004}\colon X_1^2+X_2^2+X_3^2+X_1X_2X_3=4.\]
Here, we have $a=b=c=\val(0)=\infty$, and $d=\val(4)\geq 0$; so this is an example of a holomorphic parameter.

It is known that (see, e.g., \cite[\S{2.8}]{Can09}) there is an isomorphism of varieties
\begin{align}
    \PreConeParam\colon\mathbb{G}_m^2/\langle\pm 1\rangle &\to S_{0004}, \label{eqn:pre-cone-parametrization}\\
    (U,V)^{\pm 1} &\mapsto\left(-U-\frac1U,-V-\frac1V,-UV-\frac1{UV}\right), \nonumber
\end{align}
where the action of $\pm 1$ on $\mathbb{G}_m^2$ is given by $\pm1.(U,V)=(U^{\pm1},V^{\pm1})$. 
Consider the monomial action of $\mathsf{GL}_2(\Z)$ on $\mathbb{G}_m^2$:
\[\begin{bmatrix} a & b \\ c & d \end{bmatrix}.(U,V)=(U^aV^b,U^cV^d).\]
This action induces the action of $\mathsf{PGL}_2(\Z)=\mathsf{GL}_2(\Z)/\langle\pm I_2\rangle$ on $\mathbb{G}_m^2/\langle\pm 1\rangle$, and hence the Cayley cubic $S_{0004}$ also has the action of $\mathsf{PGL}_2(\Z)$.

To describe this $\mathsf{PGL}_2(\Z)$-action, we recall the group $\Gamma_{0004}$ generated by Vieta involutions. The surface $S_{0004}$ carries more symmetry, by permuting the coordinates. So we introduce $\sigma_{12}(X_1,X_2,X_3)=(X_2,X_1,X_3)$ and $\sigma_{23}(X_1,X_2,X_3)=(X_1,X_3,X_2)$ that generates all permutations. Note that $s_2=\sigma_{23}s_3\sigma_{23}$ and $s_1=\sigma_{12}s_2\sigma_{12}$ with these transpositions.

\begin{lemma}
    \label{lem:PGL2Z-Aut-Cayley-identification}
    Let $\widetilde{\Gamma}_{0004}$ be the group of algebraic automorphisms of the Cayley cubic generated by Vieta involutions, $\sigma_{12}$, and $\sigma_{23}$.
    \begin{enumerate}[(a)]
        \item There is a group map $\widetilde{\jmath}\colon\widetilde{\Gamma}_{0004}\to\mathsf{PGL}_2(\Z)$, extending $j$ in Proposition~\ref{lem:PGL2Z-Vieta-group-identification} and sending
    \begin{align}
        \label{eqn:PGL2Z-Aut-Cayley-identification}
        s_3 &\mapsto\begin{bmatrix} 1 & 0 \\ 0 & -1 \end{bmatrix}, &
        \sigma_{12} &\mapsto\begin{bmatrix} 0 & 1 \\ 1 & 0 \end{bmatrix}, &
        \sigma_{23} &\mapsto\begin{bmatrix} 1 & 0 \\ -1 & -1 \end{bmatrix},
    \end{align}
    in addition. Then the map $\PreConeParam$ is a $\widetilde{\jmath}$-equivariant map. That is, for any $g\in\widetilde{\Gamma}_{0004}$ and $P\in\mathbb{G}_m^2/\langle\pm 1\rangle$, we have $\PreConeParam(\widetilde{\jmath}(g).P)=g.\PreConeParam(P)$.
        \item The map $\widetilde{\jmath}$ is onto.
    \end{enumerate}
\end{lemma}
\begin{proof}
    (a) Conjugating $\widetilde{\Gamma}_{0004}$ by the isomorphism $\PreConeParam\colon\mathbb{G}_m^2/\langle\pm1\rangle\to S_{0004}$ yields a group map into the group of algebraic automorphisms of $\mathbb{G}_m^2/\langle\pm1\rangle$, namely $\widetilde{\Gamma}_{0004}\to\mathsf{Aut}(\mathbb{G}_m^2/\langle\pm1\rangle)$, $\gamma\in\widetilde{\Gamma}_{0004}\mapsto\PreConeParam^{-1}\gamma\PreConeParam$. As $s_1$ and $s_2$ are generated by $s_3$, $\sigma_{12}$, and $\sigma_{23}$, it suffices to verify that the maps $\PreConeParam^{-1}\gamma\PreConeParam$ are monomial transforms, for $\gamma=s_3,\sigma_{12},\sigma_{23}$. The remainder is a computational business. First,
    \begin{align*}
        s_3(\PreConeParam(U,V))&=s_3\left(-U-\frac1U,-V-\frac1V,-UV-\frac1{UV}\right) \\
        &= \left(-U-\frac1U,-V-\frac1V,-\left(U+\frac1U\right)\left(V+\frac1V\right)+UV+\frac1{UV}\right) \\
        &= \left(-U-\frac1U,-V-\frac1V,-\frac{U}V-\frac{V}U\right)=\PreConeParam(U,V^{-1}),
    \end{align*}
    so the matrix $\widetilde{\jmath}(s_3)$ assigned. For $\sigma_{12}$, it is clear that $\sigma_{12}(\Phi(U,V))=\Phi(V,U)$, so the permutation matrix assigned. For $\sigma_{23}$, we need to make a careful comptuation:
    \begin{align*}
        \sigma_{23}(\PreConeParam(U,V))&=\left(-U-\frac1U,-UV-\frac1{UV},-V-\frac1{V}\right) \\
        &= \left(-U-\frac1U,-\frac1{UV}-UV,-\frac{U}{UV}-\frac{UV}{U}\right) =\PreConeParam(U,U^{-1}V^{-1}),
    \end{align*}
    and this explains the choice of the matrix $\widetilde{\jmath}(\sigma_{23})$.

    (b) The group $\mathsf{PGL}_2(\Z/2\Z)$ is generated by mod 2 images of $\widetilde{\jmath}(\sigma_{12})$ and $\widetilde{\jmath}(\sigma_{23})$. Combine with Proposition~\ref{lem:PGL2Z-Vieta-group-identification}(b).
\end{proof}

To put it short, $\PreConeParam$ is linearlizing the action $\Gamma_{0004}\ACTS S_{0004}$. Such a linearlization descends to the tropical regime as well; see Theorem~\ref{thm:tropical-dynamics-cayley-cubic} below. Unfortunately, we need to establish some terms and facts to properly state it.

First, let $\val^-(X):=\min\{0,\val(X)\}$ be the \emph{(nonpositively) rectified value}.
\begin{lemma}
    \label{lem:ReLU-value}
    Let $U\in K^\times$. Then we have $\val^-(U+U^{-1})=-|\val(U)|$.
\end{lemma}
\begin{proof}
    If $\val(U)\neq 0$, then we have $\val(U+U^{-1})=\min\{\val(U),\val(U^{-1})\}=-|\val(U)|<0$ because $\val(U)\neq\val(U^{-1})$. Otherwise, we have $\val(U+U^{-1})\geq 0$, which still yields $\val^-(U+U^{-1})=0=-|\val(U)|$.
\end{proof}
\begin{lemma}
    \label{lem:l1-linfty-equivalence}
    For any $u,v\in\R$, we have $\max\{2|u|,2|v|\}=|u+v|+|u-v|$, and $\max\{|u+v|,|u-v|\}=|u|+|v|$.
\end{lemma}
\begin{proof}
    One implies the other by an appropriate linear change of basis. So we show the latter. It suffices to verify this for a dense set of $(u,v)\in\R^2$. As $(K,\val)$ is an algebraically closed nontrivially valued field, its value group $\val(K^\times)$ is dense in $\R$. Thus we may assume that $u,v\in\val(K^\times)$, and $u,v\neq 0$. As $uv\neq 0$, it follows that $|u+v|\neq|u-v|$; suppose, without loss of generality, $|u+v|>|u-v|$.
    
    Choose any $U,V\in K^\times$ with $u=\val(U)$ and $v=\val(V)$. Then we have
    \begin{align*}
    	-|u|-|v| &= \val\left(U+\frac1U\right)+\val\left(V+\frac1V\right) \\
    	&= \val\left(\left(U+\frac1U\right)\left(V+\frac1V\right)\right) \\
    	&= \val\left(\left(UV+\frac1{UV}\right)+\left(\frac{U}V+\frac{V}U\right)\right) \\
    	&=\min\{-|u+v|,-|u-v|\},
    \end{align*}
    since $\val(UV+\frac1{UV})=-|u+v|$ is nonzero and is different from $\val(\frac{U}V+\frac{V}U)=-|u-v|$. Multiplying $(-1)$ to this, we get the desired.
\end{proof}

The natural ``tropicalization'' of $\mathbb{G}_m^2/\langle\pm 1\rangle$ is the plane modulo antipodes, $\R^2/\langle\pm 1\rangle$. From there, define the \emph{cone parametrization map}
\begin{align}
    \ConeParam\colon\R^2/\langle\pm 1\rangle &\to \R^3, \label{eqn:cone-parametrization} \\
    \pm(u,v) &\mapsto (-|u|,-|v|,-|u+v|). \nonumber
\end{align}
By Lemma~\ref{lem:ReLU-value}, with $\val((U,V)^{\pm 1})=\pm(\val(U),\val(V))$ and $\val^-$ extended to $K^3\to(-\infty,0]^3$, we have
\begin{equation}
\label{eqn:cone-param-commutativity}
\ConeParam(\val((U,V)^{\pm 1}))=\val^-(\PreConeParam((U,V)^{\pm 1})).
\end{equation}

\begin{lemma}
    \label{lem:cone-parametrization}
    \begin{enumerate}[(a)]
        \item The map $\ConeParam$ is injective.
        \item The image of $\ConeParam$ equals to the set of points $(x_1,x_2,x_3)\in\R^3$ verifying \eqref{eqn:image-cone-parametrization}.
    \end{enumerate}
\end{lemma}
\begin{proof}
    (a) Let $\ConeParam(\pm(u,v))=\ConeParam(\pm(u',v'))$. Then from $|u|=|u'|$ and $|v|=|v'|$, we have $u=\varepsilon_1u'$ and $v=\varepsilon_2v'$ by some $\varepsilon_1,\varepsilon_2\in\{\pm 1\}$. Since $|u+v|=|u'+v'|$ as well, we have $|u+v|=|u'+\varepsilon_1\varepsilon_2v'|=|u'+v'|$. Squaring both sides, we have $\varepsilon_1\varepsilon_2u'v'=u'v'$. If $\varepsilon_1\varepsilon_2=1$, we have $\pm(u,v)=\pm(u',v')$. Otherwise, we have $u'v'=0$ and thus re-choosing $\varepsilon_1$ to $-\varepsilon_1$ or $\varepsilon_2$ to $-\varepsilon_2$ makes no difference, reducing to the previous.

    (b) To see that the image of $\phi$ verifies \eqref{eqn:image-cone-parametrization}, we compute with Lemma~\ref{lem:l1-linfty-equivalence}:
    \begin{align*}
        -|u|-|v|-|u+v| &= \min\{-|u+v|,-|u-v|\}-|u+v| \\
        &= \min\{-2|u+v|,-|u+v|-|u-v|\} \\
        &= \min\{-2|u+v|,\min\{-2|u|,-2|v|\}\} \\
        &= \min\{-2|u|,-2|v|,-2|u+v|\}.
    \end{align*}
    
    To show the converse, note that condition \eqref{eqn:image-cone-parametrization} is claiming that the following three inequalities hold,
    \begin{align*}
        x_1+x_2+x_3 &\leq 2x_1, & x_1+x_2+x_3 &\leq 2x_2, & x_1+x_2+x_3 &\leq 2x_3,
    \end{align*}
    and at least one of them is an equality. Here, the first two yields $x_3\leq x_1-x_2$ and $x_3\leq x_2-x_1$, thus it is combined to $x_3\leq -|x_1-x_2|$. (This implies $x_3\leq 0$, and $x_1,x_2\leq 0$ are argued symmetrically.) The last yields $x_3\geq x_1+x_2$. But as at least one of inequalities must be an equality, we have two cases: $x_3=-|x_1-x_2|$ or $x_3=x_1+x_2$.
    
    Note that $x_3\leq 0$ in any case, and by symmetry, we have $x_1,x_2\leq 0$ too. If $x_3=-|x_1-x_2|$, then we have $\Phi(\pm(x_1,-x_2))=(x_1,x_2,x_3)$. If $x_3=x_1+x_2$, then we have $\Phi(\pm(x_1,x_2))=(x_1,x_2,x_3)$.
\end{proof}

\begin{lemma}
\label{lem:val--image-skeleton}
    Let $Sk=Sk(\infty,\infty,\infty,\val(4))\subset\Trop(S_{0004})$ be the skeleton of the Cayley cubic. The nonpositively rectified value defines a map $\val^-\colon S_{0004}(K)\to Sk$, $(X_1,X_2,X_3)\mapsto(\val^-(X_1),\val^-(X_2),\val^-(X_3))$.
\end{lemma}
\begin{proof}
    We know that $\ConeParam\circ\val=\val^-\circ \PreConeParam$~\eqref{eqn:cone-param-commutativity}, so it remains to see that the image of $\ConeParam$ equals to $Sk$, which was established in Lemma~\ref{lem:cone-parametrization}(b).
\end{proof}

\begin{lemma}
    \label{lem:cone-parametrization-is-equivariant}
    Let $\trop(s_i)$ be set via formulae~\eqref{eqn:trop-s1-holomorphic}--\eqref{eqn:trop-s3-holomorphic}. Let $j$ be the group map defined in Proposition~\ref{lem:PGL2Z-Vieta-group-identification}. For any $\vec{u}\in\R^2/\langle\pm1\rangle$ and $i=1,2,3$, we have
    \[\trop(s_i)(\ConeParam(\vec{u}))=\ConeParam(j(s_i).\vec{u}).\]
\end{lemma}
\begin{proof}
    Fix $(A,B,C,D)=(0,0,0,4)$ as a holomorphic parameter. Via Lemma~\ref{lem:PGL2Z-Aut-Cayley-identification}, recall the group $\widetilde{\Gamma}_{0004}=\langle\Gamma_{0004},\sigma_{12},\sigma_{23}\rangle$ extending $\Gamma_{0004}=\langle s_1,s_2,s_3\rangle$ by permutations, and that the group is generated by $s_3$, $\sigma_{12}$, and $\sigma_{23}$. If we choose the tropicalizations $\trop(\sigma_{12})$ and $\trop(\sigma_{23})$ of permutations as the same permutations, then we have
    \begin{align*}
    	\trop(s_2)&=\trop(\sigma_{23})\circ\trop(s_3)\circ\trop(\sigma_{23}), \\
    	\trop(s_1)&=\trop(\sigma_{12})\circ\trop(s_2)\circ\trop(\sigma_{12}),
    \end{align*}
    on $\R^3$, analogous to $s_2=\sigma_{23}s_3\sigma_{23}$ and $s_1=\sigma_{12}s_2\sigma_{12}$ on $S_{0004}$. So it suffices to check that $\ConeParam\circ\widetilde{\jmath}(\gamma)=\trop(\gamma)\circ\ConeParam$ for $\gamma=s_3,\sigma_{12},\sigma_{23}$.
    
    Let $\pm(u,v)\in\R^2/\langle\pm1\rangle$ be any. For the case $\gamma=s_3$, we compute
    \begin{align*}
    	\trop(s_3)(\ConeParam(\pm(u,v))) &= \trop(s_3)(-|u|,-|v|,-|u+v|)\\
    	&= (-|u|,-|v|,\min\{-2|u|,-2|v|\}+|u+v|) \\
    	&= (-|u|,-|v|,-|u-v|)\quad\text{(by Lemma~\ref{lem:l1-linfty-equivalence})} \\
    	&= \ConeParam(u,-v)=\ConeParam(\widetilde{\jmath}(s_3).(u,v)).
    \end{align*}
    For the case $\gamma=\sigma_{23}$, we compute
    \begin{align*}
    	\trop(\sigma_{23})(\ConeParam(\pm(u,v))) &= \trop(\sigma_{23})(-|u|,-|v|,-|u+v|) \\
    	&= (-|u|,-|u+v|,-|v|)=(-|u|,-|-u-v|,-|u+(-u-v)|) \\
    	&= \ConeParam(u,-u-v)=\ConeParam(\widetilde{\jmath}(\sigma_{23}).(u,v)).
    \end{align*}
    The case $\gamma=\sigma_{12}$ is clear. These show the claim.
\end{proof}

\begin{theorem}[Tropical dynamics for Cayley cubic]
    \label{thm:tropical-dynamics-cayley-cubic}
    Let $Sk=Sk(\infty,\infty,\infty,\val(4))$ be the skeleton in $\Trop(S_{0004})$. Let $\val^-(X)=\min\{0,\val(X)\}$ be the rectified value, and let $\PreConeParam$ and $\ConeParam$ be respectively from \eqref{eqn:pre-cone-parametrization} and \eqref{eqn:cone-parametrization}. Then we have the following commutative diagram:
    \begin{equation}
        \label{eqn:tropical-dynamics-cayley-cubic}
    \begin{tikzcd}
        (K^\times)^2/\langle\pm1\rangle \arrow[r,"\PreConeParam"]\arrow[d,"\val"] & S_{0004}(K) \arrow[d,"\val^-"] \\
        \R^2/\langle\pm 1\rangle \arrow[r,"\ConeParam"] & Sk.
    \end{tikzcd}
    \end{equation}
    Moreover, vertical arrows are equivariant in the sense that
    \begin{enumerate}[(i)]
        \item for any $A\in\mathsf{PGL}_2(\Z)$, we have $\val(A.(U,V)^{\pm 1})=A.\pm(\val(U),\val(V))$, where the left side is the monomial action and the right side is the linear (column-vector) action; and
        \item whenever $X\in S_{0004}(K)$ and $i=1,2,3$, we have $\val^-(s_i.X)=\trop(s_i)(\val^-(X))$.
    \end{enumerate}
\end{theorem}
\begin{proof}
    All are established except the equivariance (ii). As both $\PreConeParam$ and $\ConeParam$ are bijections, this follows from Lemma~\ref{lem:cone-parametrization-is-equivariant}.
\end{proof}

\subsection{The General Case} We finalize the section by the following ``master theorem'' of tropical dynamics of Markov surfaces with holomorphic parameters. 

\begin{theorem}[Tropical dynamics for holomorphic parameters]
\label{thm:tropical-dynamics-holomorphic-parameters}
    Let $A,B,C,D\in K$ be parameters with nonnegative values $a,b,c,d$ resp., and let $Sk=Sk(a,b,c,d)$ be the corresponding skeleton in $\Trop(S_{ABCD})$. Let $\val^-(X)=\min\{0,\val(X)\}$ be the rectified value, inducing a map $\val^-\colon S_{ABCD}(K)\to Sk$, $\val^-((X_i)_{i=1}^3)=(\val^-(X_i))_{i=1}^3$. Let $\phi\colon\R^2/\langle\pm1\rangle\xrightarrow{\sim}Sk$ be as in \eqref{eqn:cone-parametrization}. Let $j\colon\Gamma_{ABCD}\to\mathsf{PGL}_2(\Z)$ be the group map \eqref{eqn:PGL2Z-Vieta-group-identification}.
    
    The map $\ConeParam^{-1}\circ\val^-\colon S_{ABCD}(K)\to\R^2/\langle\pm1\rangle$ is equivariant in the sense that, for any $X\in S_{ABCD}(K)$ and $\gamma\in\Gamma_{ABCD}$, we have $\ConeParam^{-1}\circ\val^-(\gamma.X)=j(\gamma).\ConeParam^{-1}\circ\val^-(X)$.
\end{theorem}

By Lemma~\ref{lem:cone-parametrization-is-equivariant}, we have $j(g)\circ\ConeParam^{-1}=\ConeParam^{-1}\circ\trop(g)$, for $g=s_1,s_2,s_3$, even on general Markov surface with holomorphic parameters. Thus our theorem reduces to the following

\begin{lemma}[Tropical equivariance for holomoprhic parameters]
\label{lem:tropical-dynamics-holomorphic-parameters}
    We keep the notations from Theorem~\ref{thm:tropical-dynamics-holomorphic-parameters}. The map $\val^-\colon S_{ABCD}(K)\to Sk$ is equivariant in the sense that, for any $X\in S_{ABCD}(K)$ and $i=1,2,3$, we have
    \[\val^-(s_i.X)=\trop(s_i)(\val^-(X)).\]
\end{lemma}
\begin{proof}
   Let $X=(X_1,X_2,X_3)$ denote the coordinates of $X\in S_{ABCD}(K)$ and $x_i=\val(X_i)$ denote the value of each. Then we have $(x_1,x_2,x_3)\in\Trop(S_{ABCD})$. Be aware that we are using rectified values $\val^-(X_i)$ as well below.
   
   We split into the following cases.
   \begin{enumerate}
   \item\label{enum:nonnegative-valued} When $\val^-(X)=(0,0,0)$.
   \item\label{enum:fin-point} When $\val^-(X)\neq(0,0,0)$ yet one of $\val^-(X_j)$'s is zero.
   \item\label{enum:true-skeleton} When all $\val^-(X_j)$'s is negative.
   \end{enumerate}
   
   To deal with case (\ref{enum:nonnegative-valued}), let $\mathcal{O}=\{x\in K : \val(x)\geq 0\}$ be the ring of valuations. Then we have $A,B,C,D\in\mathcal{O}$. In particular, the Vieta involutions are polynomial maps over coefficients in $\mathcal{O}$, and hence $\Gamma_{ABCD}$ leaves $S_{ABCD}(\mathcal{O})$ invariant.
   
   If $\val^-(X)=(0,0,0)$, then $x_j\geq 0$ for $j=1,2,3$. Then $X\in S_{ABCD}(\mathcal{O})$, so for any Vieta involution $s_i$, we have $\val^-(s_i.X)=(0,0,0)=\trop(s_i)(0,0,0)=\trop(s_i)(\val^-(X))$.
   
   To deal with case (\ref{enum:fin-point}), we may assume that $\val^-(X_1)=0$ and $\val^-(X_2)<0$. We first claim that $x_3=x_2$. Assume $x_2\neq x_3$. Since $x_1\geq 0>x_2$, we know that, in the tropical polynomial $f_{abcd}$~\eqref{eqn:trop-polynomial}, the possible tropical monomials attaining the minimum are $2x_2$, $2x_3$, $x_1+x_2+x_3$, and $c+x_3$. (The excluded monomials are greater than $2x_2$.) If $x_3<x_2$ then we have $2x_3<2x_2$, $2x_3<x_2+x_3\leq x_1+x_2+x_3$, and $2x_3<c+x_3$, so the point is in the cell $\cell(X_3^2)$, contradicting to $(x_1,x_2,x_3)\in\Trop(S_{ABCD})$. If $x_3>x_2$, then likewise we have $2x_2<2x_3$, $2x_2<x_1+x_2+x_3$, and $2x_2<c+x_3$, so the point is in the cell $\cell(X_2^2)$, again a contradiction. Thus $x_3=x_2$ must hold.
   
   The rest is computation from $x_1\geq 0>x_2=x_3$. We have $\val^-(X)=(0,x_2,x_3)$ and
   \begin{align*}
       \min\{2x_2,2x_3,b+x_2,c+x_3,d\} &= 2x_2=x_2+x_3, \\
       \min\{2\cdot 0,2x_3,a+0,c+x_3,d\} &= 2x_3, \\
       \min\{2\cdot 0,2x_2,a+0,b+x_2,d\} &= 2x_2,
   \end{align*}
   which verify
   \begin{align*}
       \trop(s_1)(0,x_2,x_3) &= (x_2+x_3,x_2,x_3), \\
       \trop(s_2)(0,x_2,x_3) &= (0,2x_3-x_2,x_3)=(0,x_2,x_3), \\
       \trop(s_3)(0,x_2,x_3) &= (0,x_2,2x_2-x_3)=(0,x_2,x_3).
   \end{align*}
   Meanwhile, the followings have unique minimal monomial and evaluate
   \begin{align*}
       \val(-X_2X_3+A-X_1) &= \val(X_2X_3)=x_2+x_3<0, \\
       \val(-X_1X_3+B-X_2) &= \val(X_2)=x_2<0, \\
       \val(-X_1X_2+C-X_3) &= \val(X_3)=x_3<0,
   \end{align*}
   and we see that $\trop(s_i)(\val^-(X))=\val^-(s_i.X)$, for $i=1,2,3$, in this case.
   
   Finally, for case (\ref{enum:true-skeleton}), we have $x_j<0$ for $j=1,2,3$. Then we have $(x_1,x_2,x_3)=\val^-(X)\in Sk$. As the assumptions are all symmetric in this case, it suffices to let $i=1$ and show that $\val^-(s_1.X)=\trop(s_1)(x_1,x_2,x_3)$. Let $s_1.X=(\widetilde{X}_1,X_2,X_3)$.
   
   Suppose that $\val^-(\widetilde{X}_1)=0$. Then as discussed above, we must have $x_2=x_3<0$. Moreover, we have $\val(X_1)=\val(-X_2X_3+A-\widetilde{X}_1)=\val(X_2X_3)$ in that case, so $x_1=x_2+x_3$. As $(x_1,x_2,x_3)\in Sk$, it follows that
   \begin{align*}
       \trop(s_1)(\val^-(X)) &= \trop(s_1)(x_1,x_2,x_3) \\
       &= \left(\min\{2x_2,2x_3\}-x_1,x_2,x_3\right) \\
       &= \left(2x_2-(x_2+x_3),x_2,x_3)\right)=(0,x_2,x_3)=\val^-(s_1.X).
   \end{align*}
   
   Suppose that $\val^-(\widetilde{X}_1)<0$. If $x_2=x_3(<0)$, then as $(x_1,x_2,x_3)\in Sk$ verify the equation~\eqref{eqn:image-cone-parametrization}, it follows that $x_1=2x_2$. Meanwhile, as
   \begin{align*}
       \val(X_1\widetilde{X}_1)=\val(X_2^2+X_3^2-BX_2-CX_3-D) &\geq \val(X_2^2)=2x_2,
   \end{align*}
   we have $\val(\widetilde{X}_1)\geq 2x_2-x_1=0$. This contradicts to the assumption. So $x_2\neq x_3$, and $\min\{2x_2,2x_3,b+x_2,c+x_3,d\}$ has a unique tropical monomial attaining the minimum, $2x_2$ or $2x_3$. Hence $(x_1,x_2,x_3)$ is in the tropical domain of $s_1$, and we have
   \[\val^-(s_1.X)=\val(s_1.X)=\trop(s_1)(\val(X))=\trop(s_1)(\val^-(X)),\]
   as demanded.
\end{proof}


\section{The Topology of the Skeleton}
\label{sec:skeleton}

To address the tropical dynamics on the skeleton in the meromorphic case, we first show that the topological structure of the skeletons and Vieta involutions are (piecewise-linear) homeomorphic from one another. As first steps, in this section we prove that
\begin{enumerate}
	\item (Proposition~\ref{lem:trop-foliation}) the skeleton $Sk(a,b,c,d)$ is piecewise-linear homeomorphic to the plane $\R^2$, and
	\item (Proposition~\ref{lem:fixed-set-transversality}) the fixed loci of $\trop(s_i)$'s \emph{in $\R^3$} are transversal to the skeleton, and on the skeleton, each fixed locus splits the skeleton into two pieces that are swapped by $\trop(s_i)$.
\end{enumerate}

\subsection{The Equation for the Skeleton}

Recall that we have characterized the skeleton by an equation~\eqref{eqn:skeleton-equation}. This equation may be seen as the zero level set of the following function
\begin{equation}
\label{eqn:skeletal-trop-laurent-poly}
    f_{Sk,abcd}(x_1,x_2,x_3)=\min\left\{\begin{array}{l}2x_1,2x_2,2x_3, \\ a+x_1,b+x_2,c+x_3,d \end{array}\right\}-(x_1+x_2+x_3).
\end{equation}
We name $f_{Sk,abcd}$ the \emph{skeletal tropical Laurent polynomial}. When $a,b,c,d$ are complicated, we may distinguish them as $f_{Sk;a,b,c,d}$.

\subsubsection{Invariance}

We note that this tropical Laurent polynomial is invariant under the tropicalized group $\trop(\Gamma_{ABCD})$.

First observation is this. If $w$ is in the value group $\val(K^\times)$, then the level set $f_{Sk,abcd}=w$ is a shift of some other skeleton, in the following sense.
\begin{lemma}
\label{lem:skeleton-shift}
	Fix $t^w\in K^\times$ with $\val(t^w)=w$. Consider the surface
	\[S^w_{ABCD}\colon X_1^2+X_2^2+X_3^2+t^wX_1X_2X_3=AX_1+BX_2+CX_3+D.\]
	Then the map $\Phi\colon S^w_{ABCD}\to\Aa^3$, $(X_1,X_2,X_3)\mapsto(t^{w}X_1,t^{w}X_2,t^{w}X_3)$, has the image $S_{A'B'C'D'}$ with $(A',B',C',D')=(t^wA,t^wB,t^wC,t^{2w}D)$. Moreover, we send $\{f_{Sk,abcd}=w\}$ to $Sk(a+w,b+w,c+w,d+2w)$ by a translation $\phi(x)=x+(w,w,w)$.
\end{lemma}
\begin{proof}
	The first claim is a matter of computation. For the second, note that
	\[f_{Sk,abcd}(\phi(x))=f_{Sk,abcd}(x+(w,w,w))=f_{Sk;a+w,b+w,c+w,d+2w}(x)-w.\]
	Hence $Sk(a+w,b+w,c+w,d+2w)=\{f_{Sk,a+w,b+w,c+w,d+2w}=0\}=\{f_{Sk,abcd}(\phi(x))=w\}=\phi^{-1}(\{f_{Sk,abcd}=0\})$ shows the claim.
\end{proof}

The shifting map conjugates the tropicalized Vieta involutions as well:

\begin{lemma}
\label{lem:shift-conjugation}
	Consider the following tropicalized Vieta involutions, on $S_{ABCD}$ and $S_{A'B'C'D'}$ with $(A',B',C',D')=(t^wA,t^wB,t^wC,t^{2w}D)$:
	\begin{align*}
		\trop(s_1)(x) &= \left(\min\{2x_2,2x_3,b+x_2,c+x_3,d\}-x_1,x_2,x_3\right),\quad\text{ etc.,} \\
		\trop(s_1')(x) &= \left(\min\{2x_2,2x_3,b+w+x_2,c+w+x_3,d+2w\}-x_1,x_2,x_3\right),\quad\text{ etc.}
	\end{align*}
	Let $\phi(x)=x+(w,w,w)$ be a translation. Then we have the followings, for any $x\in\R^3$: $\trop(s_1')(\phi(x))=\phi(\trop(s_1)(x))$, etc.
\end{lemma}
\begin{proof} Computation. \end{proof}

\begin{corollary}
	The function $f_{Sk,abcd}$ is invariant under $\trop(\Gamma_{ABCD})$.
\end{corollary}
\begin{proof}
	Note that, as $K$ is algebraically closed, the value group $\val(K^\times)$ is a $\Q$-vector space and thus dense. Hence it suffices to show that, for $w\in\val(K^\times)$, the level set $\{f_{Sk,abcd}=w\}$ is $\trop(\Gamma_{ABCD})$-invariant, as $f_{Sk,abcd}$ is continuous on $\R^3$.

	By the shift map $\phi\colon\{f_{Sk,abcd}=w\}\to Sk(a+w,b+w,c+w,d+2w)$, $\phi(x)=x+(w,w,w)$, we can evaluate, for $x\in\{f_{Sk,abcd}=w\}$, the tropicalized Vieta involution $\trop(s_i)(x)$ as follows (here, $\trop(s_i')$ is as in Lemma~\ref{lem:shift-conjugation}):
	\begin{align*}
		\trop(s_i)(x) &= \phi^{-1}\phi\left(\trop(s_i)(x)\right) \\
		&= \phi^{-1}\trop(s_i')(\phi(x)) \\
		&\in\phi^{-1}(Sk(a+w,b+w,c+w,d+2w))=\{f_{Sk,abcd}=w\}.
	\end{align*}
	This shows that the level set is invariant.
\end{proof}

It is known that, for nontrivially valued fields, one can extend $(K,\val)$ to have the value group $\R$ (see \cite[Sec.~5, Thm.~2]{Poonen93}, with $G=\R$). So the constraint $w\in\val(K^\times)$ can be seen to be vacuous in discussions above. In particular, any discussions on the dynamics of the level set $\{f_{Sk,abcd}=w\}$ can be reduced to that on the skeleton $Sk(a+w,b+w,c+w,d+2w)$, regardless to whether $w$ is in the value group or not.

\subsubsection{Inequality formulation}

We note the following additional fact for the values of $f_{Sk,abcd}$, as a computational gadget. The proof is elementary and thus omitted.
\begin{lemma}
\label{lem:inequality-formulation-tropical-skeletal-Laurent}
    We have $f_{Sk,abcd}(x_1,x_2,x_3)=w$ if and only if the following inequalities hold,
    \begin{align*}
        w+x_1+x_2+x_3 &\leq 2x_1, & w+x_1+x_2+x_3 &\leq a+x_1, \\
        w+x_1+x_2+x_3 &\leq 2x_2, & w+x_1+x_2+x_3 &\leq b+x_2, \\
        w+x_1+x_2+x_3 &\leq 2x_3, & w+x_1+x_2+x_3 &\leq c+x_3, \\
        && w+x_1+x_2+x_3 &\leq d.
    \end{align*}
    with at least one holding as an equality.
\end{lemma}

\subsection{Topology of Level Sets}

The skeletal tropical Laurent polynomial $f_{Sk,abcd}$ discussed above has planar level sets, in the following sense.

\begin{proposition}
    \label{lem:trop-foliation}
    Let $\Pi=\{(x_1,x_2,x_3)\in\R^3 : x_1+x_2+x_3=0\}$ be a plane in $\R^3$, and $v\colon\R^3\to\Pi$ be the orthogonal projection onto that plane, i.e.,
    \[v=(v_1,v_2,v_3)=\left(\frac{2x_1-x_2-x_3}{3},\frac{-x_1+2x_2-x_3}{3},\frac{-x_1-x_2+2x_3}{3}\right).\]
    Then the product map $f_{Sk,abcd}\times v\colon\R^3\to\R\times\Pi$ is a piecewise-linear homeomorphism.
\end{proposition}
\begin{proof}
    Let $\alpha=\frac13(x_1+x_2+x_3)$. Then the map $\alpha\times v\colon\R^3\to\R\times\Pi$ is a linear isomorphism, with the inverse $x=(x_1,x_2,x_3)=(\alpha+v_1,\alpha+v_2,\alpha+v_3)$. Let $f_0=f_{Sk,abcd}$.

    We describe the inverse $(f_0\times v)^{-1}$ by describing $(\alpha\times v)\circ(f_0\times v)^{-1}$ instead. That is, we suggest on how to recover $\alpha$ from $w\in\R$ and $v\in\Pi$ given. Apply Lemma~\ref{lem:inequality-formulation-tropical-skeletal-Laurent} with $(x_1,x_2,x_3)=(\alpha+v_1,\alpha+v_2,\alpha+v_3)$ and get
    \begin{align*}
        w + 3\alpha &\leq 2\alpha + 2v_1, & w + 3\alpha &\leq a+\alpha + v_1, \\
        w + 3\alpha &\leq 2\alpha + 2v_2, & w + 3\alpha &\leq b+\alpha + v_2, \\
        w + 3\alpha &\leq 2\alpha + 2v_3, & w + 3\alpha &\leq c+\alpha + v_3, \\
        & & w + 3\alpha &\leq d,
    \end{align*}
    with at least one of them holding as an equality. 
    Solving these inequalities in $\alpha$, we obtain
    \begin{align*}
        \alpha &\leq 2v_1 - w, & \alpha &\leq \frac12(a + v_1 - w), \\
        \alpha &\leq 2v_2 - w, & \alpha &\leq \frac12(b + v_2 - w), \\
        \alpha &\leq 2v_3 - w, & \alpha &\leq \frac12(c + v_3 - w), \\
        && \alpha &\leq \frac13(d - w),
    \end{align*}
    with at least one of them holding as an equality. This is equivalent to
    \begin{equation}
        \label{eqn:average-coordinate-formula}
        \alpha = \min\left\{\begin{array}{l} 2v_1-w, 2v_2-w, 2v_3-w, \\ \frac12(a+v_1-w),\frac12(b+v_2-w),\frac12(c+v_3-w), \\ \frac13(d-w) \end{array}\right\},
    \end{equation}
    so we can find the inverse map explicitly.
\end{proof}

\begin{corollary}
    \label{lem:skeleton-plane-identification}
    The orthogonal projection map $v\colon\R^3\to\Pi$ restricted to a skeleton $v\restriction Sk(a,b,c,d)\to\Pi$ is a piecewise-linear homeomorphism.
\end{corollary}

\subsection{Transversality of Fixed Sets}

By the definitions \eqref{eqn:trop-vieta-1} to \eqref{eqn:trop-vieta-3} as a map $\R^3\to\R^3$, the fixed sets of $\trop(s_1),\trop(s_2),\trop(s_3)$ appear as
\begin{align*}
    \Fix(\trop(s_1)) &\colon 2x_1 = \min\{2x_2,2x_3,b+x_2,c+x_3,d\}, \\
    \Fix(\trop(s_2)) &\colon 2x_2 = \min\{2x_1,2x_3,a+x_1,c+x_3,d\}, \\
    \Fix(\trop(s_3)) &\colon 2x_3 = \min\{2x_1,2x_2,a+x_1,b+x_2,d\}.
\end{align*}
Hence these fixed sets appear as graphs of some piecewise-linear functions $\R^2\to\R$, thus piecewise-linear homeomorphic to $\R^2$.

It turns out that the fixed sets are `transverse' to each level set $\{f_{Sk,abcd}=w\}$ and intersects on a topological line, as claimed in the Proposition below. Recall the foliation $f_{Sk,abcd}\times v\colon\R^3\to\R\times\Pi$ in Proposition \ref{lem:trop-foliation} above. 

\begin{proposition}
    \label{lem:fixed-set-transversality}
    There is a coordinate $(w;v_1,u_1)$ of $\R\times\Pi$ and a piecewise-linear continuous function $G\colon\R^2\to\R$ with following properties.
    \begin{enumerate}[(a)]
    \item The fixed set $\Fix(\trop(s_1))\subset\R^3$ appears as the graph $v_1=G(w,u_1)$.
    \item The map $\trop(s_1)$ sends $(w;v_1,u_1)\mapsto(w;2G(w,u_1)-v_1,u_1)$, and hence the fixed locus of $\trop(s_1)$ splits the plane $w=\mathsf{const}$ into two pieces that are swapped by the map.
    \end{enumerate}
\end{proposition}
\begin{proof}
    Specifically, we set $w=f_{Sk,abcd}(x_1,x_2,x_3)$, $v_1=\frac13(2x_1-x_2-x_3)$, and $u_1=v_2-v_3=x_2-x_3$. Then $(w;v_1,u_1)$ forms a coordinate of $\R\times\Pi$. Observe that $w$ and $u_1$ are invariant under $\trop(s_1)$. So $\trop(s_1)$ acts as an involution on $v_1$, and by the intermediate value theorem, there is a unique fixed point $v_1=G(w,u_1)$. Note that $\trop(s_1)$ sends $v_1$ to
    \[v_1\mapsto\frac23(f_1(x_2,x_3)-x_2-x_3)-v_1,\]
    where $f_1(x_2,x_3)=\min\{2x_2,2x_3,b+x_2,c+x_3,d\}$.
    
    So it remains to express $\frac13(f_1(x_2,x_3)-x_2-x_3)$ as a function $G(w,u_1)$, or equivalently, find the coordinate $v_1$ of the fixed point. This is done by the following steps.
    \begin{enumerate}
    	\item Evaluate
    	\[\tilde{v}_1=\min\left\{\begin{array}{l}u_1-2w,-u_1-2w, \\ \frac13(2b-4w+u_1),\frac13(2c-4w-u_1),\\ \frac12d-w,a-w\end{array}\right\}.\]
    	\item Evaluate $x_2=\frac12(\tilde{v}_1+u_1)$ and $x_3=\frac12(\tilde{v}_1-u_1)$.
    	\item Evaluate $x_1=\min\{x_2,x_3,\frac{b+x_2}{2},\frac{c+x_3}{2},\frac{d}{2}\}$.
    	\item Evaluate $v_1=\frac13(2x_1-x_2-x_3)$, to finish.
    \end{enumerate}
    As every step is piecewise-linear and continuous, so is $G$, the end result.
    
    To show why such $G$ is working, we first remark some properties of points in $\Fix(\trop(s_1))$. 
    With $f_1(x_2,x_3)$ above, we have $f_{Sk,abcd}(x_1,x_2,x_3)=\min\{2x_1,a+x_1,f_1(x_2,x_3)\}-(x_1+x_2+x_3)$. If $(x_1,x_2,x_3)\in\Fix(\trop(s_1))$, then $f_1(x_2,x_3)=2x_1$, so we have
    \begin{align*}
        f_{Sk,abcd}(x_1,x_2,x_3) &= \min\{2x_1,a+x_1,2x_1\}-(x_1+x_2+x_3) \\
        &= \min\{x_1,a\} - (x_2+x_3).
    \end{align*}
    Because $w=f_{Sk,abcd}(x_1,x_2,x_3)$, it follows that $w+x_2+x_3=\min(x_1,a)$. Together with $x_1=\frac12f_1(x_2,x_3)$, we have
    \begin{align}
        x_1 &= \min\{x_2,x_3,\frac{b+x_2}{2},\frac{c+x_3}{2},\frac{d}{2}\}, \label{eqn:fixed-set-equation-1} \\
        w+x_2+x_3 &= \min(x_1,a). \label{eqn:fixed-set-equation-2}
    \end{align}
    Plug in \eqref{eqn:fixed-set-equation-1} into \eqref{eqn:fixed-set-equation-2}. Then we have
    \[w+x_2+x_3=\min\{x_2,x_3,\frac{b+x_2}{2},\frac{c+x_3}{2},\frac{d}{2},a\}.\]
    Set $\tilde{v}_1=x_2+x_3$ and multiply 2 on both sides. With $u_1=x_2-x_3$, we simplify
    \[2\tilde{v}_1+2w=\min\{\tilde{v}_1+u_1,\tilde{v}_1-u_1,b+\frac12(\tilde{v}_1+u_1),c+\frac12(\tilde{v}-u_1),d,2a\}.\]
    This equality can be expanded as a list of inequalities, with at least one of them holding as an equality. Solving the inequalities in $\tilde{v}_1$, one then get an equation of $\tilde{v}_1$ in $(w,u_1)$, which is the first step of the algorithm.
    
    As $\tilde{v}_1=x_2+x_3$, we have $x_2=\frac12(\tilde{v}_1+u_1)$ and $x_3=\frac12(\tilde{v}_1-u_1)$ following. We recover $x_1$ by \eqref{eqn:fixed-set-equation-1}, and $v_1$ by $v_1=\frac13(2x_1-x_2-x_3)$.
\end{proof}

Furthermore, two fixed loci interesect on a half-plane, and they intersect each level set $\{f_{Sk,abcd}=w\}$ on a half-ray.

\begin{proposition}
   \label{lem:intersection-of-fixed-sets}
   \begin{enumerate}[(a)]
       \item The intersection $\Fix(\trop(s_1))\cap\Fix(\trop(s_2))$ in $\R^3$ is topologically a copy of the half-plane, $\R\times[0,\infty)$.
       \item Moreover, there is a coordinate $(w;v_3,u_3)$ of $\R\times\Pi$ and a piecewise-linear continuous function $G\colon\R\to\R$ such that the intersection appears as $u_3=0$ and $v_3\geq-G(w)$.
       \item The intersection of all, $\bigcap_{i=1}^3\Fix(\trop(s_i))$, in $\R^3$ is topologically a half-ray that only intersects those level sets $\{f_{Sk,abcd}=w\}$ with $w\geq-\min\{a,b,c,\frac12d\}$.
   \end{enumerate}
\end{proposition}
\begin{proof}
   (a) The fixed loci $\Fix(\trop(s_1))$ and $\Fix(\trop(s_2))$ are characterized by the following equations:
   \begin{align}
       2x_1 &= \min\{2x_2,2x_3,b+x_2,c+x_3,d\}, \label{eqn:fix-trop(s1)} \\
       2x_2 &= \min\{2x_1,2x_3,a+x_1,c+x_3,d\}. \label{eqn:fix-trop(s2)}
   \end{align}
   If we assume both \eqref{eqn:fix-trop(s1)} and \eqref{eqn:fix-trop(s2)}, then we have $x_1=x_2$, as \eqref{eqn:fix-trop(s1)} yields $2x_1\leq 2x_2$ and, likewise, \eqref{eqn:fix-trop(s2)} yields $2x_2\leq 2x_1$. Moreover, combining the minimums in \eqref{eqn:fix-trop(s1)} and \eqref{eqn:fix-trop(s2)}, we have
   \begin{equation}
   \label{eqn:intersection-fix-trop(s1)-trop(s2)}
       2x_1=2x_2\leq\min\{2x_3,a+x_1,b+x_2,c+x_3,d\}.
   \end{equation}
   This inequality~\eqref{eqn:intersection-fix-trop(s1)-trop(s2)} can be expanded as a list of inequalities. Solving these inequalities in $x_1$, we have
   \[x_2=x_1\leq\min\{x_3,a,b,\frac{c+x_3}2,\frac{d}2\}.\]
   Hence, on the $(x_3,x_1)$-plane, the intersection of fixed loci appears as a hypograph of a continuous function in $x_3$. Thus the conclusion.
   
   (b) We set $w=f_{Sk,abcd}$, $v_3=\frac13(-x_1-x_2+2x_3)$, and $u_3=v_1-v_2=x_1-x_2$. We set $G(w)=\min\{0,\frac13(w+c),\frac23(w+a),\frac23(w+b),\frac23(w+\frac12d)\}$. All what remains is that why these choices are working.
   
   From \eqref{eqn:intersection-fix-trop(s1)-trop(s2)}, we have not only $u_3=0$, but also
   \begin{align*}
       f_{Sk,abcd}(x_1,x_2,x_3) &= \min\{2x_1,2x_2\}-(x_1+x_2+x_3) \\
       &= 2x_1-(2x_1+x_3)=-x_3.
   \end{align*}
   So on the intersection, we have $w=-x_3$, and $v_3=\frac13(-2x_1-2w)$ follows to yield $x_1=x_2=-\frac32v_3-w$. If we plug these into \eqref{eqn:intersection-fix-trop(s1)-trop(s2)} back, then we have
   \[-3v_3-2w\leq\min\{-2w,a-\frac32v_3-w,b-\frac32v_3-w,c-w,d\}.\]
   Expanding this to a list of inequalities and solving them in $-v_3$, we have
   \[-v_3\leq\min\{0,\frac23(a+w),\frac23(b+w),\frac13(c+w),\frac13(2w+d)\}=G(w).\]
   Thus $u_3=0$ and $v_3\geq -G(w)$ represents the intersection.
   
   (c) If we add $2x_3=\min\{2x_1,2x_2,a+x_1,b+x_2,d\}$ to \eqref{eqn:intersection-fix-trop(s1)-trop(s2)}, we have $x_1=x_2=x_3\leq\min\{a,b,c,\frac12d\}$, which is the equation of the total intersection. The equation reveals that it is a half-ray, and on that half-ray, we evaluate $f_{Sk,abcd}(x_1,x_2,x_3)=-x_3\geq-\min\{a,b,c,\frac12d\}$. Hence only the level sets $\{f_{Sk,abcd}=w\}$ with $w\geq-\min\{a,b,c,\frac12d\}$ can find the total intersection point.
\end{proof}

Topologically speaking, each $\trop(s_i)$ acts on $\{f_{Sk,abcd}=w\}\cong\Pi$ as a proper reflection on the plane, and the fixed loci of two intersect on a half-ray. On the next section, we discuss how this constructs a dynamical structure of $\trop(s_i)$'s acting on the skeleton.

The geometric shape of skeleton and fixed loci may go complicated, as sketched in Figure~\ref{fig:typical-skeleton-fixed-loci}. This convinces that a topological approach would rather be effective here.

\begin{figure}
    \centering
		\tdplotsetmaincoords{{90-.5*acos(1/3)}}{135} 
		\begin{tikzpicture}[scale={(3/2)^(1/2)},tdplot_main_coords]
			\foreach \j in {-3}{ 
            \draw[-latex,dashed,gray] (\j,0,0) -- (.5,0,0);
            \draw[-latex,dashed,gray] (0,\j,0) -- (0,.5,0);
            \draw[-latex,dashed,gray] (0,0,\j) -- (0,0,.5);

            \foreach \i/\k in {gray/.5}{
            \fill[\i,opacity=\k] (-1.1,-.4,-.9) 
            -- (-1.1,-1,-.3) 
            -- (-.5,-1,-.9) 
            -- (-1.1,-.4,-.9) 
            ;
            \fill[\i,opacity=\k] (-1.3,-1.2,-.1) 
            -- (-1.2,-1,-.3) 
			-- (-1.2,-.4,-.9) 
			-- (-1.3,-.2,-1.1) 
            -- (-1.1,-.4,-.9) 
            -- (-1.1,-1,-.3) 
            -- (-1.3,-1.2,-.1) 
            ;
			\fill[\i,opacity=\k] (-.1,-1.4,-1.3) 
            -- (-.5,-1.2,-.9) 
			-- (-1.1,-1.2,-.3) 
			-- (-1.3,-1.3,-.1) 
			-- (-1.4,-1.4,0) 
            -- (-1.3,-1.2,-.1) 
            -- (-1.1,-1,-.3) 
            -- (-.5,-1,-.9) 
            -- (-.1,-1.4,-1.3) 
            ;
            \fill[\i,opacity=\k] (-1.5,0,-1.5) 
			-- (-1.3,-.2,-1.3) 
			-- (-1.1,-.4,-1.2) 
			-- (-.5,-1,-1.2) 
			-- (0,-1.5,-1.5) 
			-- (-.1,-1.4,-1.3) 
            -- (-.5,-1,-.9) 
            -- (-1.1,-.4,-.9) 
            -- (-1.3,-.2,-1.1) 
            -- (-1.5,0,-1.5) 
			;
            }


            \draw[gray] (0,\j,\j) -- (\j,0,\j) -- (\j,\j,0) -- (0,\j,\j);

            \foreach \i in {black!75!white}{
			\draw[\i] (-.5,-1,-.9) 
			-- (-1.1,-.4,-.9) 
			-- (-1.3,-.2,-1.1) 
			-- (-1.5,0,-1.5)  
			-- (0,-1.5,-1.5) 
			-- (-.1,-1.4,-1.3) 
			-- (-.5,-1,-.9);
			\draw[\i] (-.1,-1.4,-1.3) 
			-- (-1.4,-1.4,0) 
			-- (-1.3,-1.2,-.1) 
			-- (-1.1,-1,-.3) 
			-- (-.5,-1,-.9); 
			\draw[\i] (-1.1,-1,-.3) 
			-- (-1.1,-.4,-.9); 
			\draw[\i] (-1.3,-1.2,-.1) 
			-- (-1.3,-.2,-1.1); 
            }

			\draw[red] (\j,\j,0) 
			-- (-1.4,-1.4,0) 
			-- (-1.3,-1.2,-.1) 
			-- (-1.2,-1,-.3) 
			-- (-1.2,-.4,-.9) 
			-- (-1.3,-.2,-1.1) 
			-- (-1.5,0,-1.5) 
			-- (\j,0,\j); 
			\draw[green!50!black] (0,\j,\j) 
			-- (0,-1.5,-1.5) 
			-- (-.1,-1.4,-1.3) 
			-- (-.5,-1.2,-.9) 
			-- (-1.1,-1.2,-.3) 
			-- (-1.3,-1.3,-.1) 
			-- (-1.4,-1.4,0) 
			-- (\j,\j,0); 
			\draw[blue] (\j,0,\j) 
			-- (-1.5,0,-1.5)  
			-- (-1.3,-.2,-1.3) 
			-- (-1.1,-.4,-1.2) 
			-- (-.5,-1,-1.2) 
			-- (0,-1.5,-1.5) 
			-- (0,\j,\j); 
			}
		\end{tikzpicture}
    \caption{Skeleton with $(a,b,c,d)=(-1.3,-1.4,-1.5,-2.4)$ (white with gray boundary), fixed loci $\Fix(\trop(s_i))$ (red/green/blue by $i=1,2,3$), and the bounded region surrounded by these fixed loci (gray).}\label{fig:typical-skeleton-fixed-loci}
\end{figure}

\section{Alignment of Fixed Loci}
\label{sec:alignment-fixed-loci}

The fixed loci of $\trop(s_i)$'s split the skeleton into pieces that is suitable to assess the tropical action of $\Gamma_{ABCD}$ by a ping-pong lemma. Proposition~\ref{lem:cayley-domain-description}, the main goal of this section, is aimed to describe this split in details.

\subsection{The Goal}

Recall the adjacency loci $\partial\cell(c_\alpha X^\alpha\cap c_\beta X^\beta)$~\eqref{eqn:adjacency-locus}. Let $m$ be a monomial, $X_1^2$, $X_2^2$, $X_3^2$, $AX_1$, $BX_2$, $CX_3$, or $D$. Denote by $Sk(m)$ the adjacency locus $\partial\cell(X_1X_2X_3\cap m)$, which is a subset of the skeleton $Sk(a,b,c,d)$. The \emph{interior} of each $Sk(m)$ is the interior defined in the skeleton.

With this, the main point of this section is the following
\begin{proposition}
    \label{lem:cayley-domain-description}
    Write $A_1,A_2,A_3$ for $A,B,C$ respectively.
    There exist topological (open) half-spaces $D_1,D_2,D_3\subset Sk(a,b,c,d)$ such that for $i=1,2,3$,
    \begin{enumerate}[(i)]
	    \item $\trop(s_i)$ sends $D_i$ onto $Sk(a,b,c,d)\setminus\overline{D}_i$,
	    \item the boundary $\partial D_i$ is same as the fixed set of $\trop(s_i)$, and
	    \item each $D_i$ is a subset of $Sk(X_i^2)\cup Sk(A_iX_i)$, and contains the interior of $Sk(X_i^2)$.
    \end{enumerate}

    Furthermore,
    \begin{enumerate}[(i)]
        \addtocounter{enumi}{3}
        \item the domains $D_1,D_2,D_3$ are mutually disjoint; 
        \item the intersection of boundaries $\partial D_i\cap\partial D_j$ is a half-ray in $\partial D_i\cong\R$, and different $j\neq i$'s point different directions; and
        \item the closures $\overline{D}_i$'s union to give all of $Sk(a,b,c,d)$ if and only if $\min\{a,b,c,d\}\geq 0$.
    \end{enumerate}
\end{proposition}

Although each $D_i$'s can be expressed algebraically (for instance, $D_1=\{v_1<G(0,u_1)\}$, borrowing notations from Proposition~\ref{lem:fixed-set-transversality}), these algebraic expressions do not seem to help verifying the aboves. We rather study on how the adjacency loci $Sk(m)$ interact with tropicalized Vieta involutions, and construct $D_i$'s accordingly.

\subsection{Adjacency Loci in the Skeleton}

Recall the skeletal tropical Laurent polynomial
\[f_{Sk,abcd}(x_1,x_2,x_3)=\min\left\{\begin{array}{l} 2x_1, 2x_2, 2x_3 \\ a+x_1,b+x_2,c+x_3,d\end{array}\right\}-(x_1+x_2+x_3).\]
Fix a point $x$ in the skeleton $Sk(a,b,c,d)$. Depending on which tropical monomial this function chooses to evaluate, one can determine which adjacency locus $Sk(m)$ the point lies on. To be specific, we have, for any non-cubic monomial $m$,
\begin{equation}
\label{eqn:membership-adjacency-locus}
f_{Sk,abcd}(x)=\trop(m)(x)-(x_1+x_2+x_3)\Longleftrightarrow x\in Sk(m).
\end{equation}
So we have $f_{Sk,abcd}(x)=(c+x_3)-(x_1+x_2+x_3)$ if and only if $x\in Sk(CX_3)$, etc. 

By a \emph{quadratic (skeletal) adjacency locus}, we mean one of $Sk(X_1^2)$, $Sk(X_2^2)$, or $Sk(X_3^2)$. By a \emph{subquadratic (skeletal) adjacency locus}, we mean one of $Sk(AX_1)$, $Sk(BX_2)$, $Sk(CX_3)$, or $Sk(D)$.

\subsection{Tropical Vieta Involutions and Adjacency Loci}

If we look into the proof of invariance of skeleton (Lemma~\ref{lem:skeleton-invariance}) carefully, we observe the followings. Let $f_1(x_2,x_3)=\min\{2x_2,2x_3,b+x_2,c+x_3,d\}$. Then we have $f_{Sk,abcd}(x_1,x_2,x_3)=\min\{2x_1,a+x_1,f_1\}-(x_1+x_2+x_3)$, and the following equivalences:
\begin{align*}
\min\{2(f_1-x_1),a+(f_1-x_1),f_1\}&=2(f_1-x_1) &\Leftrightarrow \min\{2x_1,a+x_1,f_1\}&=f_1; \\
\min\{2(f_1-x_1),a+(f_1-x_1),f_1\}&=a+(f_1-x_1) &\Leftrightarrow \min\{2x_1,a+x_1,f_1\}&=a+x_1; \\
\min\{2(f_1-x_1),a+(f_1-x_1),f_1\}&=f_1 &\Leftrightarrow \min\{2x_1,a+x_1,f_1\}&=2x_1.
\end{align*}
If we interpret this in terms of membership in adjacency loci, we reduce to the following. Let $Sk(f_1)=\bigcup\left\{Sk(m) : m\in\{X_2^2,X_3^2,BX_2,CX_3,D\}\right\}$.
\begin{align*}
(f_1-x_1,x_2,x_3) &\in Sk(X_1^2) &\Leftrightarrow (x_1,x_2,x_3)&\in Sk(f_1), \\
(f_1-x_1,x_2,x_3) &\in Sk(AX_1) &\Leftrightarrow (x_1,x_2,x_3)&\in Sk(AX_1), \\
(f_1-x_1,x_2,x_3) &\in Sk(f_1) &\Leftrightarrow (x_1,x_2,x_3)&\in Sk(X_1^2).
\end{align*}
Recall that $\trop(s_1)(x_1,x_2,x_3)=(f_1-x_1,x_2,x_3)$. These list of ``inversions'' of adjacency loci can be phrased as a statement of how $\trop(s_1)$ act on these loci.
\begin{proposition}
    \label{lem:reflection-cells-atlas}
    The tropicalized Vieta involution $\trop(s_1)$ leaves $Sk(AX_1)$ invariant, and sends $Sk(X_1^2)$ onto the union $Sk(X_2^2)\cup Sk(X_3^2)\cup Sk(BX_2)\cup Sk(CX_3)\cup Sk(D)$.
\end{proposition}

Symmetric versions hold as well: so $\trop(s_2)$ leaves $Sk(BX_2)$ invariant, $\trop(s_3)$ sends $Sk(X_1^2)\cup Sk(X_2^2)\cup Sk(AX_1)\cup Sk(BX_2)\cup Sk(D)$ to $Sk(X_3^2)$, etc.

\subsection{Proof of Proposition~\ref{lem:cayley-domain-description}}

Fix $i\in\{1,2,3\}$. We know that, from Proposition~\ref{lem:fixed-set-transversality}(b), that the complement of a fixed locus $\Fix(\trop(s_i))$ consists of two components, that are swapped by $\trop(s_i)$. By Proposition~\ref{lem:reflection-cells-atlas}, we know that
\begin{enumerate}
    \item the union of $Sk(X_i^2)\cup Sk(A_iX_i)$ and its $\trop(s_i)$-image covers all of the skeleton, and
    \item the interior of $Sk(X_i^2)$ is mapped into its complement, by $\trop(s_i)$.
\end{enumerate}
So we can choose a component $D_i$ that contains the interior of $Sk(X_i^2)$, and such a component must be a subset of $Sk(X_i^2)\cup Sk(A_iX_i)$. This constructs the domains $D_i\subset Sk(a,b,c,d)$, and properties (i)--(iii) are immediate.

To show property (iv), note that
\begin{align*}
    \overline{D}_1 &\subset Sk(X_1^2)\cup Sk(AX_1), &
    \overline{D}_2 &\subset Sk(X_2^2)\cup Sk(BX_2), &
    \overline{D}_3 &\subset Sk(X_3^2)\cup Sk(CX_3).
\end{align*}
As none of the adjacency loci are repeating in above inclusions, we see that any intersection $\overline{D}_i\cap\overline{D}_j$ is nowhere dense. So its interior, $D_i\cap D_j$, must be empty.

To show property (v), that $\partial D_i\cap\partial D_j$ is a ray in $\partial D_i$ is claimed in Proposition~\ref{lem:intersection-of-fixed-sets}(b). To see that they are distinct with distinct $j$'s, let $k\neq i,j$ and consider $L_j=\partial D_i\cap\partial D_j$ and $L_k=\partial D_i\cap\partial D_k$. Then $L_j\cap L_k$ is at most a point, by Proposition~\ref{lem:intersection-of-fixed-sets}(c), whence two cannot point to the same directions.

To show property (vi), note that from Corollary~\ref{lem:cubic-cell-adj-subquadratic-iff-meromorphic}, we have a nonempty subquadratic adjacency locus if and only if $\min\{a,b,c,d\}<0$. By the contrapositive, we have $Sk(a,b,c,d)=\bigcup_{i=1}^3Sk(X_i^2)$ if and only if $\min\{a,b,c,d\}\geq 0$. So if $\min\{a,b,c,d\}\geq 0$, we have $Sk(a,b,c,d)=\bigcup_{i=1}^3\overline{D}_i$ as well.

For the converse, we let $\min\{a,b,c,d\}<0$ and consider two cases: (1) $Sk(D)$ is nonempty, and (2) $Sk(A_iX_i)$ is nonempty for some $i\in\{1,2,3\}$. For (1), the interior of $Sk(D)$ is nonempty and is disjoint from the union $\bigcup_{i=1}^3\overline{D}_i$. For (2), we may suppose $i=1$. As $Sk(AX_1)$ is $\trop(s_1)$-invariant yet $\Fix(\trop(s_1))$ is a line, so $Sk(AX_1)$ is also split by two components swapped by $\trop(s_1)$. Now $D_1$ is taking one component, hence the other component is not taken by neither $\overline{D}_1$ nor the union $\bigcup_{i=1}^3\overline{D}_i$. So the union is not the whole skeleton.

\subsection{Ping-pong}
\label{sec:ping-pong-lemma}
Recall the ping-pong lemma (see, e.g., \cite[\S{II.B} Lemma 24]{PingPongText}) for the group
\begin{equation}
\label{eqn:inf-inf-inf-reflection-group}
\Gamma=\langle\sigma_1,\sigma_2,\sigma_3\mid\sigma_1^2=\sigma_2^2=\sigma_3^2=1\rangle(=(\Z/2\Z)^{\ast 3}).
\end{equation}

\begin{theorem}[Ping-pong Lemma]
    \label{thm:ping-pong-lemma}
    Let $\Gamma$ be the group~\eqref{eqn:inf-inf-inf-reflection-group} that acts on a set $X$. Suppose there exist disjoint nonempty subsets $X_1,X_2,X_3\subset X$ such that $\sigma_j.X_i\subset X_j$ whenever $i\neq j$. Then $\Gamma$ acts faithfully on $X$.
\end{theorem}

Recall the group $\trop(\Gamma_{ABCD})$ generated by tropical Vieta involutions $\trop(s_i)$, $i=1,2,3$ (see \S{\ref{sec:markov-trop-vieta}}). Observe that (1) by the natural group map $\Gamma\to\trop(\Gamma_{ABCD})$, $\sigma_i\mapsto\trop(s_i)$, and (2) by $X=Sk(a,b,c,d)$ and $X_i=D_i$, we have that the hypotheses of Theorem~\ref{thm:ping-pong-lemma} are met. (For instance, we have $\trop(s_1).D_2\subset\trop(s_1).(X\setminus\overline{D}_1)=D_1$.) Therefore $\trop(\Gamma_{ABCD})$, and $\Gamma_{ABCD}$ as well, are isomorphic to $\Gamma$.

\begin{definition}[Tropical representation]
\label{def:tropical-representation}
The group map sending $s_i\mapsto\trop(s_i)$,
\begin{equation}
\label{eqn:tropical-representation}
\trop\colon\Gamma_{ABCD}\to\trop(\Gamma_{ABCD}),
\end{equation}
is called the \emph{tropical representation} of $\Gamma_{ABCD}$. For any $\gamma\in\Gamma_{ABCD}$, we define $\trop(\gamma)$ as the image of the tropical representation.
\end{definition}

\begin{remark}
    The above ping-pong argument gives an alternative proof of El'Huti's result $\Gamma_{ABCD}\cong\Gamma$~\cite[Thm. 1]{ElHuti}, but over a non-archimedean field. The original proof of El'Huti is based on the $\Gamma_{ABCD}$-action on the Picard--Manin space~\cite[\S{3.2}]{Can11} of the (projective closure of the) surface~\cite[\S{1}]{ElHuti}. Our new proof is essentially not away from this idea, but using different language to access it.
\end{remark}

The fact that $D_i\subset Sk(a,b,c,d)$ naturally gives an environment to apply the ping-pong lemma leads us to understand the behavior of $Sk(a,b,c,d)\setminus\bigcup_{i=1}^3\overline{D}_i$ under the dynamics. This will be addressed in \S{\ref{sec:pp-theory}} below; before that, we need to understand the behavior of $\trop(s_i)$'s ``near infinity.''

\section{Quadratic Adjacency Loci}
\label{sec:quad-cell}

In this section, we discuss how a quadratic adjacency locus reduces to a subquadratic one by applying an appropriate tropicalized map. The main result of the section is announced in Proposition~\ref{lem:quadratic-reduction-result} below.

\subsection{The Goal}

Let a \emph{boundary ray} be an intersection $Sk(X_i^2)\cap Sk(X_j^2)$ ($i\neq j$) of distinct quadratic adjacency loci. In fact, some inequality manipulations show the following numerical descriptions:
\begin{align}
    Sk(X_2^2)\cap Sk(X_3^2) &=\left\{(0,t,t) : t\leq\Theta_1\right\}, \label{eqn:boundary-ray-1} \\
    Sk(X_1^2)\cap Sk(X_3^2) &=\left\{(t,0,t) : t\leq\Theta_2\right\}, \label{eqn:boundary-ray-2} \\
    Sk(X_1^2)\cap Sk(X_2^2) &=\left\{(t,t,0) : t\leq\Theta_3\right\}, \label{eqn:boundary-ray-3}
\end{align}
where $\Theta_1=\min\{0,\frac12a,b,c,\frac12d\}$, $\Theta_2=\min\{0,a,\frac12b,c,\frac12d\}$, and $\Theta_3=\min\{0,a,b,\frac12c,\frac12d\}$. 
With this term, we state the main result of this section as follows.

\begin{proposition}
\label{lem:quadratic-reduction-result}
    Suppose $\min\{a,b,c,d\}<0$. For any $x\in Sk(a,b,c,d)$, there exists some $\gamma\in\Gamma_{ABCD}$ such that $\trop(\gamma)(x)$ is on a subquadratic adjacency locus or on a boundary ray.
\end{proposition}

Points not on a subquadratic adjacency locus nor on a boundary ray are precisely those points that lie on the interior of a quadratic adjacency locus. So the discussion lies on how to understand the actions of $\trop(s_i)$, $i=1,2,3$, on quadratic adjacency locus. It turns out that the cone parametrization map $\phi\colon\R^2/\langle\pm1\rangle\to\R^3$~\eqref{eqn:cone-parametrization} is taking a central role here.

\begin{remark}
    Recall the disks $D_1$, $D_2$, and $D_3$ from Proposition~\ref{lem:cayley-domain-description}. Then we have, for $i\neq j$, $\partial D_i\cap\partial D_j=Sk(X_i^2)\cap Sk(X_j^2)$. This follows from numerical estimates found in Propostion~\ref{lem:intersection-of-fixed-sets}(b). Thus later, we will also use `boundary rays' to refer to an intersection $\partial D_i\cap\partial D_j$.
\end{remark}

\subsection{Quadratic Adjacency Loci and Cone Parametrization}

Denote the 1-norm of a space point $x=(x_1,x_2,x_3)\in\R^3$ by $\|x\|_1:=|x_1|+|x_2|+|x_3|$. We choose a norm $\|{}\cdot{}\|$ on $\R^2$ by $\|(u,v)\|=|u|+|v|+|u+v|$. Then the cone parametrization map preserves the norm: $\|\ConeParam(u,v)\|_1=\|(u,v)\|$.

We collect some facts about the quadratic adjacency loci.
\begin{lemma}
    \label{lem:quadratic-cells-basics}
    Let $Sk=Sk(a,b,c,d)$ and $\Theta=\min\{0,a,b,c,\frac12d\}$.
    \begin{enumerate}[(a)]
        \item We have $Sk\subset(-\infty,0]^3$. That is, the space coordinates of any point in the skeleton are nonpositive.
        \item The union $Q=\bigcup_{i=1}^3Sk(X_i^2)$ of quadratic adjacency loci lies on the image of $\ConeParam$.
        \item Let $i=1,2,3$. For any $x=(x_1,x_2,x_3)\in Q\cap\trop(s_i)^{-1}(Q)$, we have the following formulae, depending on $i$:
        \begin{align}
            \trop(s_1)(x_1,x_2,x_3) &= (\min\{2x_2,2x_3\}-x_1,x_2,x_3), \label{eqn:quad-trop(s1)} \\
            \trop(s_2)(x_1,x_2,x_3) &= (x_1,\min\{2x_1,2x_3\}-x_2,x_3), \label{eqn:quad-trop(s2)} \\
            \trop(s_3)(x_1,x_2,x_3) &= (x_1,x_2,\min\{2x_1,2x_2\}-x_3). \label{eqn:quad-trop(s3)}
        \end{align}
        \item If $x=(x_1,x_2,x_3)\in Sk$ is such that $x_1+x_2+x_3\leq 2\Theta$, then $x\in\bigcup_{i=1}^3Sk(X_i^2)$. In particular, $\bigcup_{i=1}^3Sk(X_i^2)$ is a neighborhood of infinity.
    \end{enumerate}
\end{lemma}
\begin{proof}
    (a) One way to show this is to use Lemma~\ref{lem:inequality-formulation-tropical-skeletal-Laurent}, with $w=0$. If $(x_1,x_2,x_3)\in Sk(a,b,c,d)$, we have the following inequalities:
    \begin{align*}
        x_1+x_2+x_3 &\leq 2x_1, & x_1+x_2+x_3 &\leq 2x_2, & x_1+x_2+x_3 &\leq 2x_3.
    \end{align*}
    The first two verifies $x_3\leq x_1-x_2$ and $x_3\leq x_2-x_1$. Hence $x_3\leq -|x_1-x_2|\leq 0$ follows. Likewise for other two variables.

    (b) Note first that, on the union $\bigcup_{i=1}^3Sk(X_i^2)$, we have
    \[\min\left\{\begin{array}{l}2x_1,2x_2,2x_3, \\ a+x_1,b+x_2,c+x_3,d\end{array}\right\}=\min\{2x_1,2x_2,2x_3\}.\]
    Hence the union lies on the locus $x_1+x_2+x_3=\min\{2x_1,2x_2,2x_3\}$~\eqref{eqn:image-cone-parametrization}, the image of $\ConeParam$.

    (c) We focus on $i=1$, as the other two are dealt similarly. If $x\in Sk(X_2^2)\cup Sk(X_3^2)$, then
    \[\min\{2x_2,2x_3\}\leq\min\{2x_1,a+x_1,b+x_2,c+x_3,d\}\leq\min\{b+x_2,c+x_3,d\},\]
    so in particular,
    \begin{equation}
        \label{eqn:quadratic-cells-basics-1}
        \min\{2x_2,2x_3,b+x_2,c+x_3,d\}=\min\{2x_2,2x_3\}.
    \end{equation}
    This suffices to verify \eqref{eqn:quad-trop(s1)} for this case. 
    
    If $x\in Sk(X_1^2)$, then from $\trop(s_1)(x)\in\bigcup_{i=1}^3Sk(X_i^2)$, we have $\trop(s_1)(x)\in Sk(X_2^2)\cup Sk(X_3^2)$. If we write $\trop(s_1)(x)=(x_1',x_2,x_3)$, then we have
    \[\min\{2x_2,2x_3\}\leq\min\{2x_1',a+x_1',b+x_2,c+x_3,d\}\leq\min\{b+x_2,c+x_3,d\},\]
    so \eqref{eqn:quadratic-cells-basics-1} still holds, which yields \eqref{eqn:quad-trop(s1)} for this case.

    (d) It suffices to show that, if $x_1+x_2+x_3<2\Theta$ then $x$ is not in any subquadratic adjacency locus. We show its contrapositive: any point $x$ in a subquadratic adjacency locus verifies $x_1+x_2+x_3\geq 2\Theta$.

    If $x\in Sk(D)$, then we have $x_1+x_2+x_3=d\geq 2\Theta=\min\{0,2a,2b,2c,d\}$. If $x\in Sk(AX_1)$, then we have $x_1+x_2+x_3=a+x_1\leq 2x_1$, so $x_2+x_3=a$ and $x_1\geq a$ follows. Hence $x_1+x_2+x_3\geq 2a\geq 2\Theta$. The cases $x\in Sk(BX_2)$ and $x\in Sk(CX_3)$ are dealt similarly with the case $x\in Sk(AX_1)$.

    The claim also verifies that $\bigcup_{i=1}^3Sk(X_i^2)$ is in the complement of the ball $\|x\|_1<-2\Theta$, thanks to part (a). Thus the union is a neighborhood of infinity.
\end{proof}

\begin{remark}
    Propositions~\ref{lem:holomorphic-same-skeletons} and \ref{lem:holomorphic-same-dynamics} can be seen as special cases of this Lemma. Note also that \eqref{eqn:quad-trop(s1)} and \eqref{eqn:trop-s1-holomorphic} are the same formulae, and similarly for the other two.
\end{remark}

\subsection{\texorpdfstring{$\mathsf{GL}_2(\Z)$}{GL2(Z)}-orbits in \texorpdfstring{$\R^2$}{R2}: the GCD Function}

Albeit a basicmost thing in arithmetic, we show that there is an upper-semicontinuous function $\gcd\colon\R^2\to\R$ that almost determines the $\mathsf{GL}_2(\Z)$-orbit in $\R^2$ (see Lemma~\ref{lem:GL2Z-orbit-plane}). To start, we introduce a function on $\R$ that is designed to return ``$\gcd(1,x)$.'' This function will be called \emph{Thomae's function}, following \cite{BRS09}.

\begin{definition}[Thomae's Function]
    Let $f\colon\R\to[0,1]$ be a function defined as follows. If $x=p/q$ is a rational number with $p,q\in\Z$ coprime, then we define $f(x)=1/|q|$. If $x$ is irrational, we define $f(x)=0$.

    This function $f$ is called the \emph{Thomae's function}. We will denote this function by $\gcd(1,x)$.
\end{definition}

We think that 0 is coprime with 1, so that we declare $\gcd(1,0)=1$. We also have the following classical fact.

\begin{theorem}
    \label{thm:thomae}
    Thomae's function is continuous at irrational numbers, and discontinuous at rational numbers. More precisely, we have $\lim_{x\to c}\gcd(1,x)=0$ for any real $c$.
\end{theorem}

\begin{definition}
    \label{def:gcd-function}
Given Thomae's function $\gcd(1,x)$, we define $\gcd(a,b)$ as
\begin{equation}
    \label{eqn:gcd-function}
    \gcd(a,b)=\begin{cases} |a|\cdot\gcd(1,b/a) & (a\neq 0), \\ |b| & (a=0). \end{cases}
\end{equation}
\end{definition}
This is compatible with the usual GCD of integers. To see why, we suppose $a,b\in\Z$ and both are nonzero. If $\delta>0$ is the GCD of $a$ and $b$, then we can write $a=\pm p\delta$ and $b=\pm q\delta$ for $p,q\in\Z_{>0}$ coprime. Thus $|a|\cdot\gcd(1,b/a)=p\delta\cdot\gcd(1,\pm q/p)=p\delta(1/p)=\delta$ follows.

Not only $\gcd(a,b)$ above generalizes GCD of integers, but for $a,b$ reals: it measures how two numbers are rationally linearly dependent. We list some properties of $\gcd(a,b)$ in this vein, without proofs.
\begin{proposition}
    \label{lem:gcd-properties}
    \begin{enumerate}[(a)]
        \item We have $\gcd(cx,cy)=|c|\cdot\gcd(x,y)$ for all $c\in\R$.
        \item Suppose $a,b$ are $\Q$-linearly dependent with $b\neq 0$. If $a/b=p/q$ with $p,q\in\Z$ coprime, we have
        \[\gcd(a,b) = \left|\frac{a}{p}\right| = \left|\frac{b}{q}\right|.\]
        (If $a=0$, ignore $|a/p|$.)
        \item If $a,b$ are $\Q$-linearly independent, then $\gcd(a,b)=0$.
        \item Whenever $A=\begin{bmatrix}
            \alpha & \beta \\ \gamma & \delta
        \end{bmatrix}$ is in $\mathrm{GL}_2(\Z)$, we have $\gcd(\alpha x+\beta y,\gamma x+\delta y)=\gcd(x,y)$.
    \end{enumerate}
\end{proposition}

From Theorem \ref{thm:thomae}, one can show that $\lim_{(x,y)\to(a,b)}\gcd(x,y)=0$ for all $(a,b)\in\R^2$. As $\gcd(a,b)\geq 0$ by definition, this limit indicates that $\gcd$ is an upper-semicontinuous function. In particular, the superlevel sets $\{(a,b)\in\R^2 : \gcd(a,b)\geq\eta\}$ are closed. In fact, if $\eta>0$, we can explicitly describe this set as a family of rays $\bigcup_{(p,q)=1}\{(tp,tq) : t\geq\eta\}$, where $(p,q)$ runs through all pairs of integers that are coprime.

\begin{remark}
    The concept and notation of the $\gcd$ function is not original; it has previously been encountered in a graphing software \cite{grafeq}, where the function is denoted by $\gcd(x,y)$.
\end{remark}

The ray description of the superlevel set can be addressed in dynamical terms, as follows.
\begin{lemma}
\label{lem:GL2Z-orbit-plane}
    Let $G=\mathsf{GL}_2(\Z)$ and $G(2)=\ker(\mathsf{GL}_2(\Z)\to\mathsf{GL}_2(\Z/2\Z))$. Let $(x,y)\in\R^2$ be any vector.
    \begin{enumerate}[(a)]
        \item If $\gcd(x,y)=0$, then for every neighborhod $U$ of the origin, there exists $g\in G(2)$ such that $g.(x,y)\in U$.
        \item If $\delta=\gcd(x,y)>0$, then there exists $g\in G$ such that $g.(x,y)=(\delta,0)$.
        \item If $\delta=\gcd(x,y)>0$, then the $G(2)$-orbit of the vector $(x,y)$ passes one of the following 6 vectors: $(\pm\delta,0)$, $(0,\pm\delta)$, or $(\pm\delta,\mp\delta)$.
    \end{enumerate}
\end{lemma}
\begin{proof}
    Throughout the proof, we fix a set of matrices
    \begin{align*}
        g_1 &= I_2, &
        g_2 &= \begin{bmatrix} -1 & -1 \\ 0 & 1 \end{bmatrix}, &
        g_3 &= \begin{bmatrix} 1 & 0 \\ -1 & -1 \end{bmatrix}, \\
        g_4 &= \begin{bmatrix} 0 & 1 \\ 1 & 0 \end{bmatrix}, &
        g_5 &= \begin{bmatrix} -1 & -1 \\ 1 & 0 \end{bmatrix}, &
        g_6 &= \begin{bmatrix} 0 & 1 \\ -1 & -1 \end{bmatrix},
    \end{align*}
    so that the quotient group $G(2)\backslash G$ equals to $\{G(2).g_i : 1\leq i\leq 6\}$, and $\{g_i.(\delta,0) : 1\leq i\leq 6\}=\{(\pm\delta,0),(0,\pm\delta),(\pm\delta,\mp\delta)\}$.
    
    (a) If $x=0$, then $\gcd(x,y)=|y|=0$ makes the claim obvious. If not, the ratio $y/x$ defines an irrational number. Let $a_0=\lfloor y/x\rfloor$ and $a_1,a_2,\ldots\in\Z_{>0}$ be such that $y/x-a_0=[0;a_1,a_2,\ldots]$, i.e.,
    \[\frac{y}{x}=a_0+\frac1{a_1+\dfrac1{a_2+\dfrac1{a_3+\cdots}}}.\]
    Consider the sequence $(A_n)_{n=0}^\infty$ of matrices,
    \[A_n=\begin{bmatrix} a_0 & 1 \\ 1 & 0 \end{bmatrix}\begin{bmatrix} a_1 & 1 \\ 1 & 0 \end{bmatrix} \cdots \begin{bmatrix} a_n & 1 \\ 1 & 0 \end{bmatrix}.\]
    Then by the $n$-th convergents $p_n/q_n\approx y/x$, we have $A_n=\begin{bmatrix} p_n & p_{n-1} \\ q_n & q_{n-1} \end{bmatrix}$ for $n\geq 0$. These convergents verify $|p_n-(y/x)q_n|\leq|q_n|^{-1}$ for all $n\geq 0$. Thus we have
    \[A_n^{-1}\begin{bmatrix} x \\ y \end{bmatrix}=\pm x\cdot\begin{bmatrix} p_{n-1}-\frac{y}{x}q_{n-1} \\ p_n-\frac{y}{x}q_n \end{bmatrix}\to 0,\]
    as $n\to\infty$. Now for each $A_n$, there is a unique index $1\leq i(n)\leq 6$ such that $g_{i(n)}A_n^{-1}\in G(2)$, and we still have $g_{i(n)}A_n^{-1}.(x,y)\to(0,0)$. Choose $g=g_{i(n)}A_n^{-1}$, with $n\gg 0$, for the goal.
    
    (b) We know that $p=x/\delta$ and $q=y/\delta$ are relatively prime integers. With $r,s\in\Z$ such that $ps-qr=1$, $g=\begin{bmatrix} p & r \\ q & s \end{bmatrix}^{-1}\in\mathsf{SL}_2(\Z)\subset G$ will verify $g.(x,y)=(\delta,0)$.
    
    (c) As the $G$-orbit of $(x,y)$ passes $(\delta,0)$, the $G(2)$-orbit of $(x,y)$ passes $\{g_i.(\delta,0) : 1\leq i\leq 6\}=\{(\pm\delta,0),(0,\pm\delta),(\pm\delta,\mp\delta)\}$.
\end{proof}

We focus on the group $G(2)$ in Lemma~\ref{lem:GL2Z-orbit-plane}, because $G(2)/\langle\pm I_2\rangle\subset\mathsf{PGL}_2(\Z)$ is precisely the image of the map $j\colon\Gamma_{ABCD}\to\mathsf{PGL}_2(\Z)$ in Proposition~\ref{lem:PGL2Z-Vieta-group-identification}.

\subsection{Proof of Proposition~\ref{lem:quadratic-reduction-result}}

Suppose otherwise: that is, we have a point $x$ in the union $\bigcup_{i=1}^3Sk(X_i^2)$ of quadratic adjacency loci such that for every $\gamma\in\Gamma_{ABCD}$, $\trop(\gamma)(x)$ is on the interior of some quadratic adjacency locus. In that case, the formulae \eqref{eqn:quad-trop(s1)} to \eqref{eqn:quad-trop(s3)} apply for any point $\trop(\gamma)(x)$ in the tropical orbit.

Thanks to Lemma~\ref{lem:cone-parametrization-is-equivariant}, this implies that we can evaluate $\trop(\gamma)(x)$ as follows. Let $\vec{u}=\ConeParam^{-1}(x)$. Then by the group map $j\colon\Gamma_{ABCD}\to\mathsf{PGL}_2(\Z)$ in Proposition~\ref{lem:PGL2Z-Vieta-group-identification},
\[\trop(\gamma)(x)=\ConeParam\left(j(\gamma).\vec{u}\right).\]
Taking the norm, we have $\|\trop(\gamma)(x)\|_1=\|j(\gamma).\vec{u}\|$ in addition.

If $\gcd(\vec{u})=0$, then by Lemma~\ref{lem:GL2Z-orbit-plane}(a), $\|j(\gamma).\vec{u}\|=\|\trop(\gamma)(x)\|_1$ can be arbitrarily small. But since $\min\{a,b,c,d\}<0$, we have $(0,0,0)\notin Sk(a,b,c,d)$, thanks to Proposition~\ref{lem:holomorphic-same-skeletons}. Thus there exists $\epsilon>0$ such that, every $x\in Sk(a,b,c,d)$ verifies $\|x\|_1\geq\epsilon$. In particular, $\|\trop(\gamma)(x)\|_1\geq\epsilon$ for all $\gamma\in\Gamma_{ABCD}$; contradiction.

If $\delta=\gcd(\vec{u})>0$, then by Lemma~\ref{lem:GL2Z-orbit-plane}(c), we can choose $j(\gamma)\in G(2)/\langle\pm I_2\rangle$ so that $j(\gamma).\vec{u}$ takes one of the following forms: $\pm(\delta,0)$, $\pm(0,\delta)$, or $\pm(\delta,-\delta)$. Taking $\ConeParam$, we have $\ConeParam(j(\gamma).\vec{u})=\trop(\gamma)(x)$ equal to $(-\delta,0,-\delta)$, $(0,-\delta,-\delta)$, or $(-\delta,-\delta,0)$. Since each verifies $2x_1=2x_3<2x_2$, $2x_2=2x_3<2x_1$, or $2x_1=2x_2<2x_3$ respectively, none of them can be the interior of a quadratic adjacency locus; contradiction.

\subsection{A Greedy Algorithm}
\label{sec:greedy-algorithm}

Although our proof of Proposition~\ref{lem:quadratic-reduction-result} is non-constructive, here is a ``greedy'' method to find a working $\gamma\in\Gamma_{ABCD}$.
\begin{lemma}
    \label{lem:contraction-norm}
    For each $i=1,2,3$, we have $\|\trop(s_i)(x)\|_1\leq\|x\|_1$ whenever $x\in Sk(X_i^2)$. The inequality is strict if $x$ is in the interior of $Sk(X_i^2)$.
\end{lemma}
\begin{proof}
    Let $i=1$; the others are dealt similar. Let $x=(x_1,x_2,x_3)\in Sk(X_1^2)$, and let $x_1'$ be such that $\trop(s_1)(x)=(x_1',x_2,x_3)$. We have
    \begin{align*}
        x_1+x_2+x_3=2x_1 &\leq \min\left\{2x_2,2x_3,a+x_1,b+x_2,c+x_3,d\right\} \tag{$\ast$} \\
        &\leq \min\left\{2x_2,2x_3,b+x_2,c+x_3,d\right\}=x_1+x_1'.
    \end{align*}
    Hence we have $x_1\leq x_1'$, which yields $x_1+x_2+x_3\leq x_1'+x_2+x_3$. If we translate this in terms of 1-norms, we have $\|x\|_1\geq\|\trop(s_1)(x)\|_1$.
    
    If $x$ is in the interior of $Sk(X_1^2)$, then we have a strict inequality at ($\ast$). So $x_1<x_1'$.
\end{proof}

The greedy method is this: whenever we have a point in the interior of $Sk(X_i^2)$, we apply $\trop(s_i)$ and see if it still lies on the interior of a quadratic adjacency locus. If not, then we are done; otherwise, we apply another $\trop(s_j)$ to reduce further. One can demonstrate that this method yields the desired $\gamma\in\Gamma_{ABCD}$ in Proposition~\ref{lem:quadratic-reduction-result}, but the proof is mostly wrestling with the Euclidean algorithm that the approach implicitly follows. In addition, the key step---that the method terminates---relies on a proof by contradiction that is essentially no different from the proof presented above.

\section{The Ping-Pong Theory}
\label{sec:pp-theory}

Recall the group
\[
\tag{\ref{eqn:inf-inf-inf-reflection-group}}
\Gamma=\langle\sigma_1,\sigma_2,\sigma_3\mid\sigma_1^2=\sigma_2^2=\sigma_3^2=1\rangle=(\Z/2\Z)^{\ast 3}.
\]
This group is isomorphic to the group $\Gamma_{ABCD}$ and the tropicalized group $\trop(\Gamma_{ABCD})$, thanks to the structure discovered in Proposition~\ref{lem:cayley-domain-description} plus the ping-pong lemma (see \S{\ref{sec:ping-pong-lemma}}).

But instead of letting the ping-pong lemma to classify the group $\trop(\Gamma_{ABCD})$, we may use the lemma to classify the skeleton $Sk(a,b,c,d)$ as a $\Gamma$-space (topological space with a continuous $\Gamma$-action). We introduce the notion of a \emph{ping-pong structure}, aimed to be a tool to compare two $\Gamma$-spaces. The main comparison result is Proposition~\ref{thm:ping-pong-table} below, and this will be the key fact to understand the tropical dynamics of the skeleton with meromorphic parameters.

\subsection{A Recall from Ping-Pong Lemma}

There are some details in the proof of the ping-pong lemma (Theorem~\ref{thm:ping-pong-lemma}) that we would like to remark and generalize, so we post a proof below.

Recall the hypotheses: $X$ is a set with a $\Gamma$-action, and $X_1,X_2,X_3\subset X$ are disjoint nonempty subsets such that $\sigma_i.X_j\subset X_i$ holds whenever $i\neq j$.

\begin{proof}[Proof of Theorem~\ref{thm:ping-pong-lemma}]
    Let $w\in\Gamma$ be a nontrivial reduced word, so that $w=\sigma_{i(n)}\sigma_{i(n-1)}\cdots\sigma_{i(1)}$, where each $i(j)\in\{1,2,3\}$ and $i(j+1)\neq i(j)$ for $j=1,\ldots,n-1$. Let $w$ be cyclically reduced, i.e., $n=1$ or $\sigma_{i(n)}\neq\sigma_{i(1)}$. If otherwise, we conjugate $w$ by $\sigma_{i(1)}$ to remove $\sigma_{i(n)}$ and $\sigma_{i(1)}$ from the first and last of $w$.

    If $n=1$ then we may assume that $w=\sigma_1$. If $w$ acts on any point $x\in X_2$, then we have $w.x\in X_1$, and as $X_1$ and $X_2$ are disjoint, $w.x\neq x$. So $w$ acts nontrivially.

    If $n>1$, then as $\sigma_{i(n)}\neq\sigma_{i(1)}$, we have $i(n)\neq i(1)$ as well. Reindexing the indices if necessary, we may assume that $i(n)=2$ and $i(1)=1$. Pick any point $x\in X_3$. Then one can inductively show that $\sigma_{i(j)}\sigma_{i(j-1)}\cdots\sigma_{i(1)}.x\in X_{i(j)}$. It follows that
    \[w.x=\sigma_{i(n)}\sigma_{i(n-1)}\cdots\sigma_{i(1)}.x\in X_{i(n)}=X_2,\]
    and as $X_2$ and $X_3$ are disjoint, we have $w.x\neq x$. So $w$ acts nontrivially.
\end{proof}

The following comes from the idea of showing $\sigma_{i(j)}\cdots\sigma_{i(1)}.x\in X_{i(j)}$ in the proof above. Recall that a $\Gamma$-space is a topological space with a continuous $\Gamma$-action.
\begin{corollary}
    \label{lem:ping-pong-general}
    Suppose $X$ is a $\Gamma$-space. Let $X_1$, $X_2$, and $X_3$ be disjoint nonempty subsets of $X$ such that $\sigma_i.X_j\subset X_i$ whenever $i\neq j$. Suppose $\sigma_i$ maps $X_i$ onto $X\setminus\overline{X}_i$, for $i=1,2,3$. Then for a reduced word $\sigma_{i(n)}\cdots\sigma_{i(1)}\in\Gamma$, we have
    \begin{align}
        \sigma_{i(n)}\cdots\sigma_{i(1)}.(X\setminus X_{i(1)}) &\subset\overline{X}_{i(n)}, \label{eqn:ping-pong-general-0} \\
        \sigma_{i(n)}\cdots\sigma_{i(1)}.(X\setminus\overline{X}_{i(1)}) &\subset X_{i(n)}. \label{eqn:ping-pong-general-1}
    \end{align}
\end{corollary}
\begin{proof}
    This follows from $\sigma_i.(X\setminus\overline{X}_i)=X_i$, $\sigma_i.(X\setminus X_i)=\overline{X}_i$, $\sigma_i.X_j\subset X_i$, and $\sigma_i.\overline{X}_j\subset\overline{X}_i$, together with induction.
\end{proof}

\subsection{The Ping-pong Structure}

Motivated by the statement of the ping-pong lemma as in Theorem~\ref{thm:ping-pong-lemma}, we introduce the following notion.

\begin{definition}[Ping-pong structure]
    \label{def:ping-pong-structure}
    Let $X$ be a $\Gamma$-space. Suppose we have disjoint open subsets $X_1,X_2,X_3\subset X$ such that
    \begin{enumerate}[(i)]
        \item $\sigma_j.X_i\subset X_j$ whenever $i\neq j$,
        \item $\sigma_i$ fixes the boundary $\partial X_i$, pointwise, and
        \item $\sigma_i$ sends $X_i$ onto $X\setminus\overline{X}_i$.
    \end{enumerate}
    Then we say the triple $(X_1,X_2,X_3)$ the \emph{ping-pong structure} on $X$.
\end{definition}

In general, the boundaries $\partial X_i$ may have points with complicated stabilizers. Comparing them is essential when we compare two $\Gamma$-spaces with ping-pong structures.

\begin{definition}[Fat, Thin]
    Let $(X_1,X_2,X_3)$ be a ping-pong structure on $\Sigma\ACTS X$. Define $L_i=\partial X_i\setminus\bigcup_{j\neq i}\overline{X}_j$ and $X_0=(X\setminus\bigcup_{i=1}^3\overline{X}_i)\cup\bigcup_{i=1}^3 L_i$. We say the ping-pong structure is \emph{fat} if $X_0$ has a nonempty interior. Otherwise, we say the structure is \emph{thin}.
    
    For a fat ping-pong structure, we say $X_0$ is the \emph{ping-pong table} and sets $L_i$'s are \emph{ping-pong nets}. We may denote a fat ping-pong structure with its table, by $(X_1,X_2,X_3;X_0)$.
\end{definition}

Figure \ref{fig:fat-ping-pong-structure} sketches a fat ping-pong structure. Figure \ref{fig:thin-ping-pong-structure} sketches a thin ping-pong structure on a circle.

Note that for fat ping-pong structures, we can recover the nets $L_i$ by $L_i=X_0\cap\overline{X}_i$. So it is not necessary to indicate the nets in the description of a fat ping-pong space.

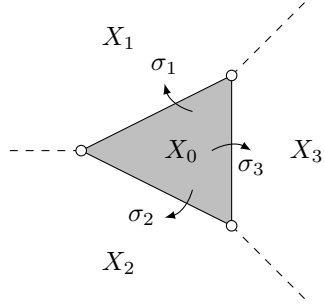
\begin{figure}
    \centering
    \begin{tikzpicture}
        \draw[dashed] (0,1) -- (1,2);
        \draw[dashed] (0,-1) -- (1,-2);
        \draw[dashed] (-2,0) -- (-3,0);
        
        \fill[gray,opacity=.5] (-2,0) -- (0,1) -- (0,-1) -- (-2,0);
        \draw (-2,0) -- (0,-1)--(0,1) -- cycle;
        \draw[fill=white] (0,1) circle (.07);
        \draw[fill=white] (0,-1) circle (.07);
        \draw[fill=white] (-2,0) circle (.07);
        \node at ({-2/3},0) {$X_0$};

        \node at (1,0) {$X_3$};
        \node at (-1.5,1.5) {$X_1$};
        \node at (-1.5,-1.5) {$X_2$};
        \foreach \i in {.75em}{
            \draw[-latex] (-\i,0) to[bend left] (\i,0) node[below] {$\sigma_3$};
            \draw[-latex] (-.707cm+.707*\i,.707cm-.707*\i) to[bend left] (-.707cm-.707*\i,.707cm+.707*\i) node[above] {$\sigma_1$};
            \draw[-latex] (-.707cm+.707*\i,-.707cm+.707*\i) to[bend left] (-.707cm-.707*\i,-.707cm-.707*\i) node[left] {$\sigma_2$};
        }
    \end{tikzpicture}
    \caption{A sketch of a fat ping-pong structure.} \label{fig:fat-ping-pong-structure}
\end{figure}
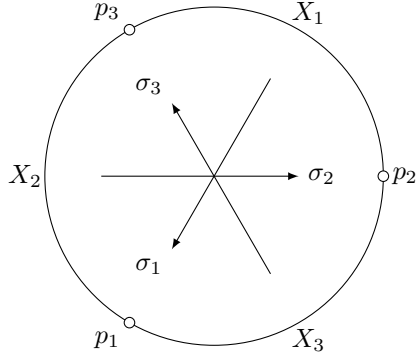
\begin{figure}
    \centering
    \begin{tikzpicture}[scale=.75]
        \draw (0,0) circle (3);
        \draw[fill=white] (0:3) circle ({.07/.75}) node[right] {$p_2$};
        \draw[fill=white] (120:3) circle ({.07/.75}) node[above left] {$p_3$};
        \draw[fill=white] (240:3) circle ({.07/.75}) node[below left] {$p_1$};
        \node at (60:3cm+1em) {$X_1$};
        \node at (180:3cm+1em) {$X_2$};
        \node at (-60:3cm+1em) {$X_3$};
        \draw[-latex] (60:2) -- (60:-1.5) node[below left] {$\sigma_1$};
        \draw[-latex] (180:2) -- (180:-1.5) node[right] {$\sigma_2$};
        \draw[-latex] (-60:2) -- (-60:-1.5) node[above left] {$\sigma_3$};
    \end{tikzpicture}
    \caption{A sketch of a thin ping-pong structure.} \label{fig:thin-ping-pong-structure}
\end{figure}

For the tropical dynamics on a skeleton, it naturally carries a ping-pong structure, and the fat-thin dichotomy corresponds to the meromorphic-holomorphic dichotomy. Since the thin case is already addressed in Theorem~\ref{thm:tropical-dynamics-holomorphic-parameters} above, here we only focus on fat ping-pong structures.
\begin{lemma}
    \label{lem:skeletal-pp-structure}
    Let $X=Sk(a,b,c,d)$ and let $D_1$, $D_2$, and $D_3$ be the domains as in Proposition~\ref{lem:cayley-domain-description}. Then $(D_1,D_2,D_3)$ is a ping-pong structure. It is fat if and only if $\min\{a,b,c,d\}<0$.
\end{lemma}
\begin{proof}
    Proposition~\ref{lem:cayley-domain-description}, parts (i), (ii), and (iv) show that $(D_1,D_2,D_3)$ is a ping-pong structure. It is fat if and only if $X\setminus\bigcup_{i=1}^3\overline{D}_i$ is nonempty, if and only if $\min\{a,b,c,d\}<0$, by part (vi).
\end{proof}

\subsection{Comparing Fat Ping-pong Structures}

Here is the main comparison result.

\begin{proposition}[Comparing Fat Ping-pong Structures]
    \label{thm:ping-pong-table}
    Let $X,Y$ be $\Gamma$-spaces with fat ping-pong strutures $(X_1,X_2,X_3;X_0)$ and $(Y_1,Y_2,Y_3;Y_0)$ respectively. Any continuous map $\psi\colon X_0\to Y_0$ that respects ping-pong nets, i.e., that $\psi(X_0\cap\overline{X}_i)\subset Y_0\cap\overline{Y}_i$, extends $\Gamma$-equivariantly to a continuous map $\Psi\colon\Gamma.X_0\to\Gamma.Y_0$.
\end{proposition}

\begin{lemma}
    \label{lem:cayley-observations}
    Suppose $X$ is a $\Gamma$-space with a fat ping-pong structure $(X_1,X_2,X_3;X_0)$. Any point of $X_0$ has the stabilizer group which is either trivial or one of the subgroup $\{1,\sigma_i\}$. The subset $X_0\cap\overline{X}_i$ is precisely the set of points $\in X_0$ that have the stabilizer $\{1,\sigma_i\}$.

    Furthermore, if $x_0,x_1\in X_0$ and $g,g'\in\Gamma$ has $g.x_0=g'.x_1$, we have $x_0=x_1$ and $g^{-1}g'$ is in the stabilizer of $x_0$.
\end{lemma}
\begin{proof}
    Suppose $x\in X_0\setminus\bigcup_{i=1}^3\overline{X}_i$. Then for any nontrivial reduced $g=\sigma_{i(n)}\cdots\sigma_{i(1)}\in\Sigma$, because $x\in X\setminus\overline{X}_{i(1)}$, we have $g.x\in X_{i(n)}$ by ping-pong \eqref{eqn:ping-pong-general-1}. Hence $g.x\notin X_0$, so $g$ cannot stabilize $x$.

    Suppose $x\in X_0\cap\overline{X}_i$. We may assume that $i=1$. Suppose $g=\sigma_{i(n)}\cdots\sigma_{i(1)}\in\Sigma$ is nontrivial, reduced, and $i(1)\neq 1$. Then because $x\in X\setminus\overline{X}_{i(1)}$, we have $g.x\in X_{i(n)}\subset X\setminus X_0$ by ping-pong \eqref{eqn:ping-pong-general-1}. So $g$ cannot stabilize $x$ unless $i(1)=1$. But even if $i(1)=1$, unless $g=\sigma_{i(1)}$, $g\sigma^{-1}_{i(1)}.x\notin X_0$ by the same reason above. Thus the stabilizer of $x$ is $\{1,\sigma_1\}$.

    The above argument also shows that $g.x\notin X_0$ if $g$ is not in the stabilizer of $x$. Hence if $x_0=g^{-1}g'.x_1$ for some $x_0,x_1\in X_0$, we must have $g^{-1}g'$ in the stabilizer of $x_1$. Therefore $x_0=g^{-1}g'.x_1=x_1$ follows.
\end{proof}

\begin{lemma}
    \label{lem:cayley-neighborhoods}
    Suppose $X$ is a $\Gamma$-space with a fat ping-pong structure $(X_1,X_2,X_3;X_0)$. The set $X_0\cup\sigma_i.X_0$ is a neighborhood of the ping-pong net $L_i=X_0\cap\overline{X}_i$.
\end{lemma}
\begin{proof}
    Let $i=1$, without loss of generality. We know that $X\setminus\bigcup_{i=1}^3\overline{X}_i$ is the interior of $X_0$, and we denote that set by $X_0^\circ$. The set $G=X_1\cup L_1\cup X_0^\circ$ then equals to $X\setminus(\overline{X}_2\cup\overline{X}_3)$ and is thus an open neighborhood of $L_1$.

    Next, the set $F=X_1\setminus\sigma_1.X_0^\circ$ is a closed subset of $X_1$ which is a subset of $\sigma_1.(\overline{X}_2\cup\overline{X}_3)$. Since $\sigma_1.L_1=L_1$ is disjoint with $\sigma_1.(\overline{X}_2\cup\overline{X}_3)$, the set $G\setminus F$ is an open neighborhood of $L_1$. We evaluate
    \begin{align*}
        G\setminus F &= (X_1\cup L_1\cup X_0^\circ)\setminus(X_1\setminus\sigma_1.X_0^\circ) \\
        &= L_1\cup X_0^\circ\cup\sigma_1.X_0^\circ,
    \end{align*}
    and we see that this is a subset of $X_0\cup\sigma_1.X_0$. The claim is shown.
\end{proof}

\begin{lemma}
    \label{lem:action-map-is-open}
    Suppose $X$ is a $\Gamma$-space with a fat ping-pong structure $(X_1,X_2,X_3;X_0)$. The action map $A\colon\Gamma\times X_0\to\Gamma.X_0$ is a quotient map.
\end{lemma}
\begin{proof}
    Suppose $U\subset\Gamma.X_0$ is a subset whose preimage $A^{-1}(U)\subset\Gamma\times X_0$ is open. We aim to show that $U\subset\Gamma.X_0$ is open.

    Note that $X_0\setminus\bigcup_{i=1}^3\overline{X}_i$ is precisely the interior $X_0^\circ$ of $X_0$. By Lemma \ref{lem:cayley-observations}, if we restrict $A$ to $\Gamma\times X_0^\circ$, then the action map is a homeomorphism onto the image $\Gamma.X_0^\circ$. Thus if $U\subset\Gamma.X_0^\circ$, $A^{-1}(U)$ is open iff $U$ is open.

    Suppose $U$ intersects $\Gamma.(X_0\setminus X_0^\circ)$ at $g.x_0$. Switching to $g^{-1}.U$ if necessary, we may assume $g=1$. From $X_0\setminus X_0^\circ=\bigcup_{i=1}^3(X_0\cap\overline{X}_i)$, we may assume $x_0\in\overline{X}_1$. We aim to show that $U$ is a neighborhood of $x_0$ (in $\Gamma.X_0$).

    Because $A^{-1}(x_0)=\{(1,x_0),(\sigma_1,x_0)\}$ by Lemma \ref{lem:cayley-observations}, we see that the open set $A^{-1}(U)$ is a neighborhood of this preimage. Thus there is a neighborhood $U'$ of $x_0$ in $X_0$ such that $\{1,\sigma_1\}\times U'\subset A^{-1}(U)$. It follows that $U'\cup\sigma_1.U'\subset U$. Hence it suffices to see if $U'\cup\sigma_1.U'$ is a neighborhood of $x_0$ in $\Gamma.X_0$ (to prove that $U$ itself is open).

    Suppose otherwise. Recall that $X_0\cup\sigma_1.X_0$ is a neighborhood of $x_0$ in $X$ and in $\Gamma.X_0$ (Lemma \ref{lem:cayley-neighborhoods}). For every neighborhood $V\subset X_0\cup\sigma_1.X_0$ of $x_0$ in $\Gamma.X_0$, $\sigma_1.V$ is also a neighborhood of $x_0$. Thus $V\cap\sigma_1.V$ is a neighborhood of $x_0$ as well. Because $x_0$ is not on the interior of $U'\cup\sigma_1.U'$ in $\Gamma.X_0$ we have an element $y_V\in V\cap\sigma_1.V$ such that $y_V\notin U'\cup\sigma_1.U'$. Here, we may set $y_V\in X_0$, by replacing $y_V\leftarrow\sigma_1.y_V$ if necessary. But as $U'$ is a neighborhood of $x_0$ in $X_0$, $y_V\in U'$ if $V$ gets small enough. This contradicts.
\end{proof}

We now prove Proposition~\ref{thm:ping-pong-table}.
\begin{proof}
    Denote the action maps by $A\colon\Gamma\times X_0\to\Gamma.X_0$ and by $A'\colon\Gamma\times Y_0\to\Gamma.Y_0$, which are quotient maps. We aim to build a map $\Psi\colon\Gamma.X_0\to\Gamma.Y_0$ which fits into the following commutative diagram:
    \[\begin{tikzcd}
        \Gamma\times X_0 \arrow[d, "\mathrm{Id}\times\psi"'] \arrow[r, "A"] & \Gamma.X_0 \arrow[d, "\exists!\Psi", dashed] \\
        \Gamma\times Y_0 \arrow[r, "A'"] & \Gamma.Y_0
    \end{tikzcd}\]
    
    Because $A$ is a quotient map, if we know
    \[A(g,x_0)=A(g',x_1)\quad\text{ implies }\quad A'(g,\psi(x_0))=A'(g',\psi(x_1)),\]
    then the map $\Psi$ is constructed by the universal property of quotient maps. Suppose $A(g,x_0)=A(g',x_1)$, i.e., $g.x_0=g'.x_1$. By Lemma \ref{lem:cayley-observations}, we have $x_0=x_1$ and $g^{-1}g'$ stabilizes $x_0$. Because $\psi$ respects ping-pong nets, $g^{-1}g'$ also fixes $\psi(x_0)$. Therefore $A'(g,\psi(x_0))=A'(g',\psi(x_1))$.
\end{proof}

\section{The Case of Meromorphic Parameters}
\label{sec:meromorphic-case}

We are now ready to prove the second half of our main Theorem~\ref{thm:A}: see Theorem~\ref{lem:one-meromorphic-model}.

\subsection{A Model Fat Ping-pong Structure: Hyperbolic Plane}

We introduce the following ping-pong structure on $\Hh^2=\{z\in\C : \Im z>0\}$ the upper half hyperbolic plane. Define the maps $r_1,r_2,r_3\colon\Hh^2\to\Hh^2$ as
\begin{equation}
    \label{eqn:hyperbolic-triangle-reflection}
    r_1(z)=\frac{\bar{z}}{2\bar{z}-1},\quad r_2(z)=2-\bar{z},\quad r_3(z)=-\bar{z},
\end{equation}
where $\bar{z}$ is the complex conjugate. Each defines an involution which is an isometry in hyperbolic metrics, i.e., a hyperbolic reflection. The reflections $r_1,r_2,r_3$ can be found in $\mathsf{PSL}_2^\pm(\R)=\mathsf{SL}_2^\pm(\R)/\langle\pm I_2\rangle$ (where $\mathsf{SL}_2^\pm(\R)$ is the set of $2\times 2$ real matrices of determinant $\pm 1$) and generates the subgroup that equals to
\[\ker(\mathsf{PGL}_2(\Z)\to\mathsf{PGL}_2(\Z/2\Z)).\]
This subgroup is isomorphic to $\Gamma=(\Z/2\Z)^{\ast 3}$, thanks to the ping-pong structure they exhibit.

\begin{proposition}[Hyperbolic ping-pong structure]
    \label{lem:hyperbolic-ping-pong-structure}
    Let $\Gamma\ACTS\Hh^2$ by $\sigma_i.z=r_i(z)$. This action admits a ping-pong structure with
    \begin{align*}
        H_1 &= \{z\in\Hh^2 : |2z-1|<1\}, \\
        H_2 &= \{z\in\Hh^2 : \Re z>1\}, \\
        H_3 &= \{z\in\Hh^2 : \Re z<0\}.
    \end{align*}
    The ping-pong table $H_0$ is topologically a closed disk minus three points on the boundary. (See Figure \ref{fig:hyperbolic-ping-pong-structure}.)
\end{proposition}
\begin{proof} Computations. \end{proof}

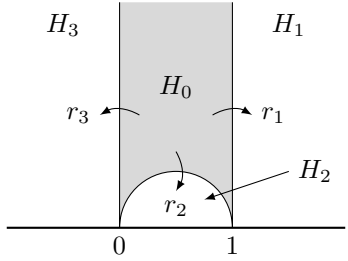
\begin{figure}
    \centering
    \begin{tikzpicture}[scale=1.5]
        \fill[gray,opacity=.3] (0,2) -- (0,0) arc (0:180:.5) -- (-1,2) -- cycle;
        \draw[thick] (-2,0)--(1,0);
        \draw (0,2) -- (0,0) arc (0:180:.5) -- (-1,2);
        \node[below] at (-1,0) {0};
        \node[below] at (0,0) {1};
        \node at (-1.5,1.8) {$H_3$};
        \node at (.5,1.8) {$H_2$};
        \draw[-latex] (.5,.5) node[right] {$H_1$} to (-.25,.25);
        
        \foreach \i in {.5em}{
        \draw[-latex] (-\i,1) to[bend left] (\i,1) node[right] {$r_2$};
        \draw[-latex] (-1cm+\i,1) to[bend right] (-1cm-\i,1) node[left] {$r_3$};
        \draw[-latex] (-.5,.5cm+\i) to[bend left] (-.5,.5cm-\i) node[below] {$r_1$};
        }

        \node at (-.5,1.25) {$H_0$};
    \end{tikzpicture}
    \caption{The ping-pong structure described in Proposition \ref{lem:hyperbolic-ping-pong-structure}.} \label{fig:hyperbolic-ping-pong-structure}
\end{figure}

\subsection{Shapes of Fat Skeletal Ping-pong Tables}

To connect our tropical ping-pong structure on the skeleton, in Lemma~\ref{lem:skeletal-pp-structure}, to the hyperbolic one, in Proposition~\ref{lem:hyperbolic-ping-pong-structure}, it suffices to demonstrate the following

\begin{lemma}
    \label{lem:fat-skeletal-pp-table}
    If $\min\{a,b,c,d\}<0$, then for the fat ping-pong structure $(D_1,D_2,D_3)$ of $X=Sk(a,b,c,d)$ as in Lemma~\ref{lem:skeletal-pp-structure}, its ping-pong table $D_0$ is topologically a closed disk minus three points on the boundary.
\end{lemma}
\begin{proof}
    Each ping-pong net $L_i=\partial D_i\setminus\bigcup_{j\neq i}\partial D_j$ is a line ($\cong\R$) minus two half-rays pointing distinct directions, by Proposition~\ref{lem:cayley-domain-description}(v). Hence each $L_i$ is an open interval; this is nondegenerate if $\min\{a,b,c,d\}<0$ (see Proposition~\ref{lem:intersection-of-fixed-sets}(c), with $w=0$). Moreover, we see that $\overline{L}_i\cap\overline{L}_j$ is a point, if $i\neq j$. Thus, if we join the intervals $\overline{L}_i$'s, we see that they form a circle, bounding $\overline{D}_0$. Thus the closure of the ping-pong table $\overline{D}_0$ is topologically a disk with boundary.

    So we ask how many points are in $\overline{D}_0\setminus D_0$. Such point corresponds to boundary rays. These rays start from the boundary of $D_0$ but does not intersect $D_0$ (or, `got removed' by the definition of $D_0$). Hence $\overline{D}_0\setminus D_0$ has 3 points, all found at the boundary. This verifies the topological description.
\end{proof}

So, together with Proposition~\ref{thm:ping-pong-table} above, we conclude with the following

\begin{theorem}[Tropical dynamics for meromorphic parameters]
    \label{lem:one-meromorphic-model}
    In the skeleton $Sk(a,b,c,d)$ with $\min\{a,b,c,d\}<0$, the $\Gamma_{ABCD}$-orbit $U=\Gamma_{ABCD}.D_0$ of the ping-pong table $D_0$ is an open dense subset, while $Sk(a,b,c,d)\setminus U$ is a countable union of rays (affine-linear images of $\R_{\geq 0}$). Furthermore, by the group isomorphism $j\colon\Gamma_{ABCD}\to\langle r_1,r_2,r_3\rangle$, $s_i\mapsto r_i$ in Proposition~\ref{lem:PGL2Z-Vieta-group-identification}, we have an $j$-equivariant homeomorphism $U\to\Hh^2$.
\end{theorem}
\begin{proof}
    Because the ping-pong tables, $D_0$ and $H_0$, are homeomorphic, by Proposition~\ref{thm:ping-pong-table} we can extend this homeomorphism to $U\cong\Hh^2$, $\Gamma_{ABCD}$-equivariantly. This is an open subset of the skeleton $Sk(a,b,c,d)$ because $D'=D_0\cup\bigcup_{i=1}^3s_i.D_0^\circ$ is open and $U=\Gamma_{ABCD}.D'$.

    To prove that $Sk(a,b,c,d)\setminus U$ is a countable union of rays, note first that $U$ contains all subquadratic cells except three points (that comes from boundary rays). From Proposition~\ref{lem:quadratic-reduction-result}, we see that all points of $Sk(a,b,c,d)\setminus U$ must be on a $\Gamma_{ABCD}$-orbit of some boundary ray, thus forms a countable union of rays. Furthermore, it must be a meager set, so by Baire category theorem, $U$ must be dense.
\end{proof}

\subsection{Some Words on the Exception Set}

In Theorem~\ref{lem:one-meromorphic-model}, the complement of $U$ is not engaged in the comparison with the hyperbolic plane. This complement consists of many rays, which rather should be compared with boundary points $\in\partial\Hh^2$ of the hyperbolic plane. Define the \emph{exception set} as the set of such rays.

Exception set always lie on the union $\bigcup_{i=1}^3Sk(X_i^2)$ of quadratic adjacency loci, and hence points on there appears as $\ConeParam(\vec{u})$ by some $\vec{u}\in\R^2/\langle\pm1\rangle$. The following method determines from which boundary ray the exception point originates.
\begin{proposition}
    Recall the numbers $\Theta_1=\min\{0,\frac12a,b,c,\frac12d\}$, $\Theta_2=\min\{0,a,\frac12b,c,\frac12d\}$, and $\Theta_3=\min\{0,a,b,\frac12c,\frac12d\}$. Let $x\in\bigcup_{i=1}^3Sk(X_i^2)$ and $\pm(u,v)=\ConeParam^{-1}(x)$. Suppose $\delta=\gcd(u,v)>0$. Then the point $x$ is on the $\Gamma_{ABCD}$-orbit of 
    \begin{enumerate}[(a)]
        \item $Sk(X_2^2)\cap Sk(X_3^2)$ if and only if $\delta\geq\Theta_1$ and $(u/\delta,v/\delta)\equiv(0,1)\pmod{2}$;
        \item $Sk(X_1^2)\cap Sk(X_3^2)$ if and only if $\delta\geq\Theta_2$ and $(u/\delta,v/\delta)\equiv(1,0)\pmod{2}$;
        \item $Sk(X_1^2)\cap Sk(X_2^2)$ if and only if $\delta\geq\Theta_3$ and $(u/\delta,v/\delta)\equiv(1,1)\pmod{2}$.
    \end{enumerate}
\end{proposition}
\begin{proof}
    We do not present the proof in details but just mention that it is a combination of Lemma~\ref{lem:quadratic-cells-basics}(c), Lemma~\ref{lem:cone-parametrization-is-equivariant}, and Lemma~\ref{lem:GL2Z-orbit-plane}(c), plus that for any $g\in\mathsf{PGL}_2(\Z)$ such that $g\equiv I_2\pmod{2}$, we have $g.(x,y)\equiv(x,y)\pmod{2}$ for $x,y\in\Z$.
\end{proof}

\subsection{Intersection of Tropical Domains}

Another application of Theorem~\ref{lem:one-meromorphic-model} is the following openness of intersections of tropical domains.
\begin{lemma}
    \label{lem:tropical-domain-open}
    The intersection $\bigcap_{\gamma\in\Gamma_{ABCD}}\bigcap_{i=1}^3\gamma.\mathrm{dom}(s_i)$ of tropical domains is open if $\min\{a,b,c,d\}<0$.
\end{lemma}
\begin{proof}
    Let $T_0=\bigcap_{i=1}^3\mathrm{dom}(s_i)$, $T=\bigcap_{\gamma\in\Gamma_{ABCD}}\gamma.T_0$, and $Sk=Sk(a,b,c,d)$. Note first that $Sk\setminus T_0$ consists of
    \begin{itemize}
        \item the boundary of each adjacency locus,
        \item line segments consisting of $Sk(AX_1)\setminus\mathrm{dom}(s_1)$; similar complements with respect to $\trop(s_2)$ and $\trop(s_3)$; and
        \item $\trop(s_1)$-images of the boundaries $\partial Sk(X_2^2)$, $\partial Sk(X_3^2)$, $\partial Sk(BX_2)$, $\partial Sk(CX_3)$, and $\partial Sk(D)$; similar images with respect to $\trop(s_2)$ and $\trop(s_3)$.
    \end{itemize}
    Let $L$ be the union of first two, so that $L\subset Sk\setminus T_0$, and $Sk\setminus T_0\subset L\cup\bigcup_{i=1}^3s_i.L$.
    
    We claim that $Sk\setminus T=\bigcup_{\gamma\in\Gamma_{ABCD}}\gamma.L$. It is clear that $\bigcup_{\gamma\in\Gamma_{ABCD}}\gamma.L\subset Sk\setminus T$. For the other direction, suppose $x\in Sk\setminus T$. Then there exists $\gamma\in\Gamma_{ABCD}$ such that $\gamma.x\in Sk\setminus T_0$. As $Sk\setminus T_0\subset L\cup\bigcup_{i=1}^3s_i.L$, we have $\gamma'\in\Gamma_{ABCD}$ such that $\gamma'.x\in L$.

    Observe also that $T$ excludes the exception set in $Sk$. So letting $U\subset Sk$ be as in Theorem~\ref{lem:one-meromorphic-model}, we have $T\subset U$. Let $D_0$ be the ping-pong table on $Sk$.     
    
    We claim that $D_0\setminus T=D_0\cap(L\cup\bigcup_{i=1}^3s_i.L)$. As $D_0$ and its images never intersect the $\Gamma_{ABCD}$-orbit of boundary rays, any intersection of the form $\gamma.D_0\cap L$ is occuring on (a) the boundary of a subquadratic adjacency locus or (b) line segments in a linear adjacency locus. But since the only $\gamma.D_0$ that can intersect subquadratic adjacency locus are those with $\gamma=1,s_1,s_2,s_3$, we see that $\gamma.D_0\cap L=\varnothing$ if $\gamma\notin\{1,s_1,s_2,s_3\}$. Hence
    \begin{align*}
        D_0\setminus T &= \bigcup_{\gamma\in\Gamma_{ABCD}}D_0\cap(\gamma.L)=\bigcup_{\gamma\in\Gamma_{ABCD}}\gamma.(\gamma^{-1}.D_0\cap L) \\
        &= \bigcup\{\gamma.(\gamma^{-1}.D_0\cap L) : \gamma=1,s_1,s_2,s_3\} \\
        &= D_0\cap(L\cup s_1.L\cup s_2.L\cup s_3.L).
    \end{align*}
    In particular, $D_0\setminus T$ is closed in $U$.

    Now $U\setminus T=\bigcup_{\gamma\in\Gamma_{ABCD}}\gamma.(D_0\setminus T)$ is a locally finite union of closed subsets of $U$, hence closed. This suffices to show that $T\subset U$ is open, thus $T\subset Sk$ is open too.
\end{proof}

\section{Application: Non-archimedean Fatou Domains}
\label{sec:nA-Fatou}

Let $(K,\val)$ be a complete, non-archimedean algebraically closed field.

\begin{theorem}
    \label{thm:p-adic-fatou-domain}
    Suppose $S_{ABCD}$ is defined over $(K,\val)$ with meromorphic parameters. Then $S_{ABCD}(K)$ has a Fatou domain. That is, there is a nonempty open subset $U\subset S_{ABCD}(K)$ such that, for every sequence $(g_n)_{n=1}^\infty$ in $\Gamma_{ABCD}$, there is a subsequence $(g_{n_k})_{n=1}^\infty$ such that for every compact $V\Subset U$, the restriction $g_{n_k}|_V$ is either uniformly Cauchy or converging to infinity.
\end{theorem}

\subsection{Context and Outline}

Recall that, over the complex numbers, the \emph{Fatou domain} $\mathcal{F}_{ABCD}$ is the open subset
\[\mathcal{F}_{ABCD}=\left\{p\in S_{ABCD}(\C) : \begin{array}{l}\text{there is an open subset }U\subset S_{ABCD}(\C) \\ \text{with $p\in U$, in which $\{g|_U\}_{g\in\Gamma}$ is normal}\end{array}\right\},\]
where $\Gamma=\langle s_1,s_2,s_3\rangle$. 
Here, by `normal,' we include the case when we have a subsequence `converging to infinity.' To elaborate, a sequence $(g_n|_U)_{n=1}^\infty$, where each $g_n\in\Gamma$, is said to \emph{converge to infinity} if for every compact $V\Subset U$ and $K\Subset S_{ABCD}(\C)$, we have $g_n(V)\cap K\neq\varnothing$ for finitely many $n$'s \cite[\S{1.5}]{RR21}.

Theorem~\ref{thm:p-adic-fatou-domain} is claiming that, we always have a Fatou domain, of `converging to infinity' type, for any Markov surface over $(K,\val)$ with meromorphic parameters.


The outline of the proof is as follows. We first prove the normal behavior at the tropicalized level (Proposition \ref{lem:tropical-normal-behavior}). Appealing to the openness of the intersection of tropical domains (Lemma~\ref{lem:tropical-domain-open}), we can lift this normal behavior to algebraic dynamics. (Such a technique is not applicable for the complex case.)

\subsection{Postfix Subsequence} We first make some reductions. If the (reduced) word lengths of $g_n$'s are bounded, then $(g_n)$ has a constant subsequence, falling to the uniformly Cauchy case. Otherwise, a simple diagonal argument shows the following. Here, $|{}\cdot{}|$ is the word metric of $\Gamma=\langle\sigma_1,\sigma_2,\sigma_3\mid\sigma_1^2=\sigma_2^2=\sigma_3^2=1\rangle$.
\begin{lemma}
    \label{lem:postfix-subsequence}
    If a sequence $(g_n)_{n=1}^\infty$ in $\Gamma$ lacks constant subsequences, then there is a left-infinite reduced word $w=\cdots w_3w_2w_1$ and a subsequence $(g_{n_k})_{k=1}^\infty$ such that each $g_{n_k}$ has the form $w'w_k\cdots w_2w_1$, for some reduced word $w'$ of length $|g_{n_k}|-k$.
\end{lemma}
\begin{proof}
    Choose $w_1$ so that there are infinitely many $g_n$'s whose rightmost alphabet is $w_1$; choose $g_{n_1}$ among such $g_n$'s with the least $n_1$. Choose $w_2w_1$ so that there are infinitely many $g_n$'s that has $w_2w_1$ as its right (reduced) substring; choose $g_{n_2}$ among such $g_n$'s with $n_2>n_1$ and $n_2$ minimal among such. Continue on, to define the (reduced) string $w_k\cdots w_2w_1$ and subsequence elements $g_{n_k}$.
\end{proof}

\subsection{Tropical Normal Behavior} 
We address the ``normal behavior'' on tropicalizations. This requires some involved analysis on, given a left-infinite reduced word $w=\cdots w_3w_2w_1$ of $\Gamma_{ABCD}$, the behavior of norms $\|j(w_n\cdots w_1).\vec{u}\|$ as $n\to\infty$, where $j$ is the group map~\eqref{eqn:PGL2Z-Vieta-group-identification}.

We first state the behavior of norms of interest. Define the matrices
    \begin{align*}
        R_1 = j(s_1) &= \begin{bmatrix} -1 & -2 \\ 0 & 1 \end{bmatrix}, &
        R_2 = j(s_2) &= \begin{bmatrix} 1 & 0 \\ -2 & -1 \end{bmatrix}, &
        R_3 = j(s_3) &= \begin{bmatrix} 1 & 0 \\ 0 & -1 \end{bmatrix}.
    \end{align*}
\begin{definition}
    The \emph{index} of a matrix $R_i$ is the subscript $i$. A vector $\vec{u}\in\R^2/\langle\pm1\rangle$ has \emph{index} $i$ if, by its $\ConeParam$-image $(x_1,x_2,x_3)=\ConeParam(\vec{u})$, we have $2x_i=\min\{2x_1,2x_2,2x_3\}$.
\end{definition}

\begin{remark} The zero vector has 1, 2, 3 as its indices. There may be a nonzero vector that has multiple indices, but for most vectors, it has a unique index. If $x\in\bigcup_{i=1}^3Sk(X_i^2)$, then $\ConeParam^{-1}(x)$ has index $i$ if and only if $x\in Sk(X_i^2)$. \end{remark}

Recall the 1-norm $\|(x_1,x_2,x_3)\|_1=|x_1|+|x_2|+|x_3|$ on $\R^3$ and its $\ConeParam$-pullback norm $\|(u,v)\|=|u|+|v|+|u+v|$ on $\R^2$.

\begin{proposition}
    \label{lem:tropical-normal-behavior-cone-param}
    Suppose $(w_i)_{i=1}^\infty$ is a sequence in $\{R_1,R_2,R_3\}$, with no two adjacent terms being equal, and let $\vec{u}\in\R^2/\langle\pm1\rangle$ be a vector where none of its index is equal to that of $w_1$. Then we have $\|w_n\cdots w_1.\vec{u}\|\to\infty$, as $n\to\infty$.
\end{proposition}
\begin{proof}
    Recall the plane $\Pi=\{(x_1,x_2,x_3) : x_1+x_2+x_3=0\}$. Instead of the usual cone parametrization $\ConeParam\colon\R^2/\langle\pm1\rangle\to\R^3$, we identify $(u,v)\in\R^2\mapsto(u,v,-u-v)\in\Pi$ and have $\ConeParamPi\colon\Pi/\langle\pm1\rangle\to\R^3$, $\pm(x_1,x_2,x_3)=(-|x_1|,-|x_2|,-|x_3|)$. A vector $(x_1,x_2,x_3)\in\Pi$ has index $i$ if the corresponding $(x_1,x_2)\in\R^2$ has index $i$.
    
    The plane $\Pi$ has three coordinates: $x^{(1)}=(x_2,x_3)$, $x^{(2)}=(x_3,x_1)$, and $x^{(3)}=(x_1,x_2)$ (beware the cyclic convention). We use subscripts to access its coordinate, so that $x^{(1)}_2=x_3$, $x^{(2)}_1=x_3$, etc. A vector $\vec{u}\in\Pi/\langle\pm1\rangle$ has index $i$ if and only if its $x^{(i)}$-coordinates lie on the first or third quadrant, i.e., $x^{(i)}_1x^{(i)}_2\geq 0$. With these coordinates, the actions of matrices $R_i\ACTS\R^2$ induced to $\Pi$ are described as follows (using cyclic index conventions, so that $R_0=R_3$ and $x^{(4)}=x^{(1)}$, etc.):
    \begin{align*}
        \begin{bmatrix} x^{(i-1)}_1 \\ x^{(i-1)}_2 \end{bmatrix}=R_{i-1}.\begin{bmatrix} x^{(i)}_1 \\ x^{(i)}_2 \end{bmatrix} &\Leftrightarrow \begin{bmatrix} x^{(i-1)}_1 \\ x^{(i-1)}_2 \end{bmatrix}=\begin{bmatrix} 1 & 1 \\ 1 & 0 \end{bmatrix}\begin{bmatrix} x^{(i)}_1 \\ x^{(i)}_2 \end{bmatrix}, \\
        \begin{bmatrix} x^{(i+1)}_1 \\ x^{(i+1)}_2 \end{bmatrix}=R_{i+1}.\begin{bmatrix} x^{(i)}_1 \\ x^{(i)}_2 \end{bmatrix} &\Leftrightarrow \begin{bmatrix} x^{(i+1)}_1 \\ x^{(i+1)}_2 \end{bmatrix}=\begin{bmatrix} 0 & 1 \\ 1 & 1 \end{bmatrix}\begin{bmatrix} x^{(i)}_1 \\ x^{(i)}_2 \end{bmatrix}.
    \end{align*}
    If we let
    \begin{align*}
        T_- &= \begin{bmatrix} 1 & 1 \\ 1 & 0 \end{bmatrix}, &
        T_+ &= \begin{bmatrix} 0 & 1 \\ 1 & 1 \end{bmatrix}, &
        U &= \begin{bmatrix} 1 & 1 \\ 0 & 1 \end{bmatrix}, &
        I &= \begin{bmatrix} 0 & 1 \\ 1 & 0 \end{bmatrix},
    \end{align*}
    then we describe the action of $R_{i-1}$ ($R_{i+1}$, resp.) on a vector in coordinates $x^{(i)}$ to $T_-.x^{(i)}$ ($T_+.x^{(i)}$, resp.), with coordinates in $x^{(i-1)}$ ($x^{(i+1)}$, resp.). Moreover, $T_-=UI$ and $T_+=IU$.

    This uniform, cyclic description of $R_i$-actions, together with the following notion, enables us to compute $w_n\cdots w_1.\vec{u}$ easily.
    \begin{definition}
        Let $i,j\in\{1,2,3\}$ be distinct indices. Define the \emph{circulation sign} $c(i,j)$ as one of $\pm1$ so that $i\equiv j+c(i,j)\pmod{3}$. That is,
        \[c(i,j)=\begin{cases}+1 & ((i,j)=(1,3),(2,1),(3,2)), \\ -1 &((i,j)=(1,2),(2,3),(3,1)). \end{cases}\]
        For distinct matrices $w,w'\in\{R_1,R_2,R_3\}$, we define $c(w,w')=c(i,j)$, where $i,j$ are respectively indices of $w,w'$. For a matrix $w\in\{R_1,R_2,R_3\}$ and a vector $\vec{u}$ (in $\R^2/\langle\pm1\rangle$ or $\Pi/\langle\pm1\rangle$) where none of its index is equal to that of $w$, then we define $c(w,\vec{u})=c(i,j)$, where $i$ is the index of $w$ and $j$ is the least index of $\vec{u}$.

        For the sequence $(w_i)_{i=1}^\infty$ and a vector $\vec{u}\in\R^2/\langle\pm1\rangle$ as in Proposition~\ref{lem:tropical-normal-behavior-cone-param}, its \emph{induced sequence of circulation signs} is the sequence of 0-th term $c(w_1,\vec{u})$ and $i$-th terms $(-1)^ic(w_{i+1},w_i)$, for $i\geq 1$. Its \emph{repetition data} is the sequence $(a_k)_{k=0}^\infty$ in $\Z_{\geq 0}\cup\{\infty\}$, finite terminating in $\infty$ or infinite without $\infty$, so that the induced sequence repeats $a_0$ plus signs, $a_1$ minus signs, $a_2$ plus signs, etc. (Here, $a_k=0$ only if $k=0$.)
    \end{definition}
    
    For instance, if $(w_i)$ is repeating $R_3$, $R_2$, and $R_1$ and $\vec{u}$ has index 1, then the induced sequence of circulation sign is $-1$, $+1$, $-1$, $+1$, $-1$, \ldots, with the repetition data $(0,1,1,1,\ldots)$.

    \begin{lemma}
    \label{lem:key-tropical-normal-estimate}
        \begin{enumerate}[(a)]
            \item The repetition data is finite if and only if $(w_i)$ eventually repeats only two matrices in $\{R_1,R_2,R_3\}$. In that case, we have a norm estimate $\|w_m\cdots w_1.\vec{u}\|\gtrsim m$.
            \item If the repetition data $(a_k)_{k=0}^\infty$ is infinite, $\vec{u}\in\Pi$ has the least index $i$, and $n\geq 1$, then by $m=\sum_{k=0}^{n-1}a_k$ and $j$ the index of $w_m$, the  $x^{(j)}$-coordinates of the vector $w_m\cdots w_1.\vec{u}$ is
            \begin{equation}
            \label{eqn:euc-coord-expansion}
            \begin{bmatrix} x^{(j)}_1 \\ x^{(j)}_2 \end{bmatrix}=\begin{bmatrix} 0 & 1 \\ 1 & 0 \end{bmatrix}^{m-n}\begin{bmatrix} 0 & 1 \\ 1 & a_{n-1} \end{bmatrix}\cdots\begin{bmatrix} 0 & 1 \\ 1 & a_0 \end{bmatrix}\begin{bmatrix} x^{(i)}_1(\vec{u}) \\ x^{(i)}_2(\vec{u}) \end{bmatrix}.
            \end{equation}
            In particular, we have a norm estimate $\|w_m\cdots w_1.\vec{u}\|\gtrsim(\frac12(1+\sqrt{5}))^n$.
        \end{enumerate}
    \end{lemma}
    \begin{proof}
        (a) If $(w_i)$ repeats two matrices, say $R_1$ and $R_2$, then eventually the sequence $c(w_{i+1},w_i)$ repeats $+1$ and $-1$ in period 2. With the correction signs $(-1)^i$ multiplied, this turns to be a constant sequence in $+1$ or $-1$. Hence one of repetition data must be infinite. Conversely, if a repetition data terminates at $a_n=\infty$, then for any index $m>\sum_{k=0}^{n-1}a_k$, we have $c(w_{m+2},w_{m+1})=-c(w_{m+1},w_m)$. The only possibility to have this is to repeat over two matrices.
        
        For the norm estimate, we note the following ping-pong type
        \begin{lemma}
            \label{lem:251213-2}
            If $\vec{u}\in\R^2/\langle\pm1\rangle$ do not have index $i$, then $R_i.\vec{u}$ has a unique index $i$.
        \end{lemma}
        \begin{proof}
            Consider any skeleton with a holomorphic parameter. Then $\vec{u}$ has a unique (do not have, resp.) index $i$ if and only if $\ConeParam(\vec{u})$ lies on the interior of $Sk(X_i^2)$ ($\ConeParam(\vec{u})\notin Sk(X_i^2)$, resp.).
            
            So if $\ConeParam(\vec{u})\notin Sk(X_i^2)$, then by Proposition~\ref{lem:reflection-cells-atlas}, $\trop(s_i).\ConeParam(\vec{u})=\ConeParam(R_i.\vec{u})\in Sk(X_i^2)$, and it is in the interior because otherwise, $\ConeParam(\vec{u})$ gets fixed by $\trop(s_i)$ and hence $\vec{u}$ has index $i$.
        \end{proof}
        Let $\vec{u}_m:=w_m\cdots w_1.\vec{u}$. For $m>\sum_{k=0}^{n-1}a_k$, we have $\vec{u}_m$ having a unique index, same as that of $w_m$; say it is $1$, for our convenience. Then $w_{m+1}$ has index $2$ or $3$, yet $w_{m+2}$ has index $1$, and so forth. So the $x^{(1)}$-coordinate of $\vec{u}_{m+2}$ is
        \[x^{(1)}(\vec{u}_{m+2})=T_\pm T_\mp.x^{(1)}(\vec{u}_m),\]
        where we choose $T_+T_-=\begin{bmatrix} 1 & 0 \\ 2 & 1 \end{bmatrix}$ if $w_{m+1}$ has index 3 and $T_-T_+=\begin{bmatrix} 1 & 2 \\ 0 & 1 \end{bmatrix}$ if $w_{m+1}$ has index 2. Repeating this multiplication we have
        \[x^{(1)}(\vec{u}_{m+2k}) = \begin{bmatrix} 1 & 0 \\ 2k & 1 \end{bmatrix}x^{(1)}(\vec{u}_m)\]
        if $w_{m+1}$ has index 3 and
        \[x^{(1)}(\vec{u}_{m+2k}) = \begin{bmatrix} 1 & 2k \\ 0 & 1 \end{bmatrix}x^{(1)}(\vec{u}_m)\]
        if $w_{m+1}$ has index 2. Here, we may assume that $x^{(1)}_1(\vec{u}_m)>0$ and $x^{(1)}_2(\vec{u}_m)>0$; then we have the estimate for both cases.
        
        (b) Let $\vec{u}_m:=w_m\cdots w_1.\vec{u}$. Let $i_m$ be the index of $w_m$, and $i_0$ be the least index of $\vec{u}$. Let $A(n)=\sum_{k=0}^{n-1}a_k$, with $A(0)=0$. Let $T(a)=IU^a=\begin{bmatrix} 0 & 1 \\ 1 & a \end{bmatrix}$, so that our formula~\eqref{eqn:euc-coord-expansion} can be simplified as
        \begin{equation}
            \label{eqn:euc-coord-expansion-symbolic}
            x^{(i_{A(n)})}(\vec{u}_{A(n)}) = I^{A(n)-n}T(a_{n-1})\cdots T(a_0)x^{(i_0)}(\vec{u}).
        \end{equation}
        The formula is clear if $n=0$.

        Assume the formula for $n\geq 0$. Then starting from $m=A(n)$, we have $i_{A(n)+1}\equiv i_{A(n)}+(-1)^{n+A(n)}\pmod{3}$, $i_{A(n)+2}\equiv i_{A(n)}-(-1)^{n+A(n)}\pmod{3}$, etc. until we reach $i_{A(n+1)}=i_{a_n+A(n)}$. If we cumulate the matrices of corresponding $w_m$'s, we have
        \[x^{(i_{A(n+1)})}(\vec{u}_{A(n+1)})=Tx^{(i_{A(n)})}(\vec{u}_{A(n)}),\]
        where
        \begin{align*}
            T&=\begin{cases} T_-^{a_{n}\mod 2}(T_+T_-)^{\lfloor a_{n}/2\rfloor} & (n+A(n)\text{ odd}) \\ T_+^{a_{n}\mod 2}(T_-T_+)^{\lfloor a_{n}/2\rfloor} & (n+A(n)\text{ even}) \end{cases} \\
            &= I^{A(n)+n}I^{a_{n}}U^{a_{n}}I^{A(n)+n} \\
            &=I^{A(n+1)-(n+1)}T(a_{n})I^{A(n)-n}.
        \end{align*}
        Combining with the induction hypothesis, we have
        \begin{align*}
        x^{(i_{A(n+1)})}(\vec{u}_{A(n+1)}) &= I^{A(n+1)-(n+1)}T(a_{n})I^{A(n)-n}x^{(i_{A(n)})}(\vec{u}_{A(n)}) \\
        &= I^{A(n+1)-(n+1)}T(a_{n})I^{A(n)-n}\cdot I^{A(n)-n}T(a_{n-1})\cdots T(a_0)x^{(i_0)}(\vec{u}) \\
        &= I^{A(n+1)-(n+1)}T(a_n)T(a_{n-1})\cdots T(a_0) x^{(i_0)}(\vec{u}),
        \end{align*}
        finishing the induction step.
        
        For the norm estimate, note that for the $n$-th convergent $p_n/q_n$ of the continued fraction $\alpha=[a_0;a_1,a_2,\ldots]$, we have $T(a_n)\cdots T(a_0)=\begin{bmatrix} q_{n-1} & p_{n-1} \\ q_n & p_n \end{bmatrix}$. If $n\geq 1$, then entries of this matrix are all positive and has the size bounded below by the $(n-1)$-th term of the Fibonacci sequence $f_0=f_1=1$, $f_n=f_{n-1}+f_{n-2}$. As the coordinates of $x^{(i_0)}(\vec{u})$ can be set to be nonnegative and not all zero, we verify the norm estimate.
    \end{proof}
    The Lemma~\ref{lem:key-tropical-normal-estimate} above is what we need.
\end{proof}

\begin{proposition}
    \label{lem:tropical-normal-behavior}
    Suppose $\min\{a,b,c,d\}<0$, and let $U\subset Sk(a,b,c,d)$ be the invariant subset as in Theorem~\ref{lem:one-meromorphic-model}. Fix a point $x\in U$.
    
    For any sequence $(g_n)_{n=1}^\infty$ in $\Gamma$, there is a subsequence $(g_{n_k})_{k=1}^\infty$ which is either constant or sending $x$ to infinity: the latter means that the sequence $(g_{n_k}.x)_{k=1}^\infty$ eventually exits any compact subset of the skeleton.
\end{proposition}

\begin{proof} Although we have stated this in terms of the invariant subset $U$, we may assume that our starting point is on the ping-pong table in $Sk(a,b,c,d)$. So assume that $x$ is on the ping-pong table throughout.

By Lemma \ref{lem:postfix-subsequence}, if the sequence lacks constant subsequence, we can fix a left-infinite reduced word $w=\cdots w_3w_2w_1$ and view that $g_n=w^{(n)}w_n\cdots w_1$ for $n\geq 1$, by some word $w^{(n)}$ of length $|g_n|-n$. We aim to show that $g_n.x\to\infty$.

\begin{lemma}[Two-step then away]
\label{lem:two-step-then-away}
    Suppose $\min\{a,b,c,d\}<0$. Let $s_j$ and $s_k$ be distinct elements from $\{s_1,s_2,s_3\}$. Let $x$ be a point on the ping-pong table established in Lemma~\ref{lem:skeletal-pp-structure}. Then we have $s_js_k.x\in Sk(X_j^2)$.
\end{lemma}
\begin{proof}
    Let $D_0$ be the ping-pong table. Let $A_1,A_2,A_3$ respectively denote $A,B,C$. Then by \eqref{eqn:ping-pong-general-0}, we know that $s_js_k.x\in\overline{D}_j\subset Sk(X_j^2)\cup Sk(A_jX_j)$. Now if $s_js_k.x\in Sk(A_jX_j)$, then $s_k.x$ lies on the ping-pong table. Thus $x$ must be fixed by $s_k$, but then it should be on the interior of the adjacency locus $Sk(A_kX_k)$. This implies $s_j=s_k$, a contradiction.
\end{proof}

For each letter $w_0\in\{s_1,s_2,s_3\}$, let $Sk(w_0)$ be $Sk(X_i^2)$ if $w_0=s_i$. Lemma~\ref{lem:two-step-then-away} implies that $w_2w_1.x\in Sk(w_2)$, and as our starting point is on the ping-pong table, we are in the interior of $Sk(w_2)$. Iterating Proposition~\ref{lem:reflection-cells-atlas}, we have $w_n\cdots w_2w_1.x\in Sk(w_n)$ for all $n\geq 2$; they always lie on the interior of $Sk(w_n)$. Then Lemma~\ref{lem:contraction-norm} applies to yield $\|w_n\cdots w_2w_1.x\|_1>\|w_{n-1}\cdots w_1.x\|_1$.

We aim to understand how the norms $\|w_n\cdots w_1.x\|_1$ grows in $n$. As $w_n\cdots w_1.x\in Sk(w_n)$ for all $n\geq 2$, the formulae of $\trop(w_{n+1})$'s in Lemma~\ref{lem:quadratic-cells-basics}(b) apply in each step. Hence, by Lemma~\ref{lem:cone-parametrization-is-equivariant}, with $j$ the group map~\eqref{eqn:PGL2Z-Vieta-group-identification}, we have, for $n\geq 2$,
\[w_n\cdots w_3w_2w_1.x=\ConeParam\left(j(w_n\cdots w_3).\ConeParam^{-1}(w_2w_1.x)\right).\]
So we are reduced to the behavior of norms $\|j(w_n\cdots w_3).\vec{u}\|$ as $n\to\infty$, with $w_2w_1.x=\ConeParam(\vec{u})\notin Sk(w_3)$. This goes to infinity, by Proposition~\ref{lem:tropical-normal-behavior-cone-param}.
\end{proof}

\subsection{Proof of Theorem \ref{thm:p-adic-fatou-domain}}

Recall that $T=\bigcap_{\gamma\in\Gamma_{ABCD}}\bigcap_{i=1}^3\gamma.\mathrm{dom}(s_i)$ the intersection of tropical domains, is open in $Sk(a,b,c,d)$ if $\min\{a,b,c,d\}<0$, by Lemma~\ref{lem:tropical-domain-open}. Since $K$ is algebraically closed, its value group $\val(K^\times)$ is dense, and hence there exists $x\in T$ whose coordinates are in the value group.

The set of points $P\in S_{ABCD}(K)$ with $\val(P)=x$ forms a \emph{nonempty} open subset, by the non-archimedean property (the existence of such $P$ is due to Kapranov's theorem, \cite[Theorem 2.1.1]{EKL06}\cite[Theorem 3.1.3]{MS15}). Moreover, as $x\in T$, whenever $\gamma\in\Gamma_{ABCD}$ and $P\in S_{ABCD}(K)$ is with $\val(P)=x$, we have
\[\val(\gamma.P)=\trop(\gamma)(\val(P)).\]
The rest follows from Proposition \ref{lem:tropical-normal-behavior}.

\section{Application: Rational Points on the Compact Component}
\label{sec:rational-points-compact-component}

Another application of our theory of tropicalized dynamics is the following result on rational points of the cubic Markov surface.

Let $S_D$ and $\Gamma_D$ be shorthands of $S_{000D}$ and $\Gamma_{000D}$, respectively. If $D$ is a real number, $0<D<4$, then $S_D(\R)$ has a unique compact component~\cite[\S{3.3}]{Goldman03}, denoted by $S_D(\R)^{cpt}$; this is $\Gamma_D$-invariant. If $A\subset\R$ is a subring and $D\in A$, then we may write $S_D(A)^{cpt}=S_D(\R)^{cpt}\cap S_D(A)$, which is again $\Gamma_D$-invariant.
\begin{theorem}
    \label{thm:rational-points-compact-component}
    Suppose $p$ is an odd prime and $D\in\Z[\frac1p]$. If $0<D<1$, then there are finitely many points $P_1,\ldots,P_N\in S_{D}(\Z[\frac1p])^{cpt}$ such that, any point in $S_D(\Z[\frac1p])^{cpt}$ is either
    \begin{enumerate}[(i)]
        \item on the $\Gamma_{D}$-orbit of some $P_i$, or
        \item on the $\Gamma_{D}$-orbit of some point in the set
        \[H=\left\{(X_1,X_2,X_3)\in S_D(\Z[\tfrac1p]) : \begin{array}{l} X_1X_2X_3=0, \\
            (\exists i)(2\val_p(X_i)<\val_p(D))\end{array}\right\}.\]
    \end{enumerate}
    Moreover, if $p\equiv 1\pmod{4}$ and there is $(x_0,y_0)\in\Z[\frac1p]$ with $x_0^2+y_0^2=D$, then $H$ is infinite; otherwise, $H$ is empty.
\end{theorem}

\subsection{The real place and the \texorpdfstring{$p$}{p}-adic place}

Given a nonzero number $x\in\Z[\frac1p]$, denote its absolute value on $\R$ by $|x|_\infty$, as of ``infinite place.'' We want to compare this with its $p$-adic value $|x|_p=p^{-\val_p(x)}$. Since $x=mp^{\val_p(x)}$ for some $m\in\Z$ relatively prime to $p$, we have
    \[|x|_\infty\cdot|x|_p = |m|p^{\val_p(x)}\cdot p^{-\val_p(x)}=|m|\geq 1.\]
    Taking $\log_p$, we thus have
    \begin{equation}
        \label{eqn:p-adic-value-inequality}
        \log_p|x|_\infty \geq \val_p(x).
    \end{equation}

A straight consequence is the following
\begin{lemma}
    \label{lem:251218-1}
    Let $p$ be a prime. If $x\in\Z[\frac1p]$ is nonzero and $-1<x<1$, then $\val_p(x)<0$.
\end{lemma}

\begin{lemma}
    \label{lem:251218-2}
    Suppose $0<D<1$. The compact part $S_D(\R)^{cpt}\subset\R^3$ lies on $(-1,1)^3$.
\end{lemma}
\begin{proof}
    By~\cite[\S{3.3}]{Goldman03}, coordinates of points in $S_D(\R)^{cpt}$ are traces of $2\times 2$ unitary matrices, which lies on the interval $[-2,2]$. So $S_D(\R)^{cpt}\subset[-2,2]^3$ follows.

    It suffices to show that $S_D(\R)\cap[-2,2]^3\subset(-1,1)^3$ if $0<D<1$. Pick any point $(X_1,X_2,X_3)\in S_D(\R)\cap[-2,2]^3$. As $|X_1|\leq 2$, there is an angle $\theta\in[0,\pi]$ such that $X_1=2\cos\theta$. By the equation of $S_D$, we have
    \begin{align*}
        D-X_1^2 &= X_2^2+X_3^2+X_1X_2X_3 \\
        &= X_2^2+X_3^2+2X_2X_3\cos\theta \\
        &= (X_2+X_3\cos\theta)^2+(X_3\sin\theta)^2\geq 0.
    \end{align*}
    It follows that $|X_1|\leq\sqrt{D}<1$. Same for other coordinates.
\end{proof}

Combining two, we see that any point $(X_1,X_2,X_3)\in S_D(\Z[\frac1p])^{cpt}$ is either
\begin{itemize}
    \item a point with at least one of coordinates being zero, or
    \item a point whose coordinates have negative $p$-adic values.
\end{itemize}
Name the former a \emph{halo point} and the latter a \emph{skeletal point}.

\subsection{Halo points}

Up to permuting coordinates, we observe that halo points take the form $(0,X_2,X_3)\in S_D(\Z[\frac1p])^{cpt}$. Hence we reduce to $\Z[\frac1p]$-points of $X_2^2+X_3^2=D$. Let
\[H_1 = \left\{(X_2,X_3)\in\Z[\tfrac1p]^2 : X_2^2+X_3^2=D\right\}.\]
Let $d=\val_p(D)$. From $0<D<1$, we have $d<0$.

\begin{lemma}
    \label{lem:halo-finitude-3mod4}
    If $p\equiv 3\pmod{4}$, then $H_1$ is finite, and all points $(X_2,X_3)$ within has values $\geq\frac12d$.
\end{lemma}
\begin{proof}
    Let $-\nu=\min\{\val_p(X_2),\val_p(X_3)\}$. This means that $p^\nu X_2$ and $p^\nu X_3$ are both integers, with one of them being coprime to $p$. It follows that
    \[\val_p(p^{2\nu}D)=\val_p((p^\nu X_2)^2+(p^\nu X_3)^2)\geq 0.\]
    Hence $p^{2\nu}D$ is an integer as well.

    If $p^{2\nu}D\equiv 0\pmod{p}$, then we have $(p^\nu X_2)^2+(p^\nu X_3)^2\equiv 0\pmod{p}$. Say if $p^\nu X_3$ is coprime to $p$, then we have $(p^\nu X_2)/(p^\nu X_3)$ (modulo $p$) a square root of $-1$ modulo $p$. Contradiction to $p\equiv 3\pmod{4}$.

    So $\nu=-\frac12d$ must hold (which implies the 2nd claim). As there are finitely many integer points $(X,Y)$ on the circle $X^2+Y^2=p^{2\nu}D$, we have our finitude.
\end{proof}

In fact, the trick of setting $-\nu=\min\{\val_p(X_2),\val_p(X_3)\}$ and reducing to an integer equation $X^2+Y^2=p^{2\nu}D$ (with $X=p^\nu X_2$ and $Y=p^\nu X_3$) is applicable in the case $p\equiv 1\pmod{4}$ as well. However, there are various choices of $\nu$ in this case.

\begin{lemma}
    \label{lem:halo-finitude-1mod4}
    Suppose $p\equiv 1\pmod{4}$. Then $H_1$ is either empty or infinite.
\end{lemma}
\begin{proof}
    The trick is to construct infinitely many solutions from $(X_2,X_3)\in H_1$. Let $-\nu=\min\{\val_p(X_2),\val_p(X_3)\}$, so that $\val_p(p^{2\nu}D)\geq 0$. Let $x_0=p^\nu X_2$ and $y_0=p^\nu X_3$. Fix $a,b\in\Z_{>0}$ with $a^2+b^2=p$, $p\nmid a,b$, and $ax_0\not\equiv by_0\pmod{p}$ (exchange $a$ and $b$ for the last one, if necessary).

    For $n\geq 0$, let
    \[\begin{bmatrix} x_n \\ y_n \end{bmatrix}=M^n\begin{bmatrix} x_0 \\ y_0 \end{bmatrix},\quad\text{ where }\quad M=\begin{bmatrix} a & -b \\ b & a \end{bmatrix}.\]
    That is, $x_n+\sqrt{-1}y_n=(x_0+\sqrt{-1}y_0)(a+\sqrt{-1}b)^n$. Then we have $x_n^2+y_n^2=p^nD$, whence $(p^{-n}x_{2n},p^{-n}y_{2n})\in H_1$ for $n=0,1,2,\cdots$.
    
    As $\det M=a^2+b^2\equiv 0\pmod{p}$ yet the trace of $M$ is $2a\not\equiv 0\pmod{p}$, $M$ is diagonalizable over $\Z/p\Z$, with eigenvalues $0$ and $2a$. Hence we have $M^n\equiv(2a)^{n-1}M\pmod{p}$ for $n\geq 1$. Thus $x_n\equiv(2a)^{n-1}x_1\pmod{p}$ and $y_n\equiv(2a)^{n-1}y_1\pmod{p}$ follows. Meanwhile, as $ax_0\not\equiv by_0\pmod{p}$ and $b^2\equiv-a^2\pmod{p}$, we see that
    \begin{align*}
        x_1 &= ax_0-by_0\not\equiv 0\pmod{p}, \\
        y_1 &= bx_0+ay_0 
        \equiv -a^2b^{-1}x_0+ay_0 \pmod{p} \\
        &= -ab^{-1}(ax_0-by_0)\not\equiv 0\pmod{p}.
    \end{align*}
    Hence all $x_n$ and $y_n$ are coprime to $p$, for $n=1,2,\cdots$. Thus the points $(p^{-n}x_{2n},p^{-n}y_{2n})$ are points with distinct minimum $p$-adic values.
\end{proof}

\subsection{Skeletal points}

Fix $0<D<1$, $D\in\Z[\frac1p]$, and let $d=\val_p(D)<0$ and $m=p^{-d}D$.

We now view $\Z[\frac1p]$ as a subring of $(\Q_p,\val_p)$, the field of $p$-adic numbers. The $p$-adic complex number field $(\C_p,\val_p)$ is an algebraically closed non-archimedean field containing $\Q_p$. Hence we study the tropicalization of $S_D$ over $(\C_p,\val_p)$.

Pick any skeletal point $X=(X_1,X_2,X_3)\in S_D(\Z[\frac1p])^{cpt}$, and let $x_i=\val_p(X_i)$, $i=1,2,3$. As $0<|X_i|<1$, we have $x_i<0$, for $i=1,2,3$. (On contrary, if there is $i$ with $x_i\geq 0$, then $X$ is necessarily a halo point.) Although $x=(x_1,x_2,x_3)$ is in the tropicalization $\Trop(S_D)$, it does not guarantee that $x$ is on the skeleton. Nonetheless, we do find $x$ in the skeleton if its 1-norm $-(x_1+x_2+x_3)$ is big enough.

\begin{lemma}
    If $(X_1,X_2,X_3)\in S_D(\C_p)$ is such that the values $x_i:=\val_p(X_i)$ are nonpositive and verify $x_1+x_2+x_3\leq d$, then $(x_1,x_2,x_3)$ is in the skeleton of $\Trop(S_D)$, as well as the image of the cone paramterization $\ConeParam$~\eqref{eqn:cone-parametrization}.
\end{lemma}
\begin{proof}
    It suffices to see that
    \[x_1+x_2+x_3\leq\min\{2x_1,2x_2,2x_3\}\]
    holds, since then $x_1+x_2+x_3\leq\min\{2x_1,2x_2,2x_3,d\}$ will follow.
    
    Assume that $x_1=\min\{x_1,x_2,x_3\}$. Suppose $x_1+x_2+x_3>2x_1$. Then we have
    \begin{align*}
        \val_p(X_1^2+X_1X_2X_3) &= 2x_1, \\
        \val_p(D-X_2^2-X_3^2) &\geq\min\{d,2x_2,2x_3\}.
    \end{align*}
    By the equation of $S_D$, it follows that $2x_1\geq\min\{d,2x_2,2x_3\}$. As $d\geq x_1+x_2+x_3>2x_1$, it follows that $2x_1\geq\min\{2x_2,2x_3\}$ yet $x_1\leq\min\{x_2,x_3\}$; so $x_1=\min\{x_2,x_3\}$. Suppose $x_1=x_2$. Then from $x_1+x_2+x_3>2x_1$, we get $x_3>0$. Contradiction.
\end{proof}

From what we know about the tropical dynamics on the skeleton, we can now present a

\begin{proof}[Proof of Theorem~\ref{thm:rational-points-compact-component}]
    We first claim that there are finitely many points $X\in S_D(\Z[\frac1p])^{cpt}$ with $x=(x_1,x_2,x_3)=\val_p(X)$ such that $x_1+x_2+x_3\geq d=\val_p(D)$. Such points will verify $x_i\geq d$ for $i=1,2,3$, and hence each $X_i$ take the form $X_i=p^dm_i$, by some $m_i\in\Z$. Since $S_D(\Z[\frac1p])^{cpt}\subset[-2,2]^3$, we see that such $m_i$'s range in $[-2p^{-d},2p^{-d}]$. So there are at most $(2p^{-d}+1)^3$ points $P_1,\ldots,P_N\in S_D(\Z[\frac1p])^{cpt}$ with the property of interest.

    Pick any $X\in S_D(\Z[\frac1p])^{cpt}$ with $x=\val_p(X)$ such that $x_1+x_2+x_3<d$. From Proposition~\ref{lem:quadratic-reduction-result}, there is $\gamma_0\in\Gamma_D$ such that $\trop(\gamma_0).x$ is either on a boundary ray or the adjacency locus $Sk(D)$. We consider what this implies for $\gamma_0.X$.

    Extended from \eqref{eqn:251218-1}, we have, for any $\gamma\in\Gamma_D$,
    \begin{equation}
        \label{eqn:251218-2}
        \trop(\gamma).x \leq \val_p(\gamma.X),
    \end{equation}
    where $\leq$ above means each coordinate of the left hand side is no more than the corresponding coordinate of the right hand side. 
    
    If $\trop(\gamma_0).x$ is on a boundary ray, then $\gamma_0.X$ has at least one of its coordinates having nonnegatively valued, thus it must be a halo point. If $\trop(\gamma_0).x$ is on the adjacency locus $Sk(D)$, then unless $\gamma_0.X$ is a halo point, we have $\val_p(\gamma_0.X)\leq(0,0,0)$, thus the 1-norm of $\val_p(\gamma_0.X)$ is at most that of $\trop(\gamma_0).x$, which equals to $-d$ (note that $Sk(D)\subset\{-(x_1+x_2+x_3)=-d\}$). Hence $\gamma_0.X$ falls into one of $P_i$'s listed above.

    To address the properties of the halo set $H$, we have posed $\val_p(X_i)<\frac12\val_p(D)$ so that, by Lemma~\ref{lem:halo-finitude-3mod4}, $H$ becomes empty if $p\equiv 3\pmod{4}$. Together with the dichotomy addressed in Lemma~\ref{lem:halo-finitude-1mod4}, we have our claim on $H$.
\end{proof}

\begin{remark}
    We note that Theorem \ref{thm:rational-points-compact-component} has some tangencies with a lemma \cite[Lemma 5.5]{GMS21} on $\Z[\frac1p]$-points of $S_{000D}$ by Ghosh, Meiri, and Sarnak. The lemma focuses on the subset $\{(x_1,x_2,x_3)\in\Trop(S_{000D}) : x_2=x_3,\ x_1\geq 0\}\cup\{x\in\Trop(S_{000D}) : x_i\geq 0\}$ (this appears as some `fins' budded on the skeleton: see Figure \ref{fig:second-tropicalization}), which almost complements the skeleton.
\end{remark}

\subsection{The case of \texorpdfstring{$1\leq D<4$}{1<=D<4}}

As the compact part $S_D(\Z[\frac1p])^{cpt}$ is well-defined for $0<D<4$, one may wonder what happens for the case untouched, $1\leq D<4$. In that case, the estimate in Lemma~\ref{lem:251218-2} becomes $S_D(\R)^{cpt}\subset(-\sqrt{D},\sqrt{D})^3$ and hence there may be some additional points with an integer coordinate (say, $(X_1,X_2,X_3)\in S_D(\Z[\frac1p])^{cpt}$ with $X_1\in\Z$). Nonetheless, the possible integer coordinates are $0$ and $\pm 1$, so we can still discuss the finitude of halo points (those points with at least one of coordinates being an integer).

After some easy yet tedious task, one can see that the essence of Theorem~\ref{thm:rational-points-compact-component} remain the same in the case $1\leq D<4$: all orbits come from a halo point or a finite number of skeletal points. The finitude of halo points will depend on what $p$ is modulo $12$, which will influence the existence of integer solutions $(a,b)$ to the equations $a^2+b^2=p$ or $a^2+b^2+ab=p$.

Our restriction $0<D<1$ here is therefore to make less case-by-case splittings and have a shorter presentation.

\bibliography{references}{}
\bibliographystyle{amsalpha}

\end{document}